\documentclass[10pt,leqno]{article}
\usepackage[toc,page]{appendix}

\usepackage{amsmath,amssymb,amsfonts,graphicx,bm,amsthm,mathrsfs}
\usepackage[usenames]{xcolor}
\usepackage{soul}
\usepackage[colorlinks=true]{hyperref}
\hypersetup{citecolor={blue},urlcolor={red}}
\usepackage{cleveref}
\usepackage{appendix}
\usepackage{pstricks}

\usepackage{enumitem}
\usepackage{latexsym}
\usepackage{stmaryrd}
\usepackage{bussproofs}
\usepackage{pstricks}
\usepackage{aliascnt}

\numberwithin{equation}{section} 
\newtheorem{theorem}{Theorem}[section]
\newaliascnt{lemma}{theorem}  
\newtheorem{lemma}[lemma]{Lemma}  
\aliascntresetthe{lemma}  
  
\newaliascnt{corollary}{theorem}  
\newtheorem{corollary}[corollary]{Corollary}  
\aliascntresetthe{corollary}  
 
\newaliascnt{proposition}{theorem}  
  
\aliascntresetthe{proposition}  
 
\newtheoremstyle{exampstyle}
{\topsep} 
{\topsep} 
{} 
{} 
{\bfseries} 
{.} 
{.5em} 
{} 
\theoremstyle{exampstyle}
\newaliascnt{fact}{theorem}  
  
\aliascntresetthe{fact}  
 
\newaliascnt{claim}{theorem}  
  
\aliascntresetthe{claim}  
 
\newaliascnt{remark}{theorem}  
\newtheorem{remark}[remark]{Remark}  
\aliascntresetthe{remark}  
 
\newaliascnt{notation}{theorem}  
\newtheorem{notation}[notation]{Notation}  
\aliascntresetthe{notation}  
 
\newaliascnt{example}{theorem}  
  
\aliascntresetthe{example}  
 
\newaliascnt{conjecture}{theorem}  
\newtheorem{conjecture}[conjecture]{Conjecture}  
\aliascntresetthe{conjecture}  
 
\newaliascnt{question}{theorem}  
  
\aliascntresetthe{question}  
 
\newaliascnt{definition}{theorem}  
\newtheorem{definition}[definition]{Definition}  
\aliascntresetthe{definition}  
 
\newaliascnt{convention}{theorem}  
  
\aliascntresetthe{definition}

\newcommand{\Boxdot}{\,
                    \setlength{\unitlength}{1ex}
                    \begin{picture}(2,2)
                    \put(0,0){$\Box$}
                    \put(.55,.65){.}
                    \end{picture}
                    }

\def\kcal{\mathcal{K}}

\def\IQC{\hbox{\sf IQC}{} }
\def\CQC{\hbox{\sf CQC}{} }
\def\IPC{\hbox{\sf IPC} }

\def\HA{{\sf HA}}

\def\PA{{\sf PA}}

\def\CP{\hbox{\sf CP}{} }
\def\TP{\hbox{\sf TP}{} }
\def\NNIL{\hbox{\sf NNIL}{} }
\def\TNNIL{\hbox{\sf TNNIL}{} }

\def\K4{\hbox{\sf K4}{} }
\def\GL{\hbox{\sf GL}{} }

\def\iGL{\hbox{\sf iGL}{} }

\def\R{\sqsubset}

\def\PL{\mathcal{PL}}
\def\lcal{\mathcal{L}}
\def\PLS{\mathcal{PL}_{_{\Sigma_1}}}

\newcommand{\Ax}{\AxiomC}
\newcommand{\UI}{\UnaryInfC}
\newcommand{\BI}{\BinaryInfC}

\newcommand{\DP}{\DisplayProof}

\newcommand{\ra }{\rightarrow }

\newcommand{\lr }{\leftrightarrow}

\newcommand{\bo }{\Box }

\newcommand{\Prf}[3]{{\sf Proof}_\tinysub{\sf #1}(#2,\gnumber{#3})}
\newcommand{\Prv}[2]{{\sf Prov}_\tinysub{\sf #1}(\gnumber{#2})}

\newcommand{\gnumber}[1]{\ulcorner#1\urcorner}
\newcommand{\uparan}[1]{\textup{(}#1\textup{)}}
\newcommand{\tinysub}[1]{{^{_{#1}}}}

\newcommand{\nat}{\mathbb{N}}
\renewcommand{\figurename}{Diagram}
\numberwithin{equation}{section}

\def\PLG{\mathcal{PL}_{_\Gamma}}
\def\PLGP{\mathcal{PL}_{_{\Gamma'}}}

\newcommand{\wit}[1]{[\![#1]\!]}
\newcommand{\Kripke}{\mathcal{K}=(K,\preccurlyeq,\sqsubset,V)}
\newcommand{\GLS}{{\sf GL\underline{S}}}
\newcommand{\SFT}{{\sf T}}
\newcommand{\SFU}{{\sf U}}
\newcommand{\AC}{\mathcal{AC}_{_\Gamma}(\SFV;\SFT,\SFU)}
\newcommand{\ACP}{\mathcal{AC}_{_{\Gamma'}}(\SFV';\SFT',\SFU')}
\newcommand{\ACPP}{\mathcal{AC}_{_{\Gamma''}}(\SFV'';\SFT'',\SFU'')}
\newcommand{\fbar}[1][A]{\bar{f}_{_{#1}}}
\newcommand{\lcalb}{\lcal_\Box}
\newcommand{\negout}{{\neg\!\,\uparrow}}

\newcommand{\negin}{{\neg\!\,\downarrow}}
\newcommand{\ac}[3]{\mathcal{AC}(#1;#2,#3)}
\newcommand{\acs}[3]{\mathcal{AC}_{_{\Sigma_1}}(#1;#2,#3)}
\newcommand{\AS}{\mathcal{AS}_{_\Gamma}(\SFV;\SFT,\SFU)}

\newcommand{\SFV}{{\sf V}_{_{\! 0}}}

\newcommand{\PEM}{{\sf PEM} }
\newcommand{\igl}{{\sf iGL} }
\newcommand{\iglsigma}{{\sf iGLC_a}}
\newcommand{\iglphat}{{\sf iGL}\overline{\sf P} }
\newcommand{\iglphatsigma}{{\sf iGL}\overline{\sf P}{\sf C_a}}
\newcommand{\iglc}{{\sf iGLC} }
\newcommand{\iglct}{{\sf iGLCT} }
\newcommand{\ihatglchat}{{\sf i}{\sf GL}\overline{\sf C}\underline{\sf P}{\sf C_a}}

\newcommand{\ihatglchatt}{{\sf i}{\sf GL}\overline{\sf C}{\sf T}\underline{\sf P}}

\newcommand{\iglchatthat}{{\sf iGL}\overline{\sf C}\overline{\sf T}}
\newcommand{\iglchatthatsigma}{{\sf iGL}\overline{\sf CT}{\sf C_a}}
\newcommand{\ihatglchattsigma}{{\sf iGL}\overline{\sf C}{\sf T}\underline{\sf P}{\sf C_a}}
\newcommand{\iglchattsigma}{{\sf iGL}\overline{\sf C}{\sf TC_a}}
\newcommand{\ihatglchatts}{{\sf iGL}\overline{\sf C}{\sf T}\underline{\sf S^*P}}
\newcommand{\ihatglchattssigma}{{\sf iGL}\overline{\sf C}{\sf T}\underline{\sf S^*P}{\sf C_a}}
\newcommand{\ihatglchats}{{\sf iGL}\overline{\sf C}\underline{\sf SP}{\sf C_a}}

\newcommand{\ikfour}{{\sf iK4}}
\newcommand{\GLV}{{\sf GLC_a}}
\newcommand{\GLSV}{{\sf GL\underline{S}C_a}}
\newcommand{\llessstar}{\sf iH_\sigma^{*\!*}}
\newcommand{\llesstar}{{\sf iH_\sigma^*}}
\newcommand{\ihatHsigmastar}{{\sf iH_\sigma}\underline{\sf P}^*}
\def\lles{{\sf{iH_\sigma}}}
\newcommand{\ihatHsigma}{{\sf iH_\sigma}\underline{\sf P}}

\renewcommand{\hl}[1]{#1}
\newcommand{\hlf}[1]{\text{\hl{$#1$}}} 
\newcommand{\ihatHSsigma}{{\sf iH_\sigma}\underline{\sf SP}}
\newcommand{\ihatHSsigmastar}{{\sf iH_\sigma}\underline{\sf SP}^*}
\newcommand{\ihatglleplus}{{\sf iGLLe}{}^+\underline{\sf P}}
\def\iglleplus{\hbox{\sf{iGLLe}}{}^+}
\newcommand{\ihatgllepluss}{{\sf iGLLe}{}^+\underline{\sf SP}}
\newcommand{\Boxin}{{{\Box\!\downarrow}}}
\newcommand{\Boxout}{{\Box\!\uparrow}}
\newcommand{\sub}[1]{{\sf Sub}(#1)} 
\usepackage{pst-node}
\usepackage{rotating}
\usepackage{tikz-cd}

\newcommand{\sigmapa}[1]{\sigma_{_{\sf PA}}(#1)}
\newcommand{\sigmapas}[1]{\sigma_{_{\sf PA^*}}(#1)}
\newcommand{\sfk}{{\sf K}}

\usepackage{empheq}
\usepackage[most]{tcolorbox}

\tcbset{colback=yellow!10!white, colframe=red!50!black, boxsep=0pt,
		boxrule=1pt,
        highlight math style= {enhanced, 
            colframe=red!50!black,colback=blue!10!white,boxsep=0pt,boxrule=2pt}
        }

\begin{document}
\setstcolor{red}
\title{Hard Provability Logics}

\author{	Mojtaba Mojtahedi\thanks{\url{http://mmojtahedi.ir/}}\\
				Department of Mathematics, 
				Statistics and Computer Science, \\ 
		College of Sciences,  University of Tehran
}

\maketitle
{%
\centering\footnotesize Dedicated to Mohammad Ardeshir, for 17 years of his impact,  inspirations and motivations.
\par%
}
\vspace*{8mm}
\begin{abstract}
Let $\PL({\sf T},{\sf T}')$ and $\PLS({\sf T},\SFT')$   respectively
indicates the  provability logic
  and $\Sigma_1$-provability logic  of $\SFT$ relative 
  in $\SFT'$. 
  In this paper we characterize the following relative provability logics: $\PLS(\HA,\mathbb{N})$, $\PLS(\HA,\PA)$,
 $\PLS(\HA^*,\mathbb{N})$, $\PLS(\HA^*,\PA)$, 
 $\PL(\PA,\HA)$, $\PLS(\PA,\HA)$,    
 $\PL(\PA^*,\HA)$, $\PLS(\PA^*,\HA)$, 
 $\PL(\PA^*,\PA)$, $\PLS(\PA^*,\PA)$, 
 $\PL(\PA^*,\mathbb{N})$, $\PLS(\PA^*,\mathbb{N})$ (see Table \ref{Table-Theories}). 
 It turns out that all of these provability logics are decidable.

 The notion of {\em reduction} for provability logics, 
first informally considered in \cite{reduction}.
In this paper, we formalize a generalization of  this notion
(\Cref{Definition-Reduction-PL})
and provide several reductions of provability logics (See diagram \ref{Diagram-full}).
 The interesting fact is that   
 $\PLS(\HA,\mathbb{N})$ is the hardest 
  provability logic: the arithmetical completenesses of all provability logics 
 listed above, as well as well-known provability logics like 
 $\PL(\PA,\PA)$, $\PL(\PA,\mathbb{N})$, 
 $\PLS(\PA,\PA)$, $\PLS(\PA,\mathbb{N})$ and 
  $\PLS(\HA,\HA)$
  are all propositionally reducible 
 to the arithmetical completeness of   
 $\PLS(\HA,\mathbb{N})$. 
\end{abstract}

\tableofcontents
\listoffigures

\section{Dedication}

My works in general and the present paper in particular are highly inspired by Mohammad Ardeshir’s outstanding contributions to mathematical logic. The ideas that I developed in this paper originate from a joint paper \cite{reduction} which was initially motivated by him.
My first recollection of Mohammad Ardeshir goes back to 2003 when, in the second semester of my undergraduate studies, I attended his course on the foundations of mathematics. I was impressed by his knowledge and by the style of his teaching which encouraged me to attend most of his other courses during my undergraduate and graduate studies. I still vividly remember how deeply I was fascinated by his graduate course on Gödel’s incompleteness theorems in Fall 2006. It was this course which made me determinate to do my PhD on mathematical logic and under the supervision of Mohammad Ardeshir. His influence on me is not restricted to my academic work. He has been a source of inspiration on many aspects of my life; and that is why dedicating this paper to him is the least thing I can do to thank him.

\section{Introduction}
There are two excellent surveys on provability logic: 
\cite{VisBek,ArtBekProv}.  
To be self-contained, we bring  some selected subjects  from them here,
 and then review some related 
 recent results on this subject.

The provability interpretation for the modal operator $\Box$, 
first considered by Kurt G\"odel \cite{Godel33}, intending to 
provide a semantic for Heyting's formalization of the 
intuitionistic logic, $\IPC$. On the other hand, and again 
by innovative and celebrated   G\"odel's incompleteness results
 \cite{Godel}, for a recursively enumerable theory $\SFT$ 
 and a sentence in the language of $\SFT$, one may 
formalize ``$A$ is provable in $\SFT$"
 via a simple  ($\Sigma_1$) formula 
 $\Prv{T}{A}$
in the first-order language of arithmetic,
 in which $\gnumber{A}$ is the G\"odel number of $A$.
Let $\PL(\SFT,\SFT')$ and $\PLS(\SFT,\SFT')$   respectively
indicates the  provability logic
  and $\Sigma_1$-provability logic  of $\SFT$ relative 
  in $\SFT'$ (\Cref{Definition-Provability Logic}).  Here is a 
  list of results on provability logics with arithmetical flavour:
  \begin{enumerate}
  \item $\neg\Box\bot\not\in \PL(\PA,\PA)$, \cite{Godel}
  \item $\Box(\Box A\to A)\to \Box A\in \PL(\PA,\PA)$,
  \cite{Lob}
\item $A\in\PL(\HA,\HA)$ for a nonmodal proposition $A$,  iff 
 	$A$ is valid in the intuitionistic logic $\IPC$. \cite{dejongh,vi1}
  \item $\GL=\PL(\PA,\PA)$ and $\GLS=\PL(\PA,\mathbb{N})=
  \PL(\PA,{\sf ZF})$,   \cite{Solovay},
   in which $\GL$ is the G\"odel-L\"ob logic, as 
  defined in \Cref{Def-Axiom schema and modal theories}.
\item  $ \Box (A\vee B)\to(\Box A\vee \Box B)\not\in \PL(\HA,\HA)$,  \cite{Myhill,Friedman75}
\item $\Box(A\vee B)\to\Box(\Boxdot A\vee\Boxdot B)\in \PL(\HA,\HA)$, in which $\Boxdot A$ is a shorthand for
$A\wedge\Box A$, \cite{Leivant-Thesis} 
\item   $\iglct=\PL(\PA^*,\PA^*)$, \cite{VisserThes,Visser82}, 
in which $\iglct$ is as defined in \Cref{Def-Axiom schema and modal theories},
\item  $ \Box\neg\neg\,\Box A\to\Box\Box A\in\PL(\HA,\HA)$ and 
$ \Box(\neg\neg\,\Box A\to\Box A)\to\Box(\Box A\vee \neg\,\Box A)\in\PL(\HA,\HA)$, \cite{VisserThes,Visser82}
\item \hl{Rosalie Iemhoff 2001 introduced} a uniform axiomatization  of 
all known axiom schemas of  $\PL(\HA,\HA)$ in an extended language 
with a bimodal operator $\rhd$. In her Ph.D. dissertation 
\cite{IemhoffT}, Iemhoff raised a conjecture that implies directly that her 
axiom system, ${\sf iPH}$,  restricted to the normal modal language, is 
equal to $\PL(\HA,\HA)$, \cite{IemhoffT} 
\item  $\PL_{\{\top,\bot\}}(\HA,\HA)$ is decidable. \cite{Visser02}.  In other words, 
he introduced a decision algorithm for 
$A\in \PL(\HA,\HA)$, for all $A$ not containing any atomic variable.
\item $\PLS(\HA,\HA)={\sf iH}_\sigma$ (\Cref{Def-lles})
is decidable,  \cite{Sigma.Prov.HA,Jetze-Visser}
\item $\PLS(\HA^*,\HA^*)={\sf iH}^*_\sigma$ 
(\Cref{Def-lles}) is decidable,  \cite{Sigma.Prov.HA*}
  \end{enumerate}

As it is known in the literature 
\cite{TD}, the Heyting Arithmetic $\HA$, enjoys  disjunction property: 
if $\HA\vdash A\vee B$, then either 
$\HA\vdash A$ or $\HA\vdash B$.
Regrettably,  $\HA$ is not able to prove this \cite{Friedman75,Myhill}. Hence, such properties, 
are not reflected in the provability logic of $\HA$,
as a valid principle 
$\Box(A\vee B)\to (\Box A\vee\Box B)$.
A natural question arises here:
{\em is there any other valid rule?} \\
One way to systematically answer this question, 
is to characterise the truth provability logic of $\HA$. 
In the case of classical arithmetic $\PA$, Robert 
Solovay in his original innovative paper \cite{Solovay}, 
characterized the truth provability logic of $\PA$. 
He showed that the only extra valid axiom is the 
soundness principle $\Box A\to A$, which is known to be 
true and unprovable in $\PA$. 
In this paper we show that, in the 
$\Sigma_1$-provablity logic of $\HA$, 
the same thing happens: 
The truth $\Sigma_1$-provability logic of $\HA$, is a 
decidable and only has the extra axiom schema 
$\Box A\to A$. 
The disjunction property, which we mentioned before, 
will be deuced from Leivant's 
principle $\Box (A\vee B)\to \Box (\Box A\vee\Box B)$ and the soundness principle.

The author of this paper in his joint paper with Mohammad Ardeshir \cite{reduction}, showed that 
the arithmetical completeness of the modal logic 
$\GL$, is reducible to the arithmetical completeness 
of $\GL+p\to \Box p$ for $\Sigma_1$ interpretations. 
The reduction involves only propositional argument. In this paper, we show that all relative 
provability logics, discussed in this paper, are reducible to the truth $\Sigma_1$-provability logic of $\HA$ (see Diagram \ref{Diagram-full}). So, in a sense, 
$\PLS(\HA,\nat)$ is the hardest among them. 

With the handful propositional reductions, we will characterize several relative provability logics for 
$\HA$, $\PA$, $\HA^*$  and 
$\PA^*$, the self-completion of $\HA$ and $\PA$ \cite{Visser82}.

\section{Definitions and Preliminaries}\label{sec-definitions}
The propositional non-modal language $\mathcal{L}_0$ contains atomic variables,
$\vee, \wedge, \ra, \bot$ and  the propositional modal language, $\mathcal{L}_\Box$ has an additional operator $\Box$. 
In this paper, the atomic propositions (\hl{in the modal} or non-modal language) \hl{include}
atomic variables and $\bot$. 
For an arbitrary proposition $A$, ${\sf Sub}(A)$ is defined to be the set
of all sub-formulae of $A$, including $A$ itself. We take
${\sf Sub}(X):=\bigcup_{A\in X}{\sf Sub}(A)$ for a set of propositions $X$.
We use 
$\Boxdot A$ as a shorthand for $A\wedge\Box A$. 
The logic \IPC is  intuitionistic propositional
non-modal logic over \hl{the usual} propositional non-modal language.
The theory $\IPC_\Box$ is the same theory \IPC in the extended language
of \hl{the propositional modal language}, i.e. its language is
\hl{the propositional modal language} and its axioms and rules are 
same as \IPC. Because we have no axioms for $\Box$
in $\IPC_\Box$, it is obvious that $\Box A$ for each $A$,
behaves exactly like an atomic variable inside $\IPC_\Box$.
First-order intuitionistic logic is denoted
 $\IQC$ and the logic $\CQC$ is its classical closure, i.e. $\IQC$ plus
the principle of excluded middle. 
 For a set of sentences and rules
$\Gamma\cup\{A\}$ in the  propositional non-modal, propositional modal
or first-order language, $\Gamma\vdash A$ means that $A$ is
derivable from $\Gamma$ in the system $\IPC, \IPC_\Box,\IQC$,
respectively. For an arithmetical formula, $\ulcorner
A\urcorner$ represents the G\"{o}del number of $A$. For an
arbitrary  arithmetical theory $T $ with a \hl{$\Delta_0(\exp)$-set of 
axioms, as far as we work in strong enough theories which is the case in this paper,} we have the $\Delta_0(\exp)$-predicate $\Prf{T}{x}{A}$, 
that is a formalization of ``$x$ is  the code of a proof
for $A$ in $T$". Note that by (inspection of the proof of) 
Craig's theorem, every recursively enumerable theory 
has a $ \Delta_0({\sf exp})$-axiomatization.
We also have the provability predicate
$\Prv{T}{A}:=\exists{x}\ \Prf{T}{x}{A}$.  The set of natural numbers is denoted by
$\omega:=\{0,1,2,\ldots\}$.

\begin{definition}\label{Definition-Arithmetical substitutions}
Suppose $\SFT$ is a  $ \Delta_0({\sf exp}) $-axiomatized 
 theory and $\sigma$ is a  substitution i.e. a
function from atomic variables to arithmetical sentences. 
We define the interpretation $\sigma_{_\SFT}$ which 
extend the substitution $\sigma$ to all modal propositions $A$, inductively:
\begin{itemize}
\item $\sigma_{_\SFT}(A):=\sigma(A)$ for atomic $A$,
\item $\sigma_{_\SFT}$ distributes over $\wedge, \vee, \ra$,
\item $\sigma_{_\SFT}(\Box A):=\Prv{T}{\sigma_{_\SFT}(A)}$.
\end{itemize}
We call $\sigma$  a $\Gamma$-substitution (in some theory $\SFT$), 
if for every
atomic $A$, $\sigma(A)\in\Gamma$ ($\SFT\vdash\sigma(A)\lr A'$ for some $A'\in \Gamma$).  We also say that $\sigma_{_T}$ is a $\Gamma$-interpretation
if $\sigma$ is a $\Gamma$-substitution.
\end{definition}

\begin{definition}\label{Definition-Provability Logic}
The relative provability logic of $\SFT$ in  some 
 sufficiently strong theory $\SFU$ restricted to a set of first-order 
 sentences $\Gamma$, is defined
to be a modal propositional theory $\PLG(\SFT,\SFU)$  such that
$\PLG(\SFT,\SFU)\vdash A$  iff for all arithmetical substitutions
$\sigma$ in $\Gamma$,  we have  $\SFU\vdash\sigma_{_\SFT}(A)$.   
We make this convention: 
$\PLG(\SFT,\mathbb{N})$ indicates 
$\PLG(\SFT,{\sf Theory}(\mathbb{N}))$, in which 
${\sf Theory}(\mathbb{N})$ is the set of all true sentences in 
the standard model of arithmetic.
\end{definition}

\noindent
Define  {\sf NOI} (No Outside Implication) as the set of modal
propositions $A$, such that any occurrence of $\ra$ is in the scope of
some $\Box$. To be able to state an extension of Leivant's
Principle (that is adequate to axiomatize $\Sigma_1$-provability
logic of $\HA$) we need a translation on the modal language which we
call  \emph{Leivant's translation}. We define it recursively as
follows:
\begin{itemize}
\item $A^l:=A$   for atomic or boxed $A$, 
\item $(A\wedge B)^l:=A^l\wedge B^l$,
\item $(A\vee B)^l:=\Boxdot A^l\vee\Boxdot B^l$,
\item $(A\ra B)^l$ is defined by cases: If $A\in {\sf NOI}$, we define
$(A\ra B)^l:=A\ra B^l$, otherwise we define $(A\ra B)^l:=A\ra B$.
\end{itemize}

Let us define the box translation $(.)^\Box$ and 
some variants of it:
\begin{itemize}
\item $A^\Boxout:=A^\Box:=\Boxdot A$ 
and $A^\Boxin:=A$ for atomic $A$ or $A:=\top,\bot$,
\item $(\Box A)^\Boxout:=\Box A$ and
 $(\Box A)^\Box:=(\Box A)^\Boxin:=\Box A^\Box$,
\item $(.)^\Boxout$, $(.)^\Box$ and $(.)^\Boxin$  commute  with $\wedge$ and $\vee$,
\item $(B\to C)^\Boxout:=\Boxdot(B^\Boxout\to C^\Boxout)$, 
$(B\to C)^\Box:=\Boxdot(B^\Box\to C^\Box)$ and 
$(B\to C)^\Boxin:= B^\Boxin\to C^\Boxin$.
\end{itemize}
\begin{remark}\label{Remark-10}
For every $A$ we have $A^\Box=(A^\Boxin)^\Boxout$. Also 
${\sf iK4}\vdash A^\Box\lr (A^\Boxout)^\Boxin$.
\end{remark}
\begin{proof}
 Both statements are proved easily 
 by induction on the complexity of $A$, and we leave them to the reader.
\end{proof}

\begin{definition} \label{Def-Axiom schema and modal theories}
Let us first we list some  axiom schemas:
\begin{itemize}
\item $\underline{\sf i}:= A$, for every theorem $A$ of $\IPC_\Box$,
\item $\underline{\sf K}:=\Box(A\to B)\to (\Box A\to \Box B)$, 
\item $\underline{\sf 4}:=\Box A\to \Box\Box A$,
\item $\underline{\sf Lob}:=\underline{\sf L}:=\Box(\Box A\to A)\to \Box A$,
\item The Completeness Principle: $\underline{\sf CP}:=\underline{\sf C}:= A\ra\Box A$.
\item Restriction of Completeness Principle to atomic variables: 
		$\underline{\sf CP_a}:=\underline{\sf C_a}:=p\ra\Box p$, 
		for atomic $p$.
\item The reflection principle:  $\underline{\sf S}:=\Box A\to A$.
\item The complete reflection principle:  
		$\underline{\sf S^*}:=\Box A\to A^\Box$.
\item The Principle of Excluded Middle: $\underline{\sf PEM}:=
\underline{\sf P}:=A\vee \neg A$.
\item Leivant's Principle: $\underline{\sf Le}:=\Box(B\vee C)\ra\Box (\Box B\vee C)$. \cite{Leivant-Thesis}
\item Extended Leivant's Principle: $\underline{\sf Le^+}:=\Box A\ra\Box A^l$. \cite{Sigma.Prov.HA}
\item  Trace Principle: $\underline{\sf TP}:=\Box(A\to B)\to (A\vee (A\to B))$. \cite{Visser82}
\item For an axiom schema $\underline{\sf A}$,  
the axiom schema $\overline{\sf A}$ indicates 
the box of  every axiom instance of $\underline{\sf A}$. 
Also ${\sf A}$ indicates $\underline{\sf A}\wedge \overline{\sf A}$.
\end{itemize}

All modal systems which will be defined here, only has one 
inference rule: modus ponens \Ax{$B$}\Ax{$B\to A$}\BI{$A$}\DP. Also the celebrated 
modal logics, like ${\sf K4}$, which has the necessitation 
rule of inference, \Ax{$A$}\UI{$\Box A$}\DP,
by abuse of notation, are considered here 
with the same name and with the same set of theorems, 
however without the necessitation rule.  The reason for this 
alternate definition of systems, is quite technical. Of course 
one may define them with the necessitation rule, but 
at the cost of loosing the uniformity of definitions. 
So in the rest of this paper, all modal systems, are considered 
with the modus ponens rule of inference.  

Consider a list ${\sf A}_1,\ldots,{\sf A}_n$ of axiom schemas.
The notation ${\sf A}_1{\sf A}_2\ldots{\sf A}_n$ will be used in this paper for a modal system containing all axiom instance of 
all axiom schemas ${\sf A}_i$, and is closed under modus ponens. 
This genral notation makes things uniform and easy to remember for later usage. However, we make the following exceptions:
\begin{itemize}
\item $\GL:={\sf iGLP}$,
\item $\GLS:=\GL$ plus $\underline{\sf S}$. We may define similarly 
$\GLSV$ and $\GLV$.
\end{itemize}
We also gathered the list of axioms and theories in Tables  
\ref{Table-Axioms} and \ref{Table-Theories}.
\end{definition}

\begin{lemma}\label{Lemma-BoxdotABox-ABox}
For every modal proposition  $A$,  we have  
${\sf iK4} \vdash  A^\Box\lr \Boxdot A^\Box$.
\end{lemma}
\begin{proof}
Use induction on the complexity of $A$.
\end{proof}

\begin{lemma}\label{Lemma-A-ABoxin}
For every modal proposition  $A$,  we have  
${\sf iK4}+\Boxdot\CP\vdash A\leftrightarrow A^\Box$ and 
 ${\sf iK4}+\Box\CP\vdash A\leftrightarrow A^\Boxin$.
\end{lemma}
\begin{proof}
Note that the first assertion implies the second one. 
To prove the equivalence of $A$ and $A^\Box$ in 
${\sf iK4}+\Boxdot\CP$, one must use induction 
on the complexity of $A$. All cases are simple and 
left to the reader. 
\end{proof}

\begin{lemma}\label{Lemma-3}
For every modal proposition  $A$,  we have  ${\sf iGL}\vdash A$ implies 
${\sf iGL}\vdash A^\Boxout\wedge A^\Boxin\wedge A^\Box$. The same holds for $\iglsigma$.
\end{lemma}
\begin{proof}
Use induction on the complexity of proof $\iGL\vdash A$.
\end{proof}

\begin{lemma}\label{Lemma-4-2}
Let $A$ be some proposition and $E\in {\sf sub}(A^\Box)$. 
Then ${\sf iK4}+{\sf CP_a}\vdash E^\Boxout\to \Boxdot E$.
\end{lemma}
\begin{proof}
Use induction on the complexity of $E$. All cases are trivial 
except for $E=F^\Box\to G^\Box$. In this case we have 
$E^\Boxout=\Boxdot ((F^\Box)^\Boxout\to (G^\Box)^\Boxout)$.
One may observe that $(A^\Box)^\Boxout\leftrightarrow A^\Box$ is valid in ${\sf iK4}$ and hence we have 
${\sf iK4}\vdash E^\Boxout\leftrightarrow \Boxdot E$.
\end{proof}

\subsection{Preliminaries from Arithmetic}\label{Sec-2-1}

The first-order language of arithmetic contains three functions
(successor, addition and multiplication), one predicate symbol
and a constant: \hl{$({\sf S},+,\cdot\,,\leq,0)$}. First-order intuitionistic
arithmetic ($\HA$) is the theory over $\IQC$ with the axioms:
 \begin{enumerate}
  \item[Q1] \hl{${\sf S}x\neq 0$,}
  \item[Q2] \hl{${\sf S}x={\sf S}(y)\ra x=y$,}
  \item[Q3] \hl{$x+0=x$,}
  \item[Q4] \hl{$x+{\sf S}y={\sf S}(x+y)$,}
  \item[Q5] \hl{$x.0=0$,}
  \item[Q6] \hl{$x.{\sf S}y=(x.y)+x$,}
  \item[Q7] \hl{$x\leq y\lr\exists{z} \, z+x=y $,}
  \item[Ind:] For each formula $A(x)$:
  \begin{center}
    \hl{${\sf Ind}(A,x):=\mathcal{UC}[(A(0)\wedge\forall{x}(A(x)\ra A({\sf S}x)))\ra\forall{x}A(x)]$}
    \end{center}
    In which $\mathcal{UC}(B)$ is the universal closure  of $B$.
 \end{enumerate}
Peano Arithmetic \PA\!\!, has the same axioms of \HA over
\CQC\!\!.   

\begin{notation}
From now on, when we are working in the
first-order language of
arithmetic, for a first-order sentence $A$, 
the notations $\Box A$ and
$\Box^+A$ are  shorthand for 
$\Prv{HA}{A}$ and  $\Prv{PA}{A}$, respectively. 
 Let $i{\Sigma}_1$ be the
theory $\HA$, where the induction principle is restricted to
$\Sigma_1$-formulae. 
We also define  $\HA_x$ to be
the theory with axioms of $\HA$, in which the induction principle
is restricted to  formulae satisfying at least one of the following conditions:
\begin{itemize}
\item  $\Sigma_1$-formulas,
\item formulae with G\"{o}del number less than $x$.
\end{itemize}
We  define  $\PA_x$ similarly.  
Also
define $\Box_x A$ and  $\Box^+_x A$ to be 
 provability predicates in $\HA_x$ and $\PA_x$ , respectively.
\end{notation}

\begin{lemma}\label{Lemma-Reflection}
For every formula $A$, we have $\PA\vdash \forall{x}\
\Box^+(\Box^+_xA\ra A)$ and  $\HA\vdash\forall{x}\
\Box(\Box_xA\ra A)$.
\end{lemma}
\begin{proof}
The case of $\PA$ is well known \cite{HP}. For the case $\HA$, see
\cite{Smorynski-Troelstra} or  \cite[Theorem 8.1]{Visser02}.
\end{proof}

\begin{lemma}\label{Lemma-bounded Sigma completeness}
$\HA$  proves all true $\Sigma_1$ sentences. Moreover 
this argument is formalizable and provable in $\HA$, i.e. 
for every $\Sigma_1$-formula $A(x_1,\ldots,x_k)$ we have
$\HA\vdash{
A(x_1,\ldots,x_k)\ra\Box A(\dot{x}_1,\ldots,\dot{x}_k)}$.
\end{lemma}
\begin{proof}
It is a well-known fact that any true (in the standard model
$\mathbb{N}$) $\Sigma_1$-sentence is provable in $\HA$
\cite{Visser02}. 
Moreover this argument is constructive and formalizable in
$\HA$.
\end{proof}

\begin{lemma}\label{Lemma-decidability of delta formulae}
For any \hl{$\Delta_0(\exp)$}-formula $A(\bar{x})$, we have
$\HA\vdash\forall\bar{x}(A(\bar{x})\vee\neg A(\bar{x}))$.
\end{lemma}
\begin{proof}
This is well-known in the literature \cite{TD}.
\end{proof}
\begin{lemma}\label{Lemma-Conservativity of HA}
Let $A$, $B$ be  $\Sigma_1$-formulae such that $\PA\vdash A\ra
B$. Then $\HA\vdash A\ra B$.
\end{lemma}
\begin{proof}
\hl{Observe that every implication of $\Sigma_1$-sentences in $\HA$ is equivalent to a $\Pi_2$ sentence and use the  $\Pi_2$-conservativity of
$\PA$ over $\HA$} \cite{TD}(3.3.4).
\end{proof}

\begin{definition}\label{Definition-First-order-translation}
For a first-order theory $\SFT$ and 
first-order arithmetical formula $A$, the Beeson-Visser translation $A^\SFT$ is 
defined  as follows:
\begin{itemize}
\item $A^{\SFT}:=A$ for atomic $A$,
\item $(.)^{\SFT}$ commutes with $\wedge,\vee$ and $\exists$,
\item $(A\to B)^{\SFT}:=(A^{\SFT}\to B^\SFT)\wedge
\Prv{T}{A^\SFT\to B^\SFT}$
\item $(\forall{x}A)^\SFT:=\forall{x}A^\SFT\wedge \Prv{T}{\forall{x}A^\SFT}$.
\end{itemize}
$\HA^*$ and $\PA^*$ were first introduced in \cite{Visser82}.
These theories are defined as
$$\HA^*:=\{A\mid\HA\vdash A^{{\sf HA}}\} \quad \text{ and }\quad
\PA^*:= \{A\mid\PA\vdash A^{\sf PA}\}.$$
\end{definition}

Visser in \cite{Visser82} showed that the 
($\Sigma_1$-)provability logic of
$\PA^*$ is $\iglct$, i.e. $\iglct\vdash A$ iff  for all
arithmetical substitution $\sigma$, \ $\PA^*\vdash \sigma_\tinysub{{\sf
PA^*}}(A)$. That means that $$\PL(\PA^*)=\PLS(\PA^*)=\iglct.$$

\begin{lemma}\label{Lemma-Properties of Box translation}
For any arithmetical $\Sigma_1$-formula $A$
\begin{enumerate}
\item $\HA\vdash A\lr{A^{\sf HA}}$,
\item $\HA\vdash A\lr {A^{\sf PA}}$.
\end{enumerate}
\end{lemma}
\begin{proof}
 See \cite[4.6.iii]{Visser82}.
\end{proof}
\begin{lemma}\label{Lemma-Gist-HA*-PA*}
For every arithmetical sentence $A$ we have 
\begin{itemize}
\item $\HA\vdash \Prv{\HA}{A}\to \Prv{\HA^*}{A}$,
\item $\HA^*\vdash A\to \Prv{\HA}{A}$,
\item $\PA^*\vdash A\to \Prv{\PA}{A}$.
\end{itemize}
\end{lemma}
\begin{proof}
For the first item, consider some $A$ such that $\HA\vdash A$. By induction on the proof of $A$ in $\HA$, one may prove that 
$\HA\vdash A^\HA$. Moreover this argument is formalizable and provable in $\HA$. We refer the reader to \cite{Visser82} for details. \\
For the proof of second and third items, one may use
induction on the complexity of $A$, and we leave the routine induction to the reader.
\end{proof}


\begin{lemma}\label{Label-HA-HA*-Box-trans2}
For any $\Sigma_1$-substitution $\sigma$ and each propositional
modal sentence $A$, we have 
$\HA\vdash (\sigma_\tinysub{{\sf HA}^*}(A))^{\sf HA}\leftrightarrow 
\sigmapas(A^\Boxout) $ and 
$\PA\vdash (\sigmapas{A})^{\sf PA}\leftrightarrow 
\sigmapas{A^\Boxout} $.
\end{lemma}
\begin{proof}
Use induction on the complexity of $A$. 
All cases are easily derived by  
\Cref{Lemma-Properties of Box translation}.
\end{proof}

\begin{lemma}\label{Label-HA-HA*-Box-trans}
For any $\Sigma_1$-substitution $\sigma$ and each propositional
modal sentence $A$, we have 
$\HA\vdash \sigma_\tinysub{\sf HA}(A^\Box)\leftrightarrow 
(\sigma_\tinysub{{\sf HA}^*}(A))^{\sf HA}$ and 
$\HA\vdash \sigma_\tinysub{\sf PA}(A^\Box)\leftrightarrow 
(\sigma_\tinysub{{\sf PA}^*}(A))^{\sf PA}$.
\end{lemma}
\begin{proof}
Use induction on the complexity of $A$. 
All cases are easily derived by  
\Cref{Lemma-Properties of Box translation}.
\end{proof}

\begin{lemma}\label{Lemma-Properties of Box translation 2}
For any $\Sigma_1$-substitution $\sigma$ and each propositional
modal sentence $A$, we have $\HA\vdash \sigma_\tinysub{{\sf
HA}}(A^\Boxin)\lr\sigma_\tinysub{{\sf HA}^*}(A)$
and $\HA\vdash \sigmapa{A^\Boxin}\lr\sigmapas{A}$.
\end{lemma}
\begin{proof}
We use induction on the complexity of $A$.
All cases are easy, except for boxed case, which holds by
 \Cref{Label-HA-HA*-Box-trans}.
\end{proof}

\begin{lemma}\label{Lemma-PA-Box-translate}
For any $\Sigma_1$-substitution $\sigma$ and each propositional
modal sentence $A$, we have $\HA\vdash \sigma_\tinysub{{\sf
PA}}(A^\Boxin)\lr\sigma_\tinysub{{\sf PA}^*}(A)$.
\end{lemma}
\begin{proof}
We use induction on the complexity of $A$.
All cases are easy, except for boxed case, which holds by
 \Cref{Label-HA-HA*-Box-trans}.
\end{proof}
\subsubsection{Kripke models of \HA}\label{sec-KripkeModelFirstOrder}
A first-order Kripke model for \hl{the language of arithmetic} is a triple $\kcal=(K,\preccurlyeq,\mathfrak{M})$ such that:
\begin{itemize}
\item  The frame of $\kcal$, i.e. $(K,\prec)$,
is a non-empty partially ordered set,
\item $\mathfrak{M}$ is  a function from $K$ to the first-order classical structures for the language 
of the arithmetic, i.e. $\mathfrak{M}(\alpha)$ is a first-order classical structure, for each $\alpha\in K$,
\item For any $\alpha\preccurlyeq\beta\in K$, $\mathfrak{M}(\alpha)$ is a weak  substructure of
$\mathfrak{M}(\beta)$.   
\end{itemize}
For any  $\alpha\in K$ and  first-order formula $A\in\mathcal{L}_\alpha$ (the language of arithmetic augmented with constant symbols $\bar{a}$ for each $a\in|\mathfrak{M}(\alpha)|$),
 we define $\kcal,\alpha\Vdash A$
 (or simply $\alpha\Vdash A$, if no confusion is likely) inductively as follows:
 \begin{itemize}
 \item For atomic $A$, \hl{$\kcal,\alpha\Vdash A$} iff $\mathfrak{M}(\alpha)\models A$. 
 Note that in the structure $\mathfrak{M}(\alpha)$, $\bar{a}$ is 
 interpreted as $a$,
 \item $\kcal,\alpha\Vdash A\vee B$ iff $\kcal,\alpha\Vdash A$ or $\kcal,\alpha\Vdash B$,
 \item $\kcal,\alpha\Vdash A\wedge B$ iff $\kcal,\alpha\Vdash A$ and $\kcal,\alpha\Vdash B$,
 \item $\kcal,\alpha\Vdash A\ra B$ iff for all \hl{$\beta\succcurlyeq\alpha$}, $\kcal,\beta\Vdash A$ implies $\kcal,\beta\Vdash B$,
 \item \hl{$\kcal,\alpha\Vdash \exists x A$ iff $\kcal,\alpha\Vdash A[x:\bar{a}]$, for some $a\in |\mathfrak{M}(\alpha)|$,}
 \item \hl{$\kcal,\alpha\Vdash \forall x A$ iff for all  $\beta\succcurlyeq\alpha$ and  
 $b\in|\mathfrak{M}(\beta)|$, we have $\kcal,\beta\Vdash A[x:\bar{b}]$.}
 \end{itemize}
 It is well-known in the literature \cite{TD}
  that $\HA$ is complete for first-order Kripke models. 
 \begin{lemma}\label{Lemma-Sigma-local-global}
 Let $\kcal=(K,\preccurlyeq,\mathfrak{M})$ be a Kripke model of $\HA$ and $A$ be an arbitrary $\Sigma_1$-formula.
Then for each $\alpha\in K$, we have $\alpha\Vdash A$ iff $\mathfrak{M}(\alpha)\models A$. 
 \end{lemma}
 \begin{proof}
 Use induction on the complexity of $A$ to show that for each $\alpha\in K$, we have 
 $\alpha\Vdash A$ iff $\mathfrak{M}(\alpha)\models A$. In the inductive step for $\to$ and $\forall$,
 use \Cref{Lemma-decidability of delta formulae}.
 \end{proof}
 
\subsubsection{Interpretability}
Let $T$ and $S$ be two first-order theories. Informally speaking,
we say that $T$ interprets $S$ ($T\rhd S$) if there exists a
translation from the language of $S$ to the language of $T$ such
that $T$ proves the translation of all of the theorems of $S$.
For a formal definition see \cite{VisserInterpretability}. It is
well-known that for recursive theories $T$ and $S$ containing $\PA$,
the assertion $T\rhd S$ is formalizable  in first-order language
of arithmetic. For two arithmetical sentences $A$ and $B$, we use the
notation $A\rhd B$ to mean that $\PA+A$ interprets $\PA+B$. The
following theorem due to Orey, first appeared in \cite{Feferman}.

\begin{theorem}\label{Theorem-Orey}
For recursive theories $T$ and $S$ containing $\PA$, we have: 
\[ \PA\vdash (T\rhd S) \lr \forall{x}\, \Box_T {\sf Con}(S^x),\]
 in which $S^x$ is the
restriction of the theory $S$  to axioms with G\"{o}del number
$\leq x$ and ${\sf Con}(U):=\neg\,\Box_U\bot$.
\end{theorem}
\begin{proof}
See \cite{Feferman}. p.80 or \cite{Berarducci}.
\end{proof}

\noindent\textbf{Convention.}
From \Cref{Theorem-Orey}, one can easily observe that $\PA\vdash {(A\rhd B)}\lr{\forall{x}\,\Box^+(A\ra\neg\,\Box^+_x\neg B)}$.
So from now on, 
$A\rhd B$  means its $\Pi_2$-equivalent
$\forall{x}\,\Box^+(A\ra\neg\,\Box^+_x\neg B)$, even when we
are working in weaker theories like $\HA$, \hl{for which the above theorem} (\Cref{Theorem-Orey}) \hl{doesn't hold}. We remind the reader
that $\Box^+$ stands for provability in $\PA$.

\subsubsection{Somr\'ynski's method for Constructing Kripke models of \HA}
With the general method of constructing Kripke models for $\HA$, invented by Smory\'nski \cite{Smorynski-Troelstra},
interpretability of theories containing $\PA$ plays an important role
in constructing Kripke models of \nolinebreak$\HA$.

\begin{definition}\label{Definition-Iframe}
	A triple \hl{$\mathcal{I}:=(K,\preccurlyeq,T)$} is called an I-frame
	iff it has the following properties:
	\begin{itemize}
		\item \hl{$(K,\preccurlyeq)$} is a finite tree,
		\item $T$ is a function from $K$ to arithmetical r.e. consistent
		theories containing $\PA$,
		\item if \hl{$\beta\preccurlyeq \gamma$}, then $T_\beta$ interprets $T_\gamma$  $(\, T_\beta\rhd T_\gamma\,)$.
	\end{itemize}
\end{definition}

\begin{theorem}\label{Theorem-Smorynski's general method of Kripke model construction}
	For every I-frame \hl{$\mathcal{I}:=(K,\preccurlyeq,T)$} there exists
	a first-order Kripke model  \hl{$\kcal={(K,\preccurlyeq,\mathfrak{M})}$} such that  
	$\kcal\Vdash \HA$ and moreover $\mathfrak{M}(\alpha)\models T_\alpha$,  for any $\alpha\in K$.
	Note that both of the I-frame and  Kripke model are sharing the same frame \hl{$(K,\preccurlyeq)$}.
\end{theorem}
\begin{proof} See \cite[page~372-7]{Smorynski-Troelstra}. For more detailed proof of a generalization of this theorem, see
	\cite[Theorem~4.8]{ArMo14}.
\end{proof}

\subsection{The NNIL formulae and related topics}\label{sec-nnil}
The class of {\em No Nested Implications to the Left}, \NNIL
formulae in a propositional language was introduced in
\cite{Visser-Benthem-NNIL}, and more explored in \cite{Visser02}. 
The crucial
result of \cite{Visser02} is providing an algorithm that as
input, gives a non-modal proposition $A$ and returns its best \NNIL
approximation $A^*$ from below, i.e., $\IPC\vdash A^*\ra A$ and
for all \NNIL \hl{formulae} $B$ such that $\IPC\vdash B\ra A$, we have
$\IPC\vdash B\ra A^*$. Also for all $\Sigma_1$-substitutions $\sigma$, 
we have $\HA\vdash \sigma_{_{\sf HA}}(\Box A\lr \Box A^*)$ \cite{Visser02}.

The precise definition of the class $\NNIL$ of modal propositions is
  $\NNIL:= \{A\mid \rho A\leq 1\}$, in which 
the complexity measure $\rho$, is defined inductively as follows:
\begin{itemize}[leftmargin=*]
\item $\rho(\Box A)=\rho(p)= \rho(\bot)=\rho(\top) = 0$, for an arbitrary atomic variables $p$ and modal proposition $A$,
\item $\rho(A\wedge B) = \rho(A\vee B) = \text{max} (\rho A, \rho B)$,
\item $\rho(A\ra B) = \text{max} (\rho A +1, \rho B)$,
\end{itemize}

\begin{definition}\label{definition-braket}
For any two modal propositions $A$ and $B$, we define $[A]B$  
 by induction on the complexity of $B$:
\begin{itemize}
\item $ [A]B =B$, for atomic or boxed $B$,
\item $[A](B_1\circ B_2) =[A](B_1)\circ [A](B_2)$
for $\circ\in\{\vee,\wedge\}$,
\item  $[A](B_1\ra B_2)=A'\to( B_1\ra B_2)$,
in which $A' = {A[B_1\ra B_2\mid B_2]}$, i.e., replace each \emph{outer
occurrence} of $B_1 \ra B_2$ \hl{\textup{(}}by outer occurrence we mean that it is not in the scope of any $\Box$\hl{\textup{)}}  in $A$ by $B_2$,
\end{itemize}
For a set $X$ of modal propositions, we also define $[A]X:=\bigvee_{B\in X}\hlf{[A]B}$.
\end{definition}

\subsubsection*{The $\NNIL$-algorithm}\label{subsubsec-NNIL-algorithm}
For each modal proposition $A$, the proposition 
$A^*$ is defined inductively  as follows \cite{Visser02}:
\begin{enumerate}[leftmargin=*]
\item $A$ is atomic or boxed, take $A^*:=A$.
\item $ A=B \wedge C$, take $A^*:=B^*\wedge C^*$.
\item $ A=B\vee C$, take $A^*:=B^*\vee C^*$.
\item $ A=B \ra C $, we have several sub-cases. In the following, 
an occurrence of $E$ in $D$ is called an {\em outer
occurrence}, if $E$ is neither in the scope of an implication nor
in the scope of a boxed formula.
\begin{enumerate}[leftmargin=*]
\item $C$ contains an outer occurrence of a conjunction. In this
case, there is some formula $J(q)$ such that
\begin{itemize}
\item $q$ is a propositional variable not occurring in $A$.
\item $q$ is outer in $J$ and occurs exactly once.
\item $C=J[q|(D\wedge E)]$.
\end{itemize}
Now set $C_1:=J[q|D], C_2:=J[q|E]$ and
$A_1:=B\ra C_1, A_2:=B\ra C_2$ and finally, define
$A^*:=A_1^*\wedge A_2^*$.
\item $B$ contains an outer occurrence of a disjunction. In this
case, there is some formula $J(q)$ such that
\begin{itemize}
\item $q$ is a propositional variable not occurring in $A$.
\item $q$ is outer in $J$ and occurs exactly once.
\item $B=J[q|(D\vee E)]$.
\end{itemize}
 Now set $B_1:=J[q|D], B_2:=J[q|E]$
and ${A_1:=B_1\ra C}, {A_2:=B_2\ra C}$ and finally, define
$A^*:=A_1^*\wedge A_2^*$.
\item  $B=\bigwedge X$ and $C=\bigvee Y$ and $X,Y$ are sets of
implications, atomics or boxed formulas. We have several sub-cases:
\begin{enumerate}[leftmargin=*]
\item $X$ contains \hl{an} atomic variable or \hl{a} boxed formula $E$. We
set $D:=\bigwedge(X\setminus\{E\})$ and take
 ${A^*:=E^*\ra(D\ra C)^*}$.
\item $X$ contains $\top$. Define
$D:=\bigwedge(X\setminus\{\top\})$ and take $A^*:=(D\ra C)^*$.
\item $X$ contains $\bot$. Take $A^*:=\top$.
\item $X$ contains only implications. For any $D=E\ra F\in X$,
define
$$B\!\downarrow\! D:=\bigwedge((X\setminus\{D\})\cup\{F\}).$$
Let $Z:=\{E\mid E\ra F\in X\}\cup\{C\}$ and define: 
\begin{align*}
A^*:=\bigwedge\{((B\!\downarrow\! D)\ra C)^*|D\in X\}\wedge \bigvee \{([B]E)^*\mid E\in Z\}
\end{align*}
\end{enumerate}
\end{enumerate}
\end{enumerate}

\begin{lemma}\label{Lemma-NNIL-IPC}
If  $\IPC_\Box\vdash A\to B$ then $\IPC_\Box\vdash A^*\to B^*$.
\end{lemma}
\begin{proof}
See \cite[Theorem.~4.5]{Sigma.Prov.HA}.
\end{proof}

\subsubsection*{The $\TNNIL$-algorithm}\label{subsubsec-TNNIL-algorithm}

\begin{definition}\label{Def-TNNIL-Propositions}
$\TNNIL$ \hl{\textup{(}}Thoroughly $\NNIL$\hl{\textup{)}} is the
smallest class of propositions such that
\begin{itemize}
\item $\TNNIL$ contains all atomic propositions,
\item if $A, B\in\TNNIL$, then $A\vee B, A\wedge B,\Box A\in\TNNIL$,
\item if all $\ra$ occurring in $A$ are contained in the scope of a $\Box$ \hl{\textup{(}}or equivalently
$A\in {\sf NOI}$\hl{\textup{)}} and $A,
B\in\TNNIL$, then $A\ra B\in\TNNIL$.
\end{itemize}
Let  $\TNNIL^\Box$  indicates the set of all the propositions
like $  A(\Box B_1,\ldots,\Box B_n)$, such that
$A(p_1,\ldots,p_n)$ is an arbitrary non-modal proposition and
$B_1,\ldots,B_n\in\TNNIL$.
\end{definition}

Here we define $A^+$  to be the   $\TNNIL$-formula approximating $A$.
Informally speaking, to find $A^+$, we first compute $A^*$ and
then replace all outer boxed formula $\Box B$ in $A$ by
$\Box B^+$. More precisely, we  define $A^+$ by induction
on the maximum  number of nesting $\Box$'s.  Suppose
that $A'(p_1,\ldots,p_n)$ and $\Box B_1,\ldots,\Box B_n$ are
such that $A=A'[p_1|\Box B_1,\ldots,p_n|\Box B_n]$, where
$A'$ is a non-modal proposition and $p_1,\ldots,p_n $ are fresh
atomic variables (not occurred in $A$). It is clear that
each $B_i$ has less number of nesting $\Box$'s and then we can define
$A^+:=(A')^*[p_1|\Box B_1^+,\ldots,p_n|\Box B_n^+]$.\\
For a modal proposition $A$, let
	$B(p_1,\ldots,p_n)$ is the unique 
	(modulo permutation of $p_i$)  non-modal proposition 
	such that $A:=B(\Box C_1,\ldots,\Box C_n)$.
Then define $A^-:=B(\Box C_1^+,\ldots,\Box C_n^+)$.
Next we may define the theory $\lles$ as follows:
\begin{definition}\label{Def-lles}
We define the Visser's axiom schema 
$$\underline{\sf V}:=A\lr A^-$$
Then define the following modal systems:
\begin{itemize}
\item $\lles:={\sf iGLLe^+V}$,
\item  $\llessstar:=\{ A : \lles\vdash A^\Box\}$,
\item $\llesstar:=\{ A : \lles\vdash A^\Boxin\}$.
\end{itemize}
\end{definition}
\begin{remark}
The definitions of $\lles$ in \cite[sec.~4.3]{Sigma.Prov.HA} and 
$\llessstar$ in \cite[def.~3.16]{Sigma.Prov.HA*} (which were called $\llesstar$ there) are presented 
in some other equivalent way. For the sake of simplicity of definitions, 
we preferred \Cref{Def-lles} here. To see an axiomatization for
 $\llessstar$, we refer the reader to \cite{Sigma.Prov.HA*}.
\end{remark}

\begin{lemma}\label{Lemma-TNNIL-modusponens}
 $\IPC_\Box\Vdash (A^+\wedge (A\to B)^+)\to B^+$.
\end{lemma}
\begin{proof}
By definition of $(.)^+$, for every $C$ we have $C^+=(C^-)^*$. 
Since $\IPC_\Box\vdash (A^-\wedge (A\to B)^-)\to B^-$, by
\Cref{Lemma-NNIL-IPC} we have 
$\IPC_\Box\vdash (A^-\wedge (A\to B)^-)^*\to (B^-)^*$. Then we 
have  $\IPC_\Box\vdash {(A^+\wedge (A\to B)^+)}\to B^+$ by the argument at the beginning of proof.
\end{proof}
\begin{lemma}\label{Lem-51}
Let $A$ be a modal proposition. Then ${\sf iK4}\vdash A^\Boxout\leftrightarrow \Boxdot A^\Boxout $.
\end{lemma}
\begin{proof}
Use induction on the complexity of $A$.
\end{proof}

\begin{lemma}\label{Lem-60a}
For arbitrary $A\in\TNNIL^\Box$ we have 
${\sf iK4}+{\sf CP_a}\vdash \Boxdot A^l\leftrightarrow \Boxdot A^\Boxout$.
\end{lemma}
\begin{proof}
 We use induction on the complexity of $A$:
 \begin{itemize}
 \item $A$ is atomic: then $A^l=A$ and $A^\Boxout=\Boxdot A$. Hence by 
 ${\sf CP_a}$ we have the desired equivalency.
 \item $A$ is boxed: $(\Box A)^l=\Box A=(\Box A)^\Boxout$.
 \item $A=B\wedge C$: then $(B\wedge C)^l=B^l\wedge C^l$ and 
 $(B\wedge C)^\Boxout=B^\Boxout\wedge C^\Boxout$. Hence by induction hypothesis we have the desired result.
 \item $A=B\vee C$: then $(B\vee C)^l=\Boxdot B^l\vee \Boxdot C^l$ and 
 $(B\vee C)^\Boxout= B^\Boxout\vee C^\Boxout$. Using \Cref{Lem-51}
 we have ${\sf iK4}\vdash (B\vee C)^\Boxout\leftrightarrow 
 (\Boxdot B^\Boxout\vee \Boxdot C^\Boxout)$ and hence induction hypothesis 
 implies the desired result.
 \item $A=B\to C$ and $B\in{\sf NOI}$: then $(B\to C)^l=B\to C^l$ and
 $(B\to C)^\Boxout=\Boxdot(B^\Boxout\to C^\Boxout)$. 
Observe that 
\begin{enumerate}
\item 
${\sf iK4}\vdash \Boxdot \Boxdot E\leftrightarrow \Boxdot E$ for any $A$, 
\item 
${\sf iK4}+{\sf CP_a}\vdash B^\Boxout\leftrightarrow B$,
\item  ${\sf iK4}\vdash \Boxdot (B\to E)\leftrightarrow \Boxdot(B\to\Boxdot E)$ for any 
$B\in {\sf NOI}$ and arbitrary $E$.
\end{enumerate}
We have the following equivalences in  ${\sf iK4}+{\sf CP_a}$:
\begin{align*}
\Boxdot A^\Boxout &\leftrightarrow \Boxdot(B^\Boxout\to C^\Boxout) 
&\text{by first observation}\\
\Boxdot(B^\Boxout\to C^\Boxout)&\leftrightarrow 
\Boxdot(B\to C^\Boxout)
&\text{by  second observation}\\
\Boxdot(B\to C^\Boxout)&\leftrightarrow \Boxdot(B\to \Boxdot C^\Boxout)
&\text{by  third observation}\\
\Boxdot(B\to \Boxdot C^\Boxout)&\leftrightarrow \Boxdot(B\to \Boxdot C^l)
&\text{by induction hypothesis}\\
\Boxdot(B\to \Boxdot C^l)&\leftrightarrow \Boxdot(B\to   C^l)
&\text{by third observation}
\end{align*}   \qedhere
 \end{itemize}
\end{proof}
\begin{lemma}\label{Lem-Boxin-validity}
For $A\in\TNNIL^\Box$ we have 
${\sf iK4}+{\sf Le}^++{\sf CP_a}\vdash A\leftrightarrow A^\Boxin$.
\end{lemma}
\begin{proof}
Use induction on the complexity of $A$. The only nontrivial case is when $A=\Box B$.
We have the following equivalences  in ${\sf iK4}+{\sf Le}^++{\sf CP_a}$:
\begin{align*}
(\Box B)^\Boxin &\leftrightarrow \Box (B^\Box)
&\text{by definition}\\
\Box (B^\Box) &\leftrightarrow \Box\left( (B^\Boxin)^\Boxout\right)
&\text{by \Cref{Remark-10}}\\
\Box\left( (B^\Boxin)^\Boxout\right) &\leftrightarrow 
\Box( B^\Boxout)
&\text{by induction hypothesis}\\
\Box( B^\Boxout) &\leftrightarrow \Box B^l
&\text{by \Cref{Lem-60a}}\\
\Box B^l &\leftrightarrow \Box B 
&\text{by the axiom schema ${\sf Le}^+$}
\end{align*}
\end{proof}

\begin{theorem}\label{Theorem-conservativity-1}
For any $\TNNIL$-proposition $A$, 
$\iglc\vdash A$ implies $\iglleplus\vdash A$. 
\end{theorem}
\begin{proof}
See \cite{Sigma.Prov.HA} Theorem 4.24.
\end{proof}
\begin{theorem}\label{theorem-conservativiity-ihatglchat}
For any $\TNNIL^\Box$-proposition $A$, 
$\ihatglchat\vdash A$ implies $\ihatglleplus \vdash A$. Also 
$\ihatglchats\vdash A$ implies $\ihatgllepluss \vdash A$. 
\end{theorem}
\begin{proof}
Both statements proved by induction on proofs. 
The only non-trivial case is when  $A$ is an
 axiom instance of the form $\Box A$ such that 
$\iglc\vdash A$.  In this case, \Cref{Theorem-conservativity-1} implies 
$\iglleplus\vdash A$. Hence by necessitation which is available in 
$\iglleplus$ we have $\iglleplus\vdash \Box A$. Hence 
$\ihatglleplus\vdash A$ and $\ihatgllepluss\vdash A$.
\end{proof}

\subsection{Intuitionistic Modal Kripke Semantics}\label{Sec-Kripke}
Let us first review results and notations from \cite{IemhoffT}
which will be used here. Assume two binary relations $R$ and $S$
on a set. Define $\alpha(R\hlf{\,;}S)\gamma$ iff there exists some $\beta$ such that 
$\alpha R \beta $ and $\beta S \gamma$.
\hl{We use the binary relation symbol $\preccurlyeq$ always as a  reflexive 
relation and $\prec$ for the irreflexive part of $\preccurlyeq$,
i.e. $u\prec v$ holds iff $u\preccurlyeq v$ and $u\neq v$. Moreover we use the mirror image of a relational symbol for its inverse, e.g.
$\succ$ for $\prec^{-1}$ and so on.} 

A Kripke model $\mathcal{K}$, for intuitionistic modal logic, is
a quadruple $(K,\hlf{\preccurlyeq},\R,V)$, such that $K$ is a set (we call
its elements as nodes), $(K,\hlf{\prec})$ is a partial ordering, $\R$ is a
binary relation on $K$ such that $(\hlf{\preccurlyeq}\hlf{\,;}\, \R)\subseteq\;\R$, and
$V$ is a binary relation between nodes and atomic
variables such that $\alpha V p$ and $\alpha \hlf{\preccurlyeq}\beta$ implies $\beta V
p$. Then we can extend $V $ to the modal language with $\R$
corresponding to $\Box$ and $\hlf{\preccurlyeq} $ for intuitionistic $\ra$.
More precisely, we define $\Vdash$ inductively as an extension of $V$ as follows:
\begin{itemize}
 \item $\kcal,\alpha\Vdash p$ iff 
 $\alpha Vp$, for atomic variable $p$,
 \item $\kcal,\alpha\Vdash A\vee B$ iff $\kcal,\alpha\Vdash A$ or $\kcal,\alpha\Vdash B$,
 \item $\kcal,\alpha\Vdash A\wedge B$ iff $\kcal,\alpha\Vdash A$ and $\kcal,\alpha\Vdash B$,
 \item $\kcal,\alpha\nVdash\bot$ and $\kcal,\alpha\Vdash\top$,
 \item $\kcal,\alpha\Vdash A\ra B$ iff for all $\beta 
 \hlf{\succcurlyeq} \alpha$, $\kcal,\beta\Vdash A$ 
 implies $\kcal,\beta\Vdash B$,
 \item $\kcal,\alpha\Vdash \bo A$ iff for all $\beta$ with $\alpha\R \beta$, we have $\kcal,\beta\Vdash A$.
\end{itemize}
Also we define the local truth in this way:
\begin{itemize}
 \item $\kcal,\alpha\models p$ iff 
 $\alpha Vp$, for atomic variable $p$,
 \item $\kcal,\alpha\models A\vee B$ iff $\kcal,\alpha\models A$ or $\kcal,\alpha\models B$,
 \item $\kcal,\alpha\models A\wedge B$ iff $\kcal,\alpha\models A$ and $\kcal,\alpha\models B$,
 \item $\kcal,\alpha\not\models\bot$ and $\kcal,\alpha\models\top$,
 \item $\kcal,\alpha\models A\ra B$ iff 
 either $\kcal,\alpha\not\models A$ or
  $\kcal,\alpha\models B$, 
   \item $\kcal,\alpha\models \bo A$ iff for all $\beta$ with $\alpha\R \beta$, we have $\kcal,\beta\Vdash A$.
\end{itemize}
The classical truth $\kcal,\alpha\models_cA$ is defined similar to $\kcal,\alpha\models A$, except for the 	boxed case:
\begin{itemize}
\item $\kcal,\alpha\models_c \Box A$ iff for all $\beta\sqsupset \alpha$ we have $\kcal,\beta\models_c A$.
\end{itemize}

For a boolean interpretation $I$, we also define 
the local $I$-truth $\kcal,\alpha,I\models A$ 
and 
the classical $I$-truth $\kcal,\alpha,I\models_c A$, 
similar to $\kcal,\alpha\models A$, 
and $\kcal,\alpha\models A$, 
except for atomic variables $p$ which we define:
\begin{itemize}
\item $\kcal,\alpha,I\models p$ 
iff $I\models p$
iff $\kcal,\alpha,I\models_c p$.
\end{itemize}

\begin{remark}\label{Remark2}
Note that when we consider the classical truth for a Kripke model 
$\kcal=(K,\sqsubset,\preccurlyeq,V)$,  we are ignoring the $\preccurlyeq$
from $\kcal$ and it would collapse to the well known  Kripke semantic for the
classical modal logic $\kcal_c:=(K,\sqsubseteq,V)$. The same argument holds 
for the classical $I$-truth,
except for the valuation $V$, which should be modified according to $I$, more precisely, $\kcal,\alpha,I\models_c A$ iff $\kcal^I_c,\alpha\models A$, in which 
$\kcal^I_c:=(K,\sqsubset,V^I_\alpha)$ is a classical Kripke semantic for classical modal logic with 
 $$\beta \,V^I_\alpha\, p\quad \Leftrightarrow \quad (\beta\neq \alpha \wedge \beta \,V\, p) \vee 
(\beta=\alpha \wedge I\models p) 
 $$
\end{remark}
In the rest of paper, we may simply  write $\alpha\Vdash A$ 
for $\kcal,\alpha\Vdash A$,
 if no confusion is likely.
By an induction on the complexity of $A$, one can observe that $\alpha\Vdash A$
implies $\beta\Vdash A$ for all $A$ and $\alpha\hlf{\preccurlyeq} \beta$.
 We define the
following notions.
\begin{itemize}
\item If $\alpha\hlf{\preccurlyeq}\beta$, $\beta$ is  \emph{above} $\alpha$ and
$\alpha$ is \emph{beneath} $\beta$. If  $\alpha\;\R\;\beta$, $\beta $ is
 a \emph{successor} of $\alpha$. We say that $\beta$ is an immediate successor  of $\alpha$, if $\alpha\R\beta$ and 
 there is no $\gamma$ such that $\alpha\R\gamma\R\beta$. 
\item We say that $\alpha$ is $\R$-branching, if 
the set of immediate successors of $\alpha$ is not singleton. 
\item A Kripke model is finite if its set of nodes is finite.
\item $(\alpha\!\!\R)$ indicates the set of successors of $\alpha$, and 
 $(\alpha\!\!\prec)$  and $(\alpha\!\!\preccurlyeq)$ are defined
similarly.
\item   $\alpha$ is classical, if $(\alpha\!\!\prec) =\emptyset$. 
\item  $\alpha$ is quasi-classical, if 
$(\alpha\!\!\prec)= (\alpha\!\!\R)$. 
\item $\alpha$ is complete if 
$(\alpha\!\!\R)\subseteq (\alpha\!\!\prec)$. Also we say that 
$\alpha$ is atom-complete if $\alpha\Vdash p$ and $\alpha\R\beta$ 
implies $\beta\Vdash p$, for every atomic variable $p$.
\item Let $\varphi$ indicates some property for nodes 
in $\kcal$ and $X\subseteq K$. 
We say  that $\kcal$ is $X$-$\varphi$, if every 
 $\alpha\in X$ has the property $\varphi$.  If $X=\{\alpha\}$, we may use $\alpha$-$\varphi$ instead.
We say that $\kcal$ has the property $\varphi$, 
or simply ``is $\varphi$", if  it is $K$-$\varphi$.
 For example if we set 
 ${\sf Suc}:=\bigcup_{\alpha\in K}(\alpha\!\!\R)$,
${\sf Suc}$-classical means that every $\R$-accessible node is classical.
\item $\mathcal{K}$ is called {\em neat} iff  $\alpha\R\gamma$ and 
\hl{$\alpha\preccurlyeq\beta \preccurlyeq \gamma$} implies $\alpha\R\beta$
or $\beta\R\gamma$.
\item  $\mathcal{K}$ is called {\em brilliant} iff 
$(\R\hlf{;\, } \hlf{\preccurlyeq})\subseteq\,\R$ \cite{IemhoffT}. Note that $ \alpha \R\hlf{;\, } \hlf{\preccurlyeq} \ \beta  $ iff there is some $ \delta$ such that   $ \alpha \R \delta \preccurlyeq \beta $.
\item We say that $\kcal$ has tree frame, if 
 $(K,\prec\cup\R)$ is tree. 
 A tree is a partial order  $(X,<)$ such that for every $x\in X$, the set 
 $\{y\in X: y\leq x\}$ is finite linearly ordered.
\item  $\mathcal{K}$ is called {\em semi-perfect} iff it is
(1) with finite tree frame, 
(2) brilliant, (3) neat 
and (4) $\R$ is irreflexive and transitive.  We say that $\kcal$ is perfect if it is semi-perfect and complete. Note that every quasi-classical Kripke model with finite tree frame is perfect. 
\item We say that a Kripke model $\kcal$ is
 $A$-sound at $\alpha$ ($\alpha$ is $A$-sound), if   for every 
boxed subformula $\Box B$ of $A$ we have 
$\kcal,\alpha\models \Box B\to B$.  
\item Suppose $X$ is a set of propositions that is closed under
sub-formulae \hl{(we call such $X$ \emph{adequate})}. An $X$-saturated set of
propositions $\Gamma$ with respect to some \hl{logic $L$} is a
\hl{consistent} subset
of $X$ \hl{such} that
\begin{itemize}
	 \item For each $A\in X$, \hl{$\Gamma\vdash_L A$} implies $A\in\Gamma$.
 \item For each $A\vee B\in X $, \hl{$\Gamma\vdash_L A\vee B$} implies $A\in\Gamma$ or $B\in \Gamma$.
\end{itemize}
\end{itemize}

\begin{lemma}\label{Lemma-saturation}
Let \hl{$\nvdash_L A$} and let $X$ be an adequate set. Then there is an $X$-saturated  set $\Gamma$ such that 
\hl{$\Gamma\nvdash A$}.
\end{lemma}
\begin{proof}
See \cite{IemhoffT}.
\end{proof}

\begin{theorem}\label{Theorem-Propositional Completeness LC}
$\iglc$  is sound and complete for  perfect Kripke
models. Also $\iglct$ is sound and complete for perfect quasi-classical Kripke models.
\end{theorem}
\begin{proof}
See \cite[Theorem 4.26]{Sigma.Prov.HA}   for $\iglc$ and 
\cite[Lemma 6.14]{Visser82} for $\iglct$.
\end{proof}
\noindent Since $\iglc$ and $\iglct$ have finite model property, as it is expected, we can easily deduce the decidability of $\iglc$ and $\iglct$:

\begin{corollary}\label{Corollary-decidability of iglc}
$\iglc$ and $\iglct$ are decidable.
\end{corollary}
\begin{proof}
For $\iglc$ see \cite[Corollary 4.27]{Sigma.Prov.HA}. $\iglct$ is similar and left to the reader.
\end{proof}

\begin{lemma}\label{Lem-truth-forcing}
Let $A$ be a modal proposition and 
 $\kcal=(K,\preccurlyeq,\R, V)$ be  a semi-perfect  Kripke model.
Then for every quasi-classical node $\alpha\in K$ we have 
$$
\kcal,\alpha\Vdash A^\Box \quad  \Longleftrightarrow\quad  
\kcal,\alpha\models A^{{\Box}} 
$$
\end{lemma}
\begin{proof}
We  use    induction 
on the complexity of $A$. 
The only non-trivial case is when $A=B\to C$.
Let  $\kcal,\alpha\nVdash \Boxdot (B^\Box\to C^\Box)$. If 
$\kcal,\alpha\nVdash \Box (B^\Box\to C^\Box)$
then  evidently 
$\kcal,\alpha\not\models \Box (B^\Box\to C^\Box)$ 
and we are done. If 
$\kcal,\alpha\nVdash B^\Box\to C^\Box$, then 
there exists some $\beta\succcurlyeq \alpha$ such that 
$\kcal,\beta\Vdash B^\Box$ and 
$\kcal,\beta\nVdash C^\Box$. 
Since $\alpha$ is quasi classical, hence 
$\beta \sqsupset \alpha$ or
 $\beta=\alpha$.  If $\beta\sqsupset \alpha$,
we have $\kcal,\alpha\not\models \Box (B^\Box\to C^\Box)$ and 
we are done. Otherwise, $\kcal,\alpha\Vdash B^\Box$ and 
$\kcal,\alpha\nVdash C^\Box$ and hence by induction hypothesis 
we have $\kcal,\alpha\models B^\Box$ and 
$\kcal,\alpha\not\models C^\Box$  and we are done.
For the other way around, let 
$\kcal,\alpha\not\models \Boxdot (B^\Box\to C^\Box)$. If 
$\kcal,\alpha\not\models \Box (B^\Box\to C^\Box)$, evidently we 
have $\kcal,\alpha\nVdash \Box (B^\Box\to C^\Box)$ 
and we are done.  Otherwise, let $\kcal,\alpha\not
\models B^\Box\to C^\Box$. Then $\kcal,\alpha\models B^\Box$
and $\kcal,\alpha\not\models C^\Box$. 
Induction hypothesis  implies $\kcal,\alpha\Vdash B^\Box$ and
$\kcal,\alpha\nVdash C^\Box$ and hence 
$\kcal,\alpha\nVdash \Boxdot (B^\Box\to C^\Box)$.
\end{proof}

\begin{corollary}\label{Corol-truth-forcing}
Let $A$ be a modal proposition and 
 $\kcal$ is  a semi-perfect  quasi-classical 
 Kripke model.
Then for every node  $\alpha$ we have 
$$
\kcal,\alpha\Vdash A^{{\Box}} \quad \Longleftrightarrow 
\quad \kcal,\alpha\models A^{{\Box}} \quad \Longleftrightarrow 
\quad \kcal,\alpha\models_c A^{{\Box}}
$$
$$
\kcal,\alpha\models A^{{\Boxin}}  \Longleftrightarrow 
 \kcal,\alpha\models_c A^{{\Boxin}} \quad \text{ and } \quad 
 \kcal,\alpha,I\models A^{{\Boxin}}  \Longleftrightarrow 
 \kcal,\alpha,I\models_c A^{{\Boxin}}
$$
\end{corollary}
\begin{proof}
By \Cref{Lem-truth-forcing}, for every node $\alpha$ we have 
$\kcal,\alpha\Vdash A^\Box$ iff $\kcal,\alpha\models A^\Box$. 
One can easily observe by induction on the height of the 
node $\alpha\in K$  that $\kcal,\alpha\Vdash A^\Box$ iff 
$\kcal,\alpha\models_c A^\Box$.
\end{proof}

\begin{corollary}\label{Corol-truth-forcing2}
Let $A$ be a modal proposition and 
 $\kcal$ is  a semi-perfect  quasi-classical 
 Kripke model.
Then for every node  $\alpha$ we have 
$$
\kcal,\alpha\Vdash A \quad \Longleftrightarrow 
\quad \kcal,\alpha\models A^{{\Boxout}} 
$$
\end{corollary}
\begin{proof}
Observe that $\kcal,\alpha\models B^\Boxin\lr B$, 
$\kcal,\alpha\Vdash B^\Box\lr B$ and 
$B^\Box=(B^\Boxout)^\Boxin$. Hence  
 by \Cref{Corol-truth-forcing} we have $\kcal,\alpha\Vdash A$ iff 
$\kcal,\alpha\Vdash A^\Box$ iff 
$\kcal,\alpha\models (A^\Boxout)^\Boxin$ iff 
$\kcal,\alpha \models A^\Boxout$. 
\end{proof}

\subsubsection{The Smor\'ynski Operation}\label{Sec-Smorynski-operation}
In this subsection, we define the Smory\'nski operation on Kripke models 
\cite{Smorynski-Book}.
Given a Kripke model $\kcal=(K,\preccurlyeq,\R,V)$ 
and some fixed node $\alpha\in K$, 
 define $\kcal':=(K',\preccurlyeq',\R',V')$ as the Kripke model constituted by adding one fresh node 
$\alpha'$ to $\kcal$. All nodes  of $\kcal'$ other than $\alpha'$, forces
the same atomic variables and have the same accessibility relationships as they did in $\kcal$. 
Also $\alpha'$ imitates all relationships of $\alpha$.  More precisely 
$\kcal'$ is constituted as follows:	
\begin{itemize}
\item $K':=K\cup\{\alpha'\}$, in which $\alpha'\not\in K$,
\item $\beta\preccurlyeq'\gamma$ iff $\beta\preccurlyeq \gamma$ for every $\beta,\gamma\in K$,
\item $\beta\R'\gamma$ iff $\beta\R \gamma$ for every $\beta,\gamma\in K$,
\item $\beta\; V'\, p$ iff $\beta\; V\, p$ for every $\beta\in K$,
\item $\alpha' \;V'\, p$ iff $\alpha \;V\, p$,
\item $\alpha'\preccurlyeq' \beta$ iff ($\alpha\preccurlyeq \beta$ or $\beta=\alpha'$).  Also $\beta\preccurlyeq'\alpha'$ iff $\beta=\alpha$,
\item $\alpha'\R'\beta$ iff $\alpha\sqsubseteq\beta$. 
Also $\beta\not\R'\alpha'$ for every $\beta\in K'$.
\end{itemize} 
Then we define $\kcal^{(n)}$ and $\alpha_n$ inductively:
\begin{itemize}
\item $\kcal^{(0)}:=\kcal$ and $\alpha_0:=\alpha$,
\item $\kcal^{(n+1)}:= \left(\kcal^{(n)}\right)'$ and $\alpha_{n+1}$ is defined as the fresh node which is added to $\kcal^{(n)}$ in the definition of $\left(\kcal^{(n)}\right)'$.
\end{itemize}
\begin{lemma}\label{Lemma-10}
Let $\kcal$ be a Kripke model which is  $A^\Boxin$-sound at
the quasi-classical node  $\alpha$. Then 
 for every 
subformula $B$ of $A^\Boxin$ and arbitrary 
boolean interpretation $I$ we have 
\begin{enumerate}
\item $\kcal,\alpha\models B$ iff $\kcal',\alpha'\models B$.
\item $\kcal,\alpha,I\models B$ iff $\kcal',\alpha',I\models B$.
\item $\alpha'$ is quasi-classical  and 
$\kcal'$ is  $A^\Boxin$-sound at 
$\alpha'$. 
\item If $\kcal$ is semi-perfect, perfect or quasi-classical, then $\kcal'$  
 is so.
\end{enumerate}
 \end{lemma}
\begin{proof}
\begin{enumerate}[leftmargin=*]
\item Use induction on the complexity of $B$. 
All cases are trivial, except for 
the case $B=\Box C^\Box$. If $\alpha'\models \Box C^\Box$, 
evidently $\alpha\models \Box C^\Box$ as well. 
If $\alpha\models \Box C^\Box$, then by $A^\Boxin$-soundness,
$\alpha \models C^\Box$, and by 
\Cref{Lem-truth-forcing}, $\alpha\Vdash C^\Box$. Hence 
$\alpha'\models \Box C^\Box$.
\item Similar to first item and left to the reader.
\item The fact that $\alpha'$ is quasi-classical can easily be observed by the definition of $\kcal'$ and left to the reader. The $A^\Boxin$-soundness, is derived from first item. 
\item Easy and left to the reader.
\end{enumerate}
\end{proof}

\section{Reduction of Arithmetical Completenesses}

Let us define $\wit{A;\SFT,\SFU;\Gamma}$ 
as the set of all 
$\Gamma$-substitutions $\sigma$ such that 
 $\SFU\nvdash \sigma_{_\SFT}(A)$. Hence $\PLG(\SFT,\SFU)=
 \{A:\wit{A;\SFT,\SFU;\Gamma}=\emptyset\}$. 
 For an arithmetical substitution $\sigma$, 
 let $\wit{\sigma}$ indicate the propositional closure of 
 $\sigma$, i.e. the smallest set $X$ of 
 arithmetical substitutions with the following conditions:
 \begin{itemize}
 \item $\sigma\in X$,
 \item if $\alpha\in X$,  $\tau$ is some 
 $\lcal_\Box$-substitution and $\SFT$ is some 
 recursively axiomatizable 
 arithmetical theory, 
 then $\alpha_{_\SFT}\!\circ\tau\in X$. 
 \end{itemize}
Note that
the substitution $\alpha_{_\SFT}\!\circ\tau$ is defined on atomic variable $p$ in this way: 
$\alpha_{_\SFT}\!\circ\tau(p):=\alpha_{_\SFT}(\tau(p))$.

Let $\SFV$ be a modal theory. We define  the 
$\Gamma$-arithmetical completeness of $\SFV$ with respect to $\SFT$ relative in $\SFU$ as follows:
$$\AC\equiv \text{$A\in\PLG(\SFT,\SFU)$ implies  $\SFV\vdash A$, for every $A\in\lcal_\Box$}$$
Similarly we define the Arithmetical soundness $\AS$ as follows:
$$\AS\equiv \SFV\vdash A \text{ implies $A\in\PLG(\SFT,\SFU)$, for every $A\in\lcal_\Box$} $$
When  $\Gamma$ is the set of 
all arithmetical sentences, we may omit the subscript 
$\Gamma$ in the notations
$\PLG(\SFT,\SFU)$, $\AC$ and $\AS$.
\\
Note that $\PLG(\SFT,\SFU)=\SFV$ iff $\AC$ and $\AS$.

In the following definition, we formalize  
 reduction of the arithmetical completeness of 
$\SFV$ to  $\SFV'$: 

\begin{definition}\label{Definition-Reduction-PL}
Let $\SFT$ and $\SFT'$ be consistent recursively axiomatizable and 
$\SFU$ and $\SFU'$ be strong enough arithmetical theories. 
Also let $\Gamma$ and $\Gamma'$ be sets 
of arithmetical sentences and $\SFV,\SFV'$ be modal theories.
We say that  $f,\bar{f}$ propositionally reduces $\AC$ to 
$\ACP$, with the notation 
$\AC\leq_{f,\bar{f}}\ACP$, 
if:
\begin{enumerate}
\item[R0.] $f:\lcal_\Box\longrightarrow\lcal_\Box$ and $\bar{f}=\{\bar{f}_{_A}\}_{_A}$ is a family of functions,
\item[R1.]  $\SFV'\vdash f(A)$ implies $\SFV\vdash A$,
\item[R2.] for every $A\in\lcal_\Box$, $\bar{f}_{_A}$ is 
a  function on arithmetical substitutions and
		$$\bar{f}_{_A}:\wit{f(A);\SFT',\SFU';\Gamma'}
		\longrightarrow
		\wit{A;\SFT,\SFU;\Gamma}\quad \text{and for every $\alpha$: }
		\quad 
		\bar{f}_{_A}(\sigma)\in\wit{\sigma}.
		$$
\end{enumerate}
We say that $\AC$ is  
reducible   to
$\ACP$, with the notation 
$$\AC\leq\ACP,$$ if there exists 
some $f,\bar{f}$ such that   
$\AC\leq_{f,\bar{f}}\ACP$.  
\end{definition}
\noindent 
 Following theorems  are what one expect from the reduction:

\begin{theorem}\label{Theorem-reduction-transitive}
The reduction of arithmetical completenesses is a transitive reflexive relation. 
\end{theorem}
\begin{proof}
The reflexivity is trivial and left to the reader. For the transitivity, let 
$$\AC\leq_{f,\bar{f}}\ACP\leq_{g,\bar{g}}\ACPP$$
and observe that 
$$\AC\leq_{h,\bar{h}}\ACPP$$
in which $h:=g\circ f$ and $\bar{h}_{_A}:=\bar{f}_{_A}\circ\bar{g}_{_{f(A)}}$.
\end{proof}

\begin{theorem}\label{Theorem-Reduction-1}
 $\AC \leq_{f,\bar{f}}\ACP$ and $\ACP$ implies $\AC$.
\end{theorem}
\begin{proof}
Let 
$\SFV\nvdash A$.
Then by R1 in the \Cref{Definition-Reduction-PL},
 $\SFV'\nvdash f(A)$. 
Hence by  $\ACP$,  
there exists some $\Gamma'$-substitution $\sigma$ such that 
$\SFU'\nvdash\sigma_{_{\SFT'}}(f(A))$, or in other words 
$\sigma\in\wit{f(A);\SFT',\SFU';\Gamma'}$. Hence by R2
$\bar{f}_{_A}(\sigma)\in \wit{A;\SFT,\SFU;\Gamma}$,  
which implies  $A\in\PLG(\SFT,\SFU)$.
\end{proof}

\begin{remark}
 Note that the requirement $\bar{f}_{_A}(\alpha)\in\wit{\alpha}$, did not used in the proof of arithmetical 
completeness of $\SFV$ in \Cref{Theorem-Reduction-1}. The only usage of this condition, is to restrict the way one may compute $\bar{f}_{_A}(\alpha)$ from $\alpha$: only propositional substitutions are allowed to be 
composed by $\alpha$ to produce $\bar{f}_{_A}(\alpha)$. If we remove this restriction from the definition, we would have 
a trivial reduction: every arithmetical completeness would be reducible to everyone.
\end{remark}

\begin{corollary}\label{Corollary-Reduction}
 If  $\AC \leq_{f,\bar{f}}\ACP$ and $\ACP$, then we have 
$$\SFV\vdash A\quad \Longleftrightarrow \quad 
\SFV'\vdash f(A)$$
\end{corollary}
\begin{proof}
The direction $\Longleftarrow$ holds 
by definition. For the other way around, use 
\Cref{Theorem-Reduction-1}.
\end{proof}

\begin{remark}\label{Remark-Reduction-3}
Note that $\SFV\vdash A  \Longleftrightarrow  
\SFV'\vdash f(A)$ is not enough for reduction of the 
arithmetical completenesses. This is simply because $f$   
does not have anything to do with the arithmetical 
substitutions. So one may not be able  to translate  
an arithmetical refutation  from
$\PLGP(\SFT',\SFU')$   to a refutation from 
$\PLG(\SFT,\SFU)$, via propositional translations. 
If we remove the condition 3
and replace second item by 
$\SFV\vdash A  \Longleftrightarrow  
\SFV'\vdash f(A)$ in the 
\Cref{Definition-Reduction-PL},  $\AC$ would be 
reducible to every  arithmetical completeness
via the following vicious reduction:
$$f(A):=\begin{cases}
\top &: \text{ if }\SFV\vdash A\\
\bot &: \text{ otherwise}
\end{cases}$$
\end{remark}

\begin{notation}
In the rest of the paper, we are going to characterize several 
provability logics. Our main tool for proving their 
arithmetical completeness is the reduction of arithmetical completenesses \Cref{Theorem-Reduction-1}. The  
notation $\SFV=\PLG(\SFT,\SFU)\leq \PLGP(\SFT',\SFU')=\SFV'$ means that the following items hold: (1) $\AC\leq\ACP$, (2) $\SFV=\PLG(\SFT,\SFU)$, 
(3) $\PLGP(\SFT',\SFU')=\SFV'$.\\
Let  the provability logics $\PLG(\SFT,\SFU)=\SFV$ and  
$ \PLGP(\SFT',\SFU')=\SFV'$ are already  characterized. 
Then the notation 
 $\PLG(\SFT,\SFU)\leq \PLGP(\SFT',\SFU')$ indicates 
 $\AC\leq\ACP$.   
\end{notation} 

\begin{theorem}\label{Theorem-decidability-Reduction}
Let $\PLG(\SFT,\SFU)\leq_{f,\fbar[]} \PLGP(\SFT',\SFU')$ for some 
computable function  $f$. Then the decidability of $\PLGP(\SFT',\SFU')$ implies the 
decidability of $\PLG(\SFT,\SFU)$.
\end{theorem}
\begin{proof}
Direct consequence of \Cref{Corollary-Reduction} and computability of $f$.
\end{proof}

\subsection{Two special cases}\label{Section-Reduction-tool}
In later applications, always we consider two  simple
 cases of reduction $f,\bar{f}$ 
 (\Cref{Definition-Reduction-PL}) 
 to provide new arithmetical completenesses:
\begin{itemize}
\item Identity: in this case we consider $\bar{f}_{_A}$ as 
identity function and $f(A) $  is  some
propositional translation  like $(.)^\Box$ or $(.)^\Boxin$. 
\item Substitution: in this case, we let $f(A)$ as some
$\lcal_\Box$-substitution, possibly depending on $A$. Also 
$\bar{f}_{_A}(\sigma):=\sigma_{_{\SFT'}}\!\circ\tau$.
\end{itemize}

\section{Relative $\Sigma_1$-provability logics for ${\sf HA}$}
\label{Section-HA}
In this section, we will characterize $\PLS(\HA,\mathbb{N})$, i.e.~the truth 
$\Sigma_1$-provability 
logic of $\HA$,  and $\PLS(\HA,\PA)$, i.e.~the $\Sigma_1$-provability 
logic of $\HA$, relative to $\PA$. We also show that 
$\PLS(\HA,\mathbb{N})$ is hardest among the $\Sigma_1$-provability logics of $\HA$ relative in $\HA,\PA,\nat$. In 
other words: 

\begin{figure}[h]
\[ \begin{tikzcd}[column sep=5em, row sep=4em] 
\tcboxmath{\PLS(\HA,\PA) }
\arrow[r , "(.)^\negout", "\text{\ref{Reduction-Sigma-HA-(PA-to-HA)}}" ']  
&\tcboxmath{\PLS(\HA,\HA) }
\arrow[r , "\Box(.)", "\text{\ref{Reduction-Sigma-HA-(HA-to-nat)}}" '] 
& \tcbhighmath{\PLS(\HA,\nat)}
\end{tikzcd}
\]
\caption[Diagram of  reductions for relative provability logics of $\HA$]{\label{Diagram-HA} Reductions for  relative provability logics of $\HA$}
\end{figure}

\subsection{Kripke Semantic}

\begin{lemma}\label{Lem-iglc-subformule}
For every $A$ we have 
\begin{equation*}
\iglc\vdash A  \quad \Longleftrightarrow \quad \igl\vdash 
\left[\Boxdot\bigwedge_{E\in {\sf sub}(A)} (E\to \Box E)\right]
\to A
\end{equation*}
\end{lemma}
\begin{proof}
For the simplicity of notations, in this proof,
 let $$\varphi:=
\Boxdot\bigwedge_{E\in {\sf sub}(A)} (E\to \Box E)$$ and 
$\vdash$ indicates derivablity in $\iGL+\varphi$.\\
One side is trivial. For the other way around, assume that 
$ \igl\nvdash \varphi \to A$. 
We will construct some  perfect Kripke model 
$\kcal=(K,\R,\preccurlyeq,V)$  
such that 
$\kcal,\alpha\nVdash A$, 
which by soundness of $\iglc$ for finite brilliant models with 
$\R\subseteq \prec$,
we have 
the desired result. The proof is almost identical to the proof 
of \Cref{Theorem-Propositional Completeness LC} in 
\cite[Theorem 4.26]{Sigma.Prov.HA}, but to be self-contained, 
we repeat it here.

Let ${\sf Sub}(A)$ be the set of sub-formulae of
$A$. Then define  $$X:=\{B,\Box B \ | \ B\in {\sf Sub}(A)\}$$ It is
obvious that $X$ is a finite adequate set. We define
$\mathcal{K}=(K,\hlf{\preccurlyeq},\R,V)$ as follows. Take $K$ as the set of
all $X$-saturated sets with respect to $\igl+\varphi$, 
and $\hlf{\preccurlyeq}$ is the
subset relation over $K$. Define $\alpha\R \beta$ iff for all $\Box
B\in X$, $ \Box B\in \alpha$ implies  $B\in \beta$, and also there
exists some $\Box C\in \beta\setminus \alpha$. Finally define $\alpha V
p$ iff $p\in \alpha$,  for atomic $p$. \\
It only remains to show that 
$\mathcal{K}$ is a finite brilliant
Kripke model with $\R\subseteq \prec$ which refutes $A$. 
To this end, we first show by
induction on $B\in X$ that $B\in \alpha$ iff 
$\alpha\Vdash B$, for each
$\alpha\in K$. The only non-trivial case is $B=\Box C$. Let
$\Box C\not\in \alpha$. 
We must show $\alpha\nVdash \Box C$. The other
direction is easier to prove and we leave it to reader. Let
$\beta_0:=\{D\in X\ |\ \alpha\vdash\Box D\}$. 
If $\beta_0,\Box C\vdash
C$, since by definition of $\beta_0$, 
we have $\alpha\vdash\Box \beta_0$ and
hence by L\"{o}b's axiom,  $\alpha\vdash \Box C$, which is in contradiction with
$\Box C\not\in \alpha$. 
Hence $\beta_0,\Box C\nvdash C$ and so there
exists some $X$-saturated set $\beta$ such that 
$\beta\nvdash C$,
$\beta\supseteq \beta_0\cup\{\Box C\}$. 
Hence $\beta\in K$ and $\alpha\R \beta$.
Then by \hl{the induction hypothesis}, 
$\beta\nVdash C$ and hence $\alpha\nVdash
\Box C$.

Since $\igl+\varphi\nvdash A$, by \Cref{Lemma-saturation}, there exists some
$X$-saturated set $\alpha\in K$ such that $\alpha\nvdash A$, and hence by
the above argument we have $\alpha\nVdash A$.

$\mathcal{K}$ trivially satisfies all the properties of 
finite brilliant 
Kripke model with $\R\,\subseteq\,\prec$.
 As a sample, we show that why $\R\,\subseteq\, \prec$ holds.
Assume $\alpha\R \beta$ and let $B\in \alpha$. If $B=\Box C$ for some $C$,
then by definition, $C\in \beta$ and since $C\to\Box C$ is a conjunct in $\varphi$, we have 
$\beta\vdash \Box C$ and we are
done. So assume  $B$ is not a boxed formula. Then by definition
of $X$, we have $\Box B\in X$ and  since $B\to \Box B$ is a 
conjunct in $\varphi$, 
we have  $\alpha\vdash\Box B$ and hence by definition of
$\R$, it is the case that $B\in \beta$. This shows $\alpha\subseteq \beta$ and
hence $\hlf{\alpha\preccurlyeq \beta}$. But $\alpha$ is not equal to $\beta$, because $\alpha\R \beta$
implies existence of some $\Box C\in \beta\setminus \alpha$. Hence $\hlf{\alpha\prec \beta}$, as desired.
\end{proof}

\begin{lemma}\label{Lemma-2}
For any  proposition $A$, if 
$\iglc\vdash A^\Box$ then
${\sf iGL} + {\sf CP_a}+ \Box {\sf CP}\vdash A^\Box$.
\end{lemma}
\begin{proof}
Let $\iglc\vdash A$. 
Hence by \Cref{Lem-iglc-subformule}
for some finite set $X$ of
subformulas of $A^\Box$ we have 
$$\iGL\vdash \Boxdot\left(\bigwedge_{E\in X} E\to \Box E\right)\to A^\Box$$
 \Cref{Lemma-3} implies 
 $$\iGL\vdash \Box\left(\bigwedge_{E\in X} E\to \Box E\right)\wedge\left(\bigwedge_{E\in X} \Boxdot(E^\Boxout\to \Box E)\right)\to A^\Box$$
By  \Cref{Lemma-4-2} we have  $\iGL + {\sf CP_a}+\Box\CP\vdash A^\Box$. 
\end{proof}

\begin{theorem}\label{Theorem-Propositional Completeness ihatglchat}
$\ihatglchat$ is sound and 
complete for local truth at
quasi-classical nodes in perfect Kripke
models. More precisely,  we have 
$\ihatglchat\vdash A$ iff $\kcal,\alpha\models A$ for every  
perfect Kripke model $\kcal$ and the quasi-classical node $\alpha$.
\end{theorem}
\begin{proof}
The soundness part easily derived by the soundness of 
$\iglc$ and left to the reader. \\
Since local truth at $\alpha$ is not affected by changing the
 set of $\preccurlyeq$-accessible nodes from $\alpha$, it is enough to prove the completeness part only for the perfect Kripke models.
Let
$\ihatglchat\nvdash A$.  Let $A'$ be 
a boolean equivalent of $A$ which is a conjunction of implications 
$E\to F $
in which $E$ is a conjunction of a set of atomics or boxed propositions 
and $F$ is a disjunction of atomics or boxed proposition. 
Evidently such $A'$ exists for every $A$.
Hence $\ihatglchat\nvdash A'$. Then there must be some 
conjunct $E\to F$ of 
$A'$ such that $\ihatglchat\nvdash (E\to F)^\Box$, $E$ is a conjunction of atomic and boxed propositions
 and $F$ is a disjunction of atomic and  
boxed propositions. 
Hence $\iGL+{\sf CP_a}+\Box\CP\nvdash (E\to F)^\Box$
 and by \Cref{Lemma-2}
we have $\iglc\nvdash (E\to F)^\Box$. 
By 
\Cref{Theorem-Propositional Completeness LC}, 
there exists some 
 perfect Kripke model  
$\kcal=(K,\preccurlyeq,\R,V)$ 
 such that $\kcal,\alpha\nVdash (E\to F)^\Box$ 
for some $\alpha\in K$. Since $\iglc$ is sound for $\kcal$, we 
have $\kcal,\alpha\nVdash E\to F$. 
 Hence there exists some 
$\beta\succcurlyeq \alpha$ such that $\kcal,\beta\Vdash E$ and $\kcal,\beta\nVdash F$. Then by definition of local truth we have $\kcal,\beta\models E$ and $\kcal,\beta\not\models F$. Then 
$\kcal,\beta\not\models E\to F$. Hence $\kcal,\beta\not\models A$, as desired.
\end{proof}
\begin{corollary}\label{Corollary-decidability-ihatglchat}
$\ihatglchat$ is decidable.
\end{corollary}
\begin{proof}
Direct consequence of the 
 proof of \Cref{Theorem-Propositional Completeness ihatglchat} and decidability of 
 $\iglc$ (\Cref{Corollary-decidability of iglc}).
\end{proof}

\begin{theorem}\label{theorem-Kripke-completeness-ihatglchats}
$\ihatglchats\vdash A^\Boxin$ iff 
$\kcal,\alpha\models A^\Boxin$ for every 
perfect Kripke
models $\kcal$ and quasi-classical $A^\Boxin$-sound 
nodes $\alpha$.
\end{theorem}
\begin{proof}
Both directions are non-trivial and proved contra-positively. 
For the soundness part, assume that 
$\kcal,\alpha\not\models A^\Boxin$ for some 
 perfect Kripke
model $\kcal:=(K,\preccurlyeq,\R,V)$  
which  is  $A^\Boxin$-sound at the quasi-classical node $\alpha\in K$.
Since derivability is finite,  
it is enough to show that for every finite set 
$\Gamma$ of modal propositions we have 
$$\ihatglchat\nvdash \bigwedge_{B\in \Gamma}(\Box B\to B)\to A^\Boxin.$$
By \Cref{Theorem-Propositional Completeness ihatglchat} 
and \Cref{Lemma-10}, it is enough to
find some  number $i$ such that 
$$ \kcal^{(i)},\alpha_i\not\models \bigwedge_{B\in \Gamma}(\Box B\to B)\to A^\Boxin.$$
Let us define $n_i$ and 
$m_i$
as the number of
 propositions in the sets
 $N_i:=\{B\in \Gamma: \kcal^{(i)},\alpha_i\models B\wedge\Box B\}$ and 
 $M_i:=\{B\in \Gamma: 
 \kcal^{(i)},\alpha_i \models \Box B\wedge \neg B\}$, respectively.\\
We use induction as follows. As induction hypothesis,
assume that for any number $i$ with  $n_i<k$ there is some
$0\leq j\leq 1+n_i$ such that
\begin{equation}\label{eq-1}
\kcal^{(i+j)},\alpha_{i+j}\not\models 
\bigwedge_{B\in \Gamma}(\Box B\to B)\to A^\Boxin
\end{equation}
Let $n_i=k$. If $m_i=0$, we may let $j=0$ and by 
\Cref{Lemma-10} we have \cref{eq-1} as desired. 
So let 
$B\in \Gamma$ such that 
$\kcal^{(i)},\alpha_i\models \Box B\wedge \neg B$. We have two sub-cases:
\begin{itemize}
\item $m_{i+1}=0$: observe in this case that \cref{eq-1} holds for $j=1$. 
\item $m_{i+1}>0$: in this case we have $n_{i+1}<k$ and hence by 
application of the induction hypothesis with $i:=i+1$, we get some 
$0\leq j' \leq 1+n_{i+1}$ such that 
$$\kcal^{(i+1+j')},\alpha_{i+1+j'}\not\models 
\bigwedge_{B\in \Gamma}(\Box B\to B)\to A^\Boxin$$
Hence if we let $j:=j'+1$ we have $0\leq j\leq 1+ n_i$ and 
\cref{eq-1}, as desired.
\end{itemize}
For the completeness part, assume that $\ihatglchats\nvdash A^\Boxin$. Hence 
$$\ihatglchat\nvdash 
\left(\bigwedge_{\Box B\in{\sf Sub}(A)} (\Box B\to B)\right)\to A^\Boxin$$
Hence  \Cref{Theorem-Propositional Completeness ihatglchat}
implies the desired result.
 \end{proof}

\begin{corollary}\label{Corollary-Decidability-ihatglchats}
$\ihatglchats$ is decidable.
\end{corollary}
\begin{proof}
First observe that by \Cref{theorem-Kripke-completeness-ihatglchats,Theorem-Propositional Completeness ihatglchat},  we have 
$\ihatglchats\vdash A^\Boxin$ iff 
$$\ihatglchat\vdash \bigwedge_{\Box B\in{\sf Sub}(A^\Boxin)}(\Box B\to B)\to A^\Boxin.$$
Hence the decidability of $\ihatglchat$ (\Cref{Corollary-decidability-ihatglchat}) implies the decidability of $\ihatglchats$.
\end{proof}

\subsection{Arithmetical interpretations}
\label{Section-arithmetical-interpretations}

The following theorem is the main result in \cite{Sigma.Prov.HA}:
\begin{theorem}\label{Theorem-sigma-provability-HA}
$\lles$ is the $\Sigma_1$-provability logic of $\HA$, i.e.~$\lles\vdash A$ iff for all
$\Sigma_1$-substitution $\sigma$ we have $\HA\vdash \sigma_\tinysub{\sf HA}(A)$. Moreover $\lles$ is decidable.
\end{theorem}

 Here we bring some essential facts
and definitions from \cite{Sigma.Prov.HA}. 
Let us fix some  perfect   Kripke model 
$\kcal_0=(K_0,\R_{0},\hlf{\preccurlyeq}_0,V_0)$ with   the quasi-classical 
root $\alpha_0$ and its extension 
$\kcal:=\kcal_0'=(K,\preccurlyeq,\R,V)$ 
by the Smory\'nski operation with the new quasi-classical  root $\alpha_1$ 
(which was called $\alpha_0$ in \cite{Sigma.Prov.HA}) 
and define a recursive function $F$, called Solovay function, as
we did in \cite{Sigma.Prov.HA}. 
We have  the following definitions and facts from \cite{Sigma.Prov.HA}:
(later we refer to them simply as e.g.~``item 1'',)
\begin{enumerate}
\item The function $F$ is provably total in  $\HA$ and hence we may use the function symbol $F$ inside $\HA$ and stronger theories.
\item The $\Sigma_1$-substitution $\sigma$ is defined in this way:
$$\sigma(p):=\bigvee_{\kcal,\alpha\Vdash p}\exists x (F(x)=\alpha)$$
\item Define $L=\alpha$   as $\exists x \forall y\geq x F(y)=\alpha$.
\item $\PA\vdash \exists x\, (F(x)=\alpha)\to 
\bigvee_{\beta\succcurlyeq \alpha} L=\beta$. \cite[Lemma 5.2]{Sigma.Prov.HA} 
\item For  a modal proposition $A$ when we use 
$A$ in a context which it is expected to be some first-order formula, like 
$\HA\vdash A$, we should replace $A$ with the first-order sentence 
$\sigma_\tinysub{\sf HA}(A)$. 
\item For every $ A\in {\sf sub}(\Gamma)\cap \TNNIL$ and $\alpha\in K_0$ 
such that $\kcal_0,\alpha\Vdash A$, we have $\HA\vdash \exists x\, {F(x)=\alpha}\to A$. \cite[Lemma 5.18 \& 5.19]{Sigma.Prov.HA} 
\item For each $B\in {\sf Sub}(\Gamma)\cap\TNNIL$ and $\alpha\in K$ such that $\alpha\nVdash
	\Box B$,
	$$\HA\vdash L\!=\!\alpha\ra\neg\,\Box B.$$
	\item $\mathbb{N}\models L=\alpha_1$ and 
	$\PA+L=\alpha$ is consistent for every $\alpha\in K$. 
	\cite[Corollaries 5.20 \& 5.24 and Lemma 5.23]{Sigma.Prov.HA} 
\end{enumerate}
\begin{lemma}\label{Lem-30}
For every $A\in {\sf NOI}\cap {\sf sub}(\Gamma)$ such that 
$\kcal,\alpha_1\nVdash A$, we have 
$$\HA\vdash A\leftrightarrow \bigvee_{\alpha\in K\text{ and }
\kcal,\alpha\Vdash A}
 \exists x\, F(x)=\alpha$$
\end{lemma}
\begin{proof}
First observe that by $\Pi_2$-conservativity of $\PA$ over $\HA$ (\Cref{Lemma-Conservativity of HA}), 
it is enough to prove this lemma in $\PA$ instead of $\HA$. Then  by 
``item 4", it is enough to show that 
$$\PA\vdash A\leftrightarrow \bigvee_{\alpha\in K\text{ and }
\kcal,\alpha\Vdash A}
 L=\alpha$$
We use induction on the complexity of $A$. Since $A\in{\sf NOI}$ we do not consider the $\to$ case in the induction steps:
\begin{itemize}
\item $A$ is atomic: by definition of the arithmetical substitution, 
we have $$\sigma(A)=\bigvee_{\alpha\in K\text{ and }
\kcal,\alpha\Vdash A}
 \exists x\, F(x)=\alpha.$$
 \item $A=B\circ C$ and $\circ\in\{\vee,\wedge\}$: easy and left to the reader.
 \item $A=\Box B$: first note that by ``item 6",
 $\PA\vdash \exists x\, F(x)=\alpha\to A$
 for every $\alpha\Vdash A$ (here actually we need $\alpha_1\nVdash A$). Hence 
 $$\PA\vdash \bigvee_{\alpha\in K\text{ and }
\kcal,\alpha\Vdash A}
 \exists x\, F(x)=\alpha\to A$$
 For the other direction, it is enough  (by ``item 4")
 to show that for every 
 $\beta\in K $ such that $\kcal,\beta \nVdash A$ we have 
 $\PA\vdash A\to L\neq \beta$ or equivalently 
 $\PA\vdash L=\beta\to \neg A$, which holds by ``item 7".\qedhere
\end{itemize}
\end{proof}
\begin{lemma}\label{Lem-31}
For every $A\in{\sf sub}(\Gamma)$ and $\alpha\in K_0$, we have
$$\begin{cases}
 \kcal,\alpha\models A
 \quad &\Longrightarrow\quad \HA\vdash L=\alpha\to A\\
 \kcal,\alpha\not\models A  \quad &\Longrightarrow\quad 
\HA\vdash L=\alpha\to \neg A 
\end{cases}$$
\end{lemma}
\begin{proof}
We use induction on the complexity of $A$. All cases are simple and we only treat the case 
$A=\Box B$ here. If $\kcal,\alpha\models \Box B$, by definition, $\kcal,\alpha\Vdash \Box B$ and hence by ``item 6" we have the desired result.
If also $\kcal,\alpha\not\models \Box B$, by definition, $\kcal,\alpha\nVdash \Box B$ and hence by ``item 7" we have the desired result.
\end{proof}

\begin{lemma}\label{Lem-32}
Let $\kcal$ be $A^\Boxin$-sound at $\alpha_0$  for some 
$A^\Boxin\in{\sf sub}(\Gamma)$. 
Then for every $B\in{\sf sub}(A^\Boxin)$ we have 
$$\begin{cases}
 \kcal,\alpha_1\models B
 \quad &\Longrightarrow\quad \HA\vdash L=\alpha_1\to B\\
 \kcal,\alpha_1\not\models B  \quad &\Longrightarrow\quad 
\HA\vdash L=\alpha_1\to \neg B \\
\kcal,\alpha_1\Vdash B
 \quad &\Longrightarrow\quad \HA\vdash  B
\end{cases}$$
\end{lemma}
\begin{proof}
We prove this by induction on the complexity of $B\in{\sf sub}(A^\Boxin)$. 
\begin{itemize}
\item $B$ is atomic, conjunction or disjunction: easy and left to the reader.
\item $B=E\to F$: it is easy to show
the first two derivations 
 and 
we leave them to the reader. For the third one, assume that 
$\kcal,\alpha_1\Vdash E\to F$. If $\kcal,\alpha_1\Vdash F$ we have the desired result by induction hypothesis. So let 
$\kcal,\alpha_1\nVdash F$ and 	 Hence $\kcal,\alpha_1\nVdash E$. Hence 
by \Cref{Lem-30}, we have $\HA\vdash E\to \bigvee_{\alpha\Vdash E} \exists x\, F(x)=\alpha$. On the other hand by ``item 6"  we have 
$\HA\vdash \bigvee_{\alpha\Vdash E} \exists x\, F(x)=\alpha\to F$.
Hence we have $\HA\to E\to F$.
\item $B=\Box C^\Box$: Let $\kcal,\alpha_1\models \Box C^\Box$. Then 
by \Cref{Lemma-10} we have $\kcal,\alpha_1\models C^\Box$ and hence
by \Cref{Lem-truth-forcing}  $\kcal,\alpha_1\Vdash C^\Box$.
 Then by induction hypothesis $\HA\vdash C^\Box$ and hence $\HA\vdash L=\alpha_1\to \Box C^\Box$. \\
For the second derivation,  Let $\kcal,\alpha_1\not\models \Box C^\Box$. 
Then by ``item 7" we have the desired result. \\
For the third derivation, let $\kcal,\alpha_1\Vdash \Box C^\Box$. Then 
by \Cref{Lem-truth-forcing} we have $\kcal,\alpha_1\models \Box C^\Box$
and hence 
 \Cref{Lemma-10} implies $\kcal,\alpha_1\models C^\Box$ and then again 
 by\Cref{Lem-truth-forcing}  $\kcal,\alpha_1\Vdash C^\Box$.  Then by 
 induction hypothesis $\HA\vdash C^\Box$ and hence 
 $\HA\vdash \Box C^\Box$.
\end{itemize}
\end{proof}
\subsection{Arithmetical Completeness}
\begin{definition}\label{Def-truth-prov-logic}
Define the following modal systems:
\begin{itemize}
\item $\ihatHsigma:=\lles$ plus $\underline{\sf P}$,
\item $\ihatHSsigma:=\ihatHsigma$ plus $\underline{\sf S}$,
\item $\ihatHsigmastar:=\{A\in\lcalb : 
\ihatHsigma\vdash A^\Boxin\}$,
\item $\ihatHSsigmastar:=\{ A\in \lcalb : 
\ihatHSsigma\vdash A^\Boxin\}$.
\end{itemize}
Obviously $\ihatHSsigmastar$ and $\ihatHsigmastar$ are closed
under modus ponens.
\end{definition}

\begin{theorem}\label{Theorem-PA-relative}
$\ihatHsigma=\PLS(\HA,\PA)$, i.e.~$\ihatHsigma$ 
is the relative $\Sigma_1$-provability logic of $\HA$ 
in $\PA$.
\end{theorem}
\begin{proof}
The soundness easily deduced by use of the  soundness of the 
$\lles$ for arithmetical interpretations in $\HA$ (see Theorem 6.3 in 
\cite{Sigma.Prov.HA}). 

For the other way around, 
let $\ihatHsigma\nvdash A$. 
Then $\ihatHsigma\nvdash A^- $ in which 
$ A^- \in\TNNIL^\Box$ and 
$\lles\vdash A\leftrightarrow  A^- $.
Then  $\ihatglleplus\nvdash A^-$ and 
hence by  \Cref{theorem-conservativiity-ihatglchat} 
we have $\ihatglchat\nvdash  A^- $.
By \Cref{Theorem-Propositional Completeness ihatglchat}, there is some 
  perfect Kripke model $\kcal_0$ with the quasi-classical root $\alpha_0$ 
  such that $\kcal_0,\alpha_0\not\models  A^- $. Let 
  $\sigma$ be the $\Sigma_1$-substitution as provided in 
  \Cref{Section-arithmetical-interpretations} for the Kripke model $\kcal_0$
  and its Smory\'nski extension $\kcal$ with $\Gamma:=\{A^-\}$.
   Then by \Cref{Lem-31} we have 
   $\HA\vdash L=\alpha_0\to \sigma_\tinysub{\sf HA}(\neg A^-)$. Since 
   $\lles\vdash A\leftrightarrow  A^- $, by soundness part of 
   \Cref{Theorem-sigma-provability-HA} we have 
   $\HA\vdash L=\alpha_0\to \sigma_\tinysub{\sf HA}(\neg A)$. Hence 
   by ``item 8"
    we may deduce $\PA\nvdash \sigma_\tinysub{\sf HA}(A)$, as desired.
\end{proof}

\begin{theorem}\label{Theorem-truth-provability}
$\ihatHSsigma=\PLS(\HA,\mathbb{N})$, i.e.~$\ihatHSsigma$ 
is the truth $\Sigma_1$-provability logic of $\HA$.
\end{theorem}
\begin{proof}
The soundness easily deduced by use of the  soundness of the 
$\lles$ for arithmetical interpretations in $\HA$ (see Theorem 6.3 in 
\cite{Sigma.Prov.HA}). 

For the other way around, 
let $\ihatHSsigma\nvdash A$. 
Then by \Cref{Lem-Boxin-validity} we have 
$\ihatHSsigma\nvdash (A^-)^\Boxin $ in which 
$ (A^-)^\Boxin\in\TNNIL^\Box$ and 
$\lles\vdash A\leftrightarrow  (A^-)^\Boxin $.
Then  $\ihatgllepluss\nvdash  (A^-)^\Boxin$ and 
hence by  \Cref{theorem-conservativiity-ihatglchat} 
we have $\ihatglchats\nvdash  (A^-)^\Boxin $.

By \Cref{theorem-Kripke-completeness-ihatglchats}, there is some 
  perfect Kripke model $\kcal_0$ with the quasi-classical  
  $(A^-)^\Boxin$-sound root $\alpha_0$ 
  such that $\kcal_0,\alpha_0\not\models  (A^-)^\Boxin$. Let 
  $\sigma$ be the $\Sigma_1$-substitution as provided in 
  \Cref{Section-arithmetical-interpretations} for the Kripke model $\kcal_0$
  and its Smory\'nski extension $\kcal$ with $\Gamma:=\{(A^-)^\Boxin\}$.
   Then by \Cref{Lem-32} we have 
   $\HA\vdash L=\alpha_1\to \sigma_\tinysub{\sf HA}(\neg (A^-)^\Boxin)$. Since 
   $\lles\vdash A\leftrightarrow  (A^-)^\Boxin $, by soundness part of 
   the \Cref{Theorem-sigma-provability-HA} we have 
   $\HA\vdash L=\alpha_1\to \sigma_\tinysub{\sf HA}(\neg A)$. Hence 
   by ``item 8"
 we may deduce $\mathbb{N}\not\models\sigma_\tinysub{\sf HA}(A)$, as desired.
\end{proof}

\subsection{Reductions}
In this subsection we will show that 
$$\PLS(\HA,\PA)\leq\PLS(\HA,\HA)\leq\PLS(\HA,\nat)$$
First some definition:
\begin{definition}\label{Definition-neg-translation}
For $A\in\lcalb$ we define $A^\negin,A^\negout$ and $A^\neg$
as follows:
\begin{itemize}
\item $(A\circ B)^\neg:= \neg\neg(A^\neg\circ B^\neg)$,
		 $(A\circ B)^\negout:= \neg\neg(A^\negout\circ B^\negout)$ and  $(A\circ B)^\negin:=  A^\negin \circ B^\negin$,
\item $(\Box A)^\neg:=\neg\neg\Box A^\neg$, 
$(\Box A)^\negout:=\neg\neg\Box A$ 
and $(\Box A)^\negin:=\Box A^\neg$,
\item $(\neg A)^\neg:=\neg A^\neg$, $(\neg A)^\negout:=\neg A^\negout$ and $(\neg A)^\negin:=\neg A^\negin$,
\item $p^\neg:=p^\negout:=\neg\neg p$  and $p^\negin:=p $
for atomic $p$.
\end{itemize}
For an arithmetical formula $A$ we have these additional clauses for the definition of $A^\neg$:
\begin{itemize}
\item $(\forall x\, A)^\neg:=\neg\neg\forall x\, A^\neg$,
\item $(\exists x\, A)^\neg:=\neg\neg\exists x\, A^\neg$.
\end{itemize}
\end{definition}

\begin{lemma}\label{Lemma-HA-PA-neg-translation}
For every formula $A$, we have $\PA\vdash A$ iff $\HA\vdash A^\neg$.
\end{lemma}
\begin{proof}
The direction from right to left is trivial. For the other way 
around, one may use induction on the proof $\PA\vdash A$. 
For details see \cite{TD}.
\end{proof}
\begin{lemma}\label{Lemma-Sigma-neg-translation}
For every $\Sigma_1$-formula $A$, we have 
 $\HA\vdash A^\neg\lr \neg\neg A$.
\end{lemma}
\begin{proof}
Easy by use of the decidability of $\Delta_0$-formulas
 in $\HA$ (\Cref{Lemma-decidability of delta formulae}).
\end{proof}

\begin{lemma}\label{Lemma-neg-translate-modal}
For every $A\in\lcalb$, we have $\ihatHsigma\vdash A$ iff 
$\lles\vdash A^\negout$.
\end{lemma}
\begin{proof}
The direction from right to left holds by the classically valid
 $A\lr A^\negout$. For the other way around, one must use induction on the length of the proof
 $\ihatHsigma\vdash A$. All cases are easy and left to the reader.
\end{proof}

\begin{lemma}\label{Lemma-neg-translate-1st-order}
For every $A\in\lcalb$, recursively axiomatizable theory $\SFT$
 and any $\Sigma_1$-substitution $\sigma$,  we have 
$\HA\vdash (\sigma_{_\SFT}(A))^\neg\lr \sigma_{_\SFT}(A^\negout)$.
\end{lemma}
\begin{proof}
We use induction on the complexity of $A$. All cases are simple.
For atomic and boxed cases, use
\Cref{Lemma-Sigma-neg-translation}.
\end{proof}

\begin{theorem}\label{Reduction-Sigma-HA-(PA-to-HA)}
$\ihatHsigma=\PLS(\HA,\PA)\leq\PLS(\HA,\HA)=\lles$.
\end{theorem}
\begin{proof}
By \Cref{Theorem-PA-relative} and \Cref{Theorem-sigma-provability-HA}
 we have $\ihatHsigma=\PLS(\HA,\PA)$ and 
$\PLS(\HA,\HA)=\lles$.
We must show
$\acs{\ihatHsigma}{\HA}{\PA}\leq_{f,\fbar[]} \acs{\lles}{\HA}{\HA}$.  
Given $A\in\lcalb$, define $f(A):=A^\negout$ and observe 
by \Cref{Lemma-neg-translate-modal} we have R1 (see 
\Cref{Definition-Reduction-PL}). Also define $\fbar$ as identity function. Then by \Cref{Lemma-neg-translate-1st-order,Lemma-HA-PA-neg-translation}
the condition  R2 holds.
\end{proof}

\begin{theorem}\label{Reduction-Sigma-HA-(HA-to-nat)}
$\lles=\PLS(\HA,\HA)\leq\PLS(\HA,\nat)=\ihatHSsigma$.
\end{theorem}
\begin{proof}
By \Cref{Theorem-sigma-provability-HA,Theorem-truth-provability}
 we have $\lles=\PLS(\HA,\HA)$ and 
$\PLS(\HA,\nat)=\ihatHSsigma$.
We must show
$\acs{\lles}{\HA}{\HA}\leq_{f,\fbar[]} \acs{\ihatHSsigma}{\HA}{\nat}$.  
Given $A\in\lcalb$, define $f(A)=\Box A$ and $\fbar$ as identity function.  
\begin{itemize}
\item[R1.] Let $\ihatHSsigma\vdash \Box A$. By soundness
of $\ihatHSsigma=\PLS(\HA,\nat)$, for every $\Sigma_1$-substitution $\sigma$
we have $\nat\models \sigma_{_{\sf HA}}(\Box A)$ and hence 
$\HA\vdash \sigma_{_{\sf HA}}(A)$. Then by arithmetical 
completeness of $\PLS(\HA,\HA)$, we have $\lles\vdash A$.\\
One also may prove this item with a direct propositional argument. For simplicity reasons, we chose the indirect way.
\item[R2.] Let $\nat\not\models \sigma_{_{\sf HA}}(\Box A)$. 
Then $\HA\nvdash \sigma_{_{\sf HA}}(A)$, as desired. \qedhere
\end{itemize}
\end{proof}

\section{Relative $\Sigma_1$-provability logics for ${\sf HA}^*$}
The $\sigma_1$-provability logic of $\HA^*$, $\PLS(\HA^*,\HA^*)$, 
is already characterized \cite{Sigma.Prov.HA*}.
 In this section, we characterize the $\Sigma_1$-provability 
 logic of $\HA^*$, relative in $\PA$ and $\nat$.
 We also show that the following reductions hold:
 
\begin{figure}[h]
\[\hspace*{-1cm} \begin{tikzcd}[column sep=3em, row sep=4em] 
\tcboxmath{\PLS(\HA,\PA) }
\arrow[r , "(.)^\negout", "\text{\ref{Reduction-Sigma-HA-(PA-to-HA)}}" ']  
&\tcboxmath{\PLS(\HA,\HA) }
\arrow[rr , "\Box(.)", "\text{\ref{Reduction-Sigma-HA-(HA-to-nat)}}" '] 
&  & \tcbhighmath{\PLS(\HA,\nat)}
\\
 \tcboxmath{\PLS(\HA^*,\PA)}
\arrow[r, "(.)^\negout", "\text{\ref{Reduction-Sigma-HA*-PA-to-HA*-HA}}" '] 
\arrow[u,"(.)^\Boxin", "\text{\ref{Reduction-Sigma-HA*-PA-to-HA-PA}}" ']
&  \tcboxmath{\PLS(\HA^*,\HA) }
\arrow[u,"(.)^\Boxin", "\text{\ref{Reduction-Sigma-HA*-HA-to-HA-HA}}" ']
 & 
  \tcboxmath{\PLS(\HA^*,\HA^*) }
 \arrow[r,"\Box(.)", "\text{\ref{Reduction-Sigma-HA*-(HA*-to-nat)}}" '] 
 \arrow[l,"(.)^\Boxout"', "\text{\ref{Reduction-Sigma-HA*-(HA*-to-HA)}}"]
 &
   \tcboxmath{\PLS(\HA^*,\nat) }
\arrow[u,"(.)^\Boxin", "\text{\ref{Reduction-Sigma-HA*-nat-to-HA-nat}}" ']
\end{tikzcd}
\]
\caption[Diagram of  reductions for  relative provability logics of $\HA^*$]{\label{Diagram-HA*} Reductions for  relative provability logics of $\HA^*$}
\end{figure}

Each arrow in the above diagram, indicates a reduction of the completeness of the left hand side to the right one.
Note that the diagram of the first row is already known by 
\Cref{Reduction-Sigma-HA-(PA-to-HA),Reduction-Sigma-HA-(HA-to-nat)}.

\begin{theorem}\label{Reduction-Sigma-HA*-nat-to-HA-nat}
  $\ihatHSsigmastar=\PLS(\HA^*,\nat)\leq \PLS(\HA,\nat)=\ihatHSsigma$. \uparan{See \Cref{Def-truth-prov-logic}}
\end{theorem}
\begin{proof}
By \Cref{Theorem-truth-provability} we have $\PLS(\HA,\nat)=\ihatHSsigma$. It is enough to prove 
the arithmetical soundness 
$\mathcal{AS}_{_{\Sigma_1}}(\ihatHSsigmastar;\HA^*,\nat)$ and the reduction 
$\acs{\ihatHSsigmastar}{\HA^*}{\nat}\leq \acs{\ihatHSsigma}{\HA}{\nat}$. 
\\
$\mathcal{AS}_{_{\Sigma_1}}(\ihatHSsigmastar;\HA^*,\nat)$: Let 
$\ihatHSsigmastar\vdash A$ and $\sigma$ is a $\Sigma_1$-substitution.  Then $\ihatHSsigma\vdash A^\Boxin$, and then 
by arithmetical soundness of $\ihatHSsigma$ 
\Cref{Theorem-truth-provability}, we have 
$\nat\models \sigma_{_{\sf HA}} (A^\Boxin)$. Hence  
\Cref{Lemma-Properties of Box translation 2} implies 
$\mathbb{N}\models \sigma_{_{\sf HA^*}} (A)$, as desired.
\\
For the proof of $\acs{\ihatHSsigmastar}{\HA^*}{\nat}\leq_{f,\fbar[]}
\acs{\ihatHSsigma}{\HA}{\nat}$, define $f(A):=A^\Boxin$ and 
$\fbar$ as identity function.
\begin{itemize}
\item[R1.] Let $\ihatHSsigma\vdash A^\Boxin$. Then by definition we have $\ihatHSsigmastar\vdash A$. 
\item[R2.] Let $\nat\not\models \sigma_{_{\sf HA^*}} (A^\Boxin)$. Hence by \Cref{Lemma-Properties of Box translation 2}
$\nat\not\models \sigma_{_{\sf HA}} (A)$, as desired.\qedhere
\end{itemize}
\end{proof}

\begin{theorem}\label{Reduction-Sigma-HA*-PA-to-HA-PA}
  $\ihatHsigmastar=\PLS(\HA^*,\PA)\leq \PLS(\HA,\PA)=\ihatHsigma$.
  \uparan{See \Cref{Def-truth-prov-logic}}
\end{theorem}
\begin{proof}
Similar to the proof of 
\Cref{Reduction-Sigma-HA*-nat-to-HA-nat} and left 
to the reader.
\end{proof}

\begin{theorem}\label{Reduction-Sigma-HA*-HA-to-HA-HA}
  $\llesstar=\PLS(\HA^*,\HA)\leq \PLS(\HA,\HA)=\lles$.
  \uparan{See  \Cref{Def-lles}}
\end{theorem}
\begin{proof}
Similar to the proof of 
\Cref{Reduction-Sigma-HA*-nat-to-HA-nat} and left 
to the reader.
\end{proof}

\begin{lemma}\label{Lemm-llesstar-llessstar}
For every $A\in\lcalb$ we have 
$\llessstar\vdash A$ iff $\llesstar\vdash  A^\Boxout$. \uparan{See  \Cref{Def-lles}}
\end{lemma}
\begin{proof}
We have the following equivalents:
 $\llessstar\vdash A$ iff $\lles\vdash A^\Box$ iff 
 (by \Cref{Remark-10}) $\lles\vdash (A^\Boxout)^\Boxin$ iff
 $\llesstar\vdash A^\Boxout$.
\end{proof}
\begin{theorem}\label{Reduction-Sigma-HA*-(HA*-to-HA)}
  $\llessstar=\PLS(\HA^*,\HA^*)\leq \PLS(\HA^*,\HA)=\llesstar$.
\end{theorem}
\begin{proof}
By \Cref{Reduction-Sigma-HA*-HA-to-HA-HA} 
we have $\PLS(\HA^*,\HA)=\llesstar$. 
It is enough to prove the arithmetical soundness 
$\mathcal{AS}_{_{\Sigma_1}}(\llessstar;\HA^*,\HA^*)$ and the reduction 
$\acs{\llessstar}{\HA^*}{\HA^*}\leq 
\acs{\llesstar}{\HA^*}{\HA}$. 
\\
$\mathcal{AS}_{_{\Sigma_1}}(\llessstar;\HA^*,\HA^*)$: Let 
$\llessstar\vdash A$ and $\sigma$ is a $\Sigma_1$-substitution.  Then $\lles\vdash A^\Box$, and then 
by arithmetical soundness of $\lles$ 
\Cref{Theorem-sigma-provability-HA}, we have 
$\HA\vdash \sigma_{_{\sf HA}} (A^\Box)$. Hence  
\Cref{Label-HA-HA*-Box-trans} implies 
$\HA\vdash \sigma_{_{\sf HA^*}} (A)^{\sf HA}$, 
which implies $\HA^*\vdash \sigma_{_{\sf HA^*}} (A)$.
\\
For the proof of $\acs{\ihatHSsigmastar}{\HA^*}{\nat}\leq_{f,\fbar[]}
\acs{\ihatHSsigma}{\HA}{\nat}$, define $f(A):=A^\Boxout$ and 
$\fbar$ as identity function.
\begin{itemize}
\item[R1.] Let $\llesstar\vdash A^\Boxout$. Then by  
\Cref{Lemm-llesstar-llessstar} we have $\llessstar\vdash A$, as desired.
\item[R2.] Let $\HA\nvdash \sigma_{_{\sf HA^*}} (A^\Boxout)$. Hence by \Cref{Label-HA-HA*-Box-trans2} we have 
$\HA\nvdash (\sigma_{_{\sf HA^*}} (A))^{\sf HA}$, which
implies $\HA^*\nvdash \sigma_{_{\sf HA^*}} (A)$, as desired.\qedhere
\end{itemize}
\end{proof}

\begin{theorem}\label{Reduction-Sigma-HA*-PA-to-HA*-HA}
  $\ihatHsigmastar=\PLS(\HA^*,\PA)\leq \PLS(\HA^*,\HA)=\llesstar$.
\end{theorem}
\begin{proof}
$\ihatHsigmastar=\PLS(\HA^*,\PA)$ and $\llesstar=\PLS(\HA^*,\HA)$, by \Cref{Reduction-Sigma-HA*-PA-to-HA-PA,Reduction-Sigma-HA*-HA-to-HA-HA} holds. 
Given $A$, define  $f(A):=A^\negout$ and $\fbar$ as identity function.
\begin{itemize}
\item[R1.] By definition of 
$\ihatHsigmastar$, we have $\ihatHsigmastar\vdash A$ iff 
$\ihatHsigma\vdash A^\Boxin$.  The latter, by \Cref{Lemma-neg-translate-modal} is equivalent to $\lles\vdash (A^\Boxin)^\negout$. Since $(A^\Boxin)^\negout=(A^\negout)^\Boxin$, the latter is equivalent to $\llesstar\vdash A^\negout$.  
\item[R2.]  By \Cref{Lemma-neg-translate-1st-order,Lemma-HA-PA-neg-translation}.\qedhere
\end{itemize}
\end{proof}

\begin{theorem}\label{Reduction-Sigma-HA*-(HA*-to-nat)}
$\llessstar=\PLS(\HA^*,\HA^*)\leq
\PLS(\HA^*,\nat)=\ihatHSsigmastar$.
\end{theorem}
\begin{proof}
By 
\Cref{Reduction-Sigma-HA*-nat-to-HA-nat,Reduction-Sigma-HA*-(HA*-to-HA)} 
we have $\PLS(\HA^*,\nat)=\ihatHSsigmastar$
and $\llessstar=\PLS(\HA^*,\HA^*)$.
We must show
$\acs{\llessstar}{\HA^*}{\HA^*}\leq_{f,\fbar[]} \acs{\ihatHSsigmastar}{\HA}{\nat}$.  
Given $A\in\lcalb$, define $f(A)=\Box A$ and $\fbar$ as identity function.  
\begin{itemize}
\item[R1.] Let $\ihatHSsigmastar\vdash \Box A$. By soundness
of $\ihatHSsigmastar=\PLS(\HA^*,\nat)$, for every $\Sigma_1$-substitution $\sigma$
we have $\nat\models \sigma_{_{\sf HA^*}}(\Box A)$ and hence 
$\HA^*\vdash \sigma_{_{\sf HA^*}}(A)$. Then by arithmetical 
completeness of $\PLS(\HA^*,\HA^*)$, 
we have $\llessstar\vdash A$.\\
One also may prove this item with a direct propositional argument. For simplicity reasons, we chose the indirect way.
\item[R2.] If  $\nat\not\models \sigma_{_{\sf HA^*}}(\Box A)$ evidently we have $\HA^*\nvdash \sigma_{_{\sf HA^*}}(A)$. \qedhere
\end{itemize}
\end{proof}
\section{Relative provability logics for ${\sf PA}$}
In this section, we characterize $\PL(\PA,\HA)$ and 
$\PLS(\PA,\HA)$,
 the provability logic and $\Sigma_1$-provability logic 
 of $\PA$
relative in $\HA$. We show that $\PL(\PA,\HA)=\iglphat$ and 
$\PLS(\PA,\HA)=\iglphatsigma$. Also we show that all 
of the six
 ($\Sigma_1$-) provability logics of $\PA$ relative in $\PA,\HA,\mathbb{N}$  are reducible to $\PLS(\HA,\mathbb{N})$:

\begin{figure}[h]
\[ \begin{tikzcd}[column sep=5em, row sep=4em] 
 \tcboxmath{\PL(\PA,\HA)} 
\arrow[d, "\ref{Reduction-PA-HA-to-Sigma-PA-HA} "  ', "\tau" ]
& 
 \tcboxmath{\PL(\PA,\PA) }
 \arrow[rr,   "\Box (.)",  " \ref{Reduction-Sigma-PA-(PA-to-nat)}" ' ]
\arrow[d, "\ref{Reduction-Sigma-PA-(PA-nat)}.2" ' , "\tau"  ]
&  &
 \tcboxmath{\PL(\PA,\nat) }
\arrow[d, "\text{\ref{Reduction-Sigma-PA-(PA-nat)}.1}" ' , "\tau"]
\\
 \tcboxmath{\PLS(\PA,\HA) }
\arrow[dr, "\ref{Reduction-Sigma-(PA-to-HA)-HA} "  near end, "(.)^\dagger" ' near end]
& 
 \tcboxmath{\PLS(\PA,\PA) }
 \arrow[rr,   "\Box (.)",  " \ref{Reduction-Sigma-PA-(PA-to-nat)}" ' ]
 \arrow[l,"(.)^\negout" ' , "\ref{Reduction-Sigma-PA-(PA-to-HA)}"] 
\arrow[dl, "\ref{Reduction-Sigma-(PA-to-HA)-PA}" ' near end, "(.)^\dagger"  near end]
&  &
 \tcboxmath{\PLS(\PA,\nat) }
\arrow[d, "\text{\ref{Reduction-Sigma-(PA-to-HA)-nat}}" , "(.)^\dagger" ']
\\
\tcboxmath{\PLS(\HA,\PA) }
\arrow[r , "(.)^\negout", "\text{\ref{Reduction-Sigma-HA-(PA-to-HA)}}" ']  
&\tcboxmath{\PLS(\HA,\HA) }
\arrow[rr , "\Box(.)", "\text{\ref{Reduction-Sigma-HA-(HA-to-nat)}}" '] 
&  & \tcbhighmath{\PLS(\HA,\nat)}
\end{tikzcd}
\]
\caption[Diagram of  reductions for  relative provability logics of $\PA$]{\label{Diagram-full} Reductions for  relative provability logics of $\PA$}
\end{figure}
Let us first review some well-known results:
\begin{theorem}\label{Solovey}
We have the following provability logics:
\begin{itemize}
\item $\GL$ is the provability logic of $\PA$, i.e.~$\PL(\PA,\PA)=\GL$. 
\cite{Solovay}
\item $\GLS$ is the truth provability logic of $\PA$, i.e.~$\PL(\PA,\mathbb{N})=\GLS$.
\cite{Solovay}
\item ${\GLV}$ is the $\Sigma_1$-provability logic of $\PA$, 
i.e.~$\PLS(\PA,\PA)={\GLV}$. \cite{Visser82}
\item ${\GLSV}$ is the truth $\Sigma_1$-provability logic of $\PA$,
i.e.~$\PLS(\PA,\mathbb{N})={\GLSV}$. \cite{Visser82}
\end{itemize}
\end{theorem}

\begin{definition}\label{Definition-boxin-sub}
A propositional modal substitution 
$\tau$ is called 
 $(.)^\Boxin$-substitution, if  
for every atomic variable $p$, there is some $B$ such that
${\sf iK4}+{\sf CP_a}\vdash\tau(p)\lr B^\Boxin$ and 
${\sf iK4}\vdash\Boxdot B^\Boxin \lr B^\Box$.
\end{definition}

\begin{lemma}\label{Lemma-Boxin-sub}
For every $(.)^\Boxin$-substitution $\tau$ and every 
modal proposition $A$, we have 
${\sf iK4V}\vdash \tau(A^\Box)\lr 
\tau(A)^\Box$ and 
${\sf iK4V}\vdash \tau(A^\Boxin)\lr 
\tau(A)^\Boxin$.
\end{lemma}
\begin{proof}
First by induction on the complexity of $B$ we show
${\sf iK4V}\vdash \tau(B^\Box)\lr \tau(B)^\Box$. 
All cases are easy, except for atomic $B$, which holds by  
existence of some $C$ such that
${\sf iK4V}\vdash\tau(B)\lr C^\Boxin$ and 
${\sf iK4}\vdash\Boxdot C^\Boxin \lr C^\Box$.
\\
Then we use induction on the complexity of $A$ to deduce the second  assertion of this lemma. The only non-trivial cases are atomic and boxed cases:
\begin{itemize}[leftmargin=*]
\item $A$ is atomic. Since ${\sf iK4}\vdash B^\Boxin\leftrightarrow (B^\Boxin)^\Boxin$ for every $B$, and 
${\sf iK4V}\vdash \tau(A)\lr B^\Boxin$, we have the 
desired result.
\item $A=\Box B$. Easily deduced by 
${\sf iK4V}\vdash \tau(B^\Box)\lr \tau(B)^\Box$.
\end{itemize}
\end{proof}

\noindent The following remark, will be helpful for later reductions 
of provability logics in 
\cref{section-PA*}.
\begin{remark}\label{Remark-reduction-GL-GLS}
For every modal proposition $A$, $\GL\vdash A$ ($\GLS\vdash A$) 
iff for every 
 $(.)^\Boxin$-substitution $\tau$ we have 
${\GLV}\vdash \tau(A)$ (${\GLSV}\vdash \tau(A)$). 
\end{remark}
\begin{proof}
See \cite[Lemmas.~3.1 and 3.3]{reduction} .
\end{proof}
\begin{lemma}\label{Lemma-GLS-GL-Boxed}
For every $A\in\lcal_\Box$,
\begin{itemize}
\item $\GLS\vdash \Box A$ iff $\GL\vdash A$,
\item ${\GLSV}\vdash \Box A$ iff ${\GLV}\vdash A$.
\end{itemize}
\end{lemma}
\begin{proof}
The proof of second item is similar to the first one. Here 
we only treat the first item.
Obviously, $\GL\vdash A$ implies $\GLS\vdash \Box A$. 
For  a direct proof of the other way around, 
one may  use of Smor\'ynski's operation. However, now  that we  
enjoy the arithmetical soundness of  $\PL(\PA,\nat)=\GLS$, 
from $\GLS\vdash \Box A$  for every $\sigma$  we have
$\nat\models \sigma_{_{\sf PA}}(\Box A)$   and hence 
$\PA\vdash  \sigma_{_{\sf PA}}( A)$. 
From the arithmetical completeness of
$\GL=\PL(\PA,\PA)$, we get $\GL\vdash A$.
\end{proof}
\noindent In the following theorem, we will show that ${\GLSV}$
is the hardest provability logic among $\GL$, ${\GLV}$, $\GLS$ and ${\GLSV}$.
\begin{theorem}\label{Reduction-Sigma-PA-(PA-nat)}
We have the following reductions:
\begin{enumerate}
\item $\PL(\PA,\mathbb{N})\leq \PLS(\PA,\mathbb{N})$,
\item $\PL(\PA,\PA)\leq \PLS(\PA,\PA)$.
\end{enumerate}
\end{theorem}
\begin{proof}
We prove each item separately:
\begin{enumerate}[leftmargin=*]
\item  We must show that 
$\ac{\GLS}{\PA}{\nat}\leq_{f,\bar{f}} \acs{{\GLSV}}{\PA}{\nat}$. 
  Consider some $A\in\lcal_\Box$.
   If $\GLS\nvdash A$, by 
  \Cref{Remark-reduction-GL-GLS}, there exists some 
  $\lcal_\Box$-substitution $\tau$ such that 
  ${\GLSV}\nvdash \tau(A)$.  Let 
  $$f(A):=\begin{cases}
  \tau(A) &: \GLS\nvdash A\\
  A &: \text{otherwise}
\end{cases}  
  $$
  Hence R1 (\Cref{Definition-Reduction-PL}) holds. Also 
  let $\fbar(\sigma):=\sigma_{_{\sf PA}}\circ\tau$, which belongs to $\wit{\sigma}$.
   Then obviously R2 holds.
\item    Similar to first item and left to the reader.
  \qedhere
\end{enumerate}
\end{proof}
\subsection{Reducing $\PLS(\PA,\mathbb{N})$ to $\PLS(\HA,\mathbb{N})$}\label{Sec-reduction-tplpa-tplha}

In this subsection, we illustrate how to reduce the 
arithmetical completeness of ${\GLSV}$ to that of $\ihatHSsigma$.  First some definitions and lemmas:

\begin{definition}
For a modal proposition  $A$ let $A^\ddagger$ indicates 
the classically equivalent formula of the form 
$$A^\ddagger:=\bigwedge_i B_i\to C_i\quad \text{in which }\quad 
B_i=\bigwedge_j E^\ddagger_{i,j}  \text{ and } 
C_i=\bigvee_j F^\ddagger_{i,j} 
$$
and $E_{i,j}, F_{i,j}$ are atomic or boxed formulas. 
Also  for atomic $p$ we have $p\ddagger=p$ and 
$(\Box E)^\ddagger=\Box(E^\ddagger)$.
Then  define $A^\dagger$ in this way:
\begin{itemize}
\item $(.)^\dagger$ commutes with $\vee,\wedge,\to$,
\item $p^\dagger=p$ for atomic $p$,
\item $(\Box A)^\dagger=\Box A^\ddagger$
\end{itemize}
\end{definition}

\begin{lemma}\label{Lemma-HA-proves-dagger-eq}
For every modal proposition $A$ and arithmetical substitution 
$\alpha$, we have 
$$\HA\vdash \alpha_{_{\sf HA}}(A^\dagger) \lr 
\alpha_{_{\sf PA}}(A)$$
\end{lemma}
\begin{proof}
Easy and left to the reader. 
\end{proof}

\begin{lemma}\label{Lemma-reduction-glssigma-ihatHSsigma}
For every $A\in\lcalb$, if  
$\ihatHSsigma\vdash A^\dagger$ then ${\GLSV}\vdash A$.
\end{lemma}
\begin{proof}
Let ${\GLSV}\nvdash A$. Since in classical logic we have 
$A\lr A^\dagger$, then ${\GLSV}\nvdash A^\dagger$.
Hence by $\acs{{\GLSV}}{\PA}{\nat}$ from 
\ref{Solovey},
we have some $\Sigma_1$-substitution $\sigma$ such that 
$\nat\not\models \sigma_{_{\sf PA}}(A^\dagger)$. Then 
\Cref{Lemma-HA-proves-dagger-eq} implies 
$\nat\not\models \sigma_{_{\sf HA}}(A^\dagger)$, and hence 
by arithmetical soundness of $\ihatHSsigma$ (\Cref{Theorem-truth-provability}) we have 
$\ihatHSsigma\nvdash A^\dagger$, as desired.
\end{proof}

\begin{lemma}\label{Lemma-reduction-ihatHsigma-GLsigma}
For every $A\in\lcalb$, if  
$\ihatHsigma\vdash A^\dagger$ then ${\GLV}\vdash A$.
\end{lemma}
\begin{proof}
Let ${\GLV}\nvdash A$. Since in classical logic we have 
$A\lr A^\dagger$, then ${\GLV}\nvdash A^\dagger$.
Hence by $\acs{{\GLV}}{\PA}{\PA}$ from 
\ref{Solovey},
we have some $\Sigma_1$-substitution $\sigma$ such that 
$\PA\nvdash\sigma_{_{\sf PA}}(A^\dagger)$. Then 
\Cref{Lemma-HA-proves-dagger-eq} implies 
$\PA\nvdash \sigma_{_{\sf HA}}(A^\dagger)$, and hence 
by arithmetical soundness of $\ihatHsigma$ (\Cref{Theorem-PA-relative}) we have 
$\ihatHsigma\nvdash A^\dagger$, as desired.
\end{proof}

\begin{theorem}\label{Reduction-Sigma-(PA-to-HA)-nat}
${\GLSV}=\PLS(\PA,\mathbb{N})\leq 
\PLS(\HA,\mathbb{N})=\ihatHSsigma$.
\end{theorem}
\begin{proof}
By 
\Cref{Theorem-truth-provability,Solovey} we have 
$\ihatHSsigma=\PLS(\HA,\mathbb{N})$ and 
 ${\GLSV}= \PLS(\PA,\mathbb{N})$.
For the reduction, 
  let $f(A):=A^\dagger$   
and  $\fbar$ as identity function.
\begin{enumerate} 
\item[R1.] If $\ihatHSsigma\vdash A^\dagger$,  by  
\Cref{Lemma-reduction-glssigma-ihatHSsigma} we have 
${\GLSV}\vdash A$.  
\item[R2.] Holds by \Cref{Lemma-HA-proves-dagger-eq}.\qedhere
\end{enumerate}
\end{proof}

\subsection{Kripke Semantics}
Let ${\sf Suc}_\kcal$ or simply ${\sf Suc}$, when no confusion is likely,
 indicates the set of all 
 $\R$-accessible nodes in the Kripke model $\kcal$.
\begin{theorem}\label{Theorem-Kripke-semantic-iglphat}
$\iglphat$ is sound and complete for semi-perfect 
${\sf Suc}$-classical $\R$-branching Kripke models. 
\end{theorem}
\begin{proof}
The soundness is easy and left to the reader. 
For the 
completeness, we first show the completeness for finite  brilliant  irreflexive transitive  ${\sf Suc}$-classical Kripke models. 
Let $\iglphat\nvdash A$. Let
$$X:=\{ B,\neg B,B\vee\neg B: B\in{\sf Sub}(A)\}\cup\{\bot\}$$
and define the 
Kripke model $\kcal=(K,\preccurlyeq,\R,V)$  as follows:
\begin{itemize}
\item $K$ is the family of all $X$-saturated sets 
with respect to $\iglphat$.
\item $\alpha\preccurlyeq \beta$  iff $\alpha\subseteq \beta$.
\item $\alpha\R\beta$ iff $\beta$ is a maximally consistent
set and $\{B,\Box B: \Box B\in \alpha\}\subseteq \beta$ and 
there is some $\Box B\in \beta\setminus \alpha$.
\end{itemize}
It is straightforward to show that $\kcal$ 
is actually a finite brilliant irreflexive ${\sf Suc}$-classical Kripke model, and we leave all of them 
to the reader.

It is enough to show that $\kcal,\alpha\Vdash B$ iff 
$B\in \alpha$ for every $\alpha\in K$ and $B\in X$. 
Then we may use \Cref{Lemma-saturation} and find some 
$\alpha$ such that $\kcal,\alpha\nVdash A$. 
 We use induction on the complexity of $B\in X$. All inductive 
 steps are trivial, except for $B=\Box C$. 
 If $\Box C\in \alpha$ and $\alpha\R\beta$, then by definition, 
 $C\in \beta$ and hence by induction hypothesis 
 $\beta\Vdash C$.
 This implies $\alpha\Vdash \Box C$. For the other way around,
 let $ \Box C\not\in \alpha$. Consider the set 
 $\Delta:=\{E,\Box E:\Box E\in \alpha\}$. 
 If $\GL\vdash \bigwedge\Delta\to (\Box C\to C)$,  then 
$\igl+\Box\PEM\vdash \Box(\bigwedge \Delta)\to \Box C$. Since 
$\iglphat\vdash \alpha\to\Box\bigwedge\Delta$
and $\iglphat\vdash \Box\PEM$, we have 
$\iglphat+\alpha\vdash \Box C$ and hence $\Box C\in \alpha$, a contradiction.
Hence we have $\GL\nvdash (\bigwedge\Delta\wedge \Box C)\to C$. Then 
by \Cref{Lemma-saturation} there is some 
 $X$-saturated set 
 $\beta\supseteq \Delta\cup\{\Box C\}\cup\{E\vee\neg E: E\in {\sf Sub}(A)\}$ 
 such that $C\not\in\beta$.  Hence $\beta\sqsupset\alpha$ 
 and $\beta\nVdash C$. Then $\alpha\nVdash \Box C$, as desired.

Next we use the construction method \cite{IemhoffT}, to fulfil the other conditions: $\R$-branching, neat and tree.  Let $\kcal_t:=(K_t,\preccurlyeq_t,\R_t,V_t)$
as follows:
\begin{itemize}
\item $K_t$ is the set of all finite sequences of pairs
$r:=\langle(\alpha_0,a_0),\ldots(\alpha_n,a_n)\rangle$ 
such that 
for any $i\leq n$: 
(1) $\alpha_i\in K$, (2) $a_i\in\{0,1\}$, (3)
for $i< n$  either we have $\alpha_i\prec\alpha_{i+1}$ or $\alpha_i\R\alpha_{i+1}$.  Let $f_1(r)$ and 
$f_2(r)$ indicate  the left and right
elements in the final element of the sequence $r$. In other words, we let $(f_1(r),f_2(r))$ be the final element of the sequence $r$.
\item $r\preccurlyeq_t s$ iff $r$ is an initial segment of $s$ and $f_1(r)\preccurlyeq f_1(s)$.
\item $r\R_t s$ iff $r$ is an initial segment of $s=
\langle(\alpha_0,a_0),\ldots,(\alpha_n,a_n)\rangle$,
e.g. $r=\langle(\alpha_0,a_0),\ldots,(\alpha_k,a_k)\rangle$ 
for some $k<n$ and $ \alpha_i\R\alpha_{i+1}$ for some $k\leq i<n$.
\item $r\, V_t \,p$ iff $f_1(r)\,V \,p$.
\end{itemize}
It is straightforward to show that $\kcal_t$ is semi-perfect 
$\R$-branching
${\sf Suc}$-classical Kripke model and for every $r\in K_t$ and formula $B$ we have 
\[\kcal_t,r\Vdash B \quad \quad \Longleftrightarrow \quad\quad 
\kcal, f(r)\Vdash B.\qedhere\]
\end{proof}

\begin{theorem}\label{Theorem-Kripke-semantic-iglphatsigma}
$\iglphatsigma$ is sound and complete for semi-perfect 
${\sf Suc}$-classical atom-complete Kripke models. 
\end{theorem}
\begin{proof}
The proof is almost identical to the one for \Cref{Theorem-Kripke-semantic-iglphat}. We only explain the differences here. 
Define 
$$X:=\{B,\neg B,B\vee\neg B: B\in {\sf Sub}(A)\}\cup\{\bot\}\cup
\{\Box p: p\in {\sf Sub}(A)\text{ and } p \text{ is atomic}\}$$
and  $K$, the set of the nodes of Kripke model,  is defined as the set of 
all $X$-saturated sets with respect to $\iglphatsigma$. 
We show that  every $\alpha\in K$ is atom-complete. Let 
$p$ be  an atomic variable such that $\alpha\Vdash p$.
 Hence  $p\in \alpha$ which implies 
 $p\in \sub{A}$,  
 and since $\iglphatsigma\vdash p\to\Box p$ and $\alpha$
 is closed under deduction, we have $\Box p\in\alpha$. Then 
 $\alpha\Vdash \Box p$ and hence
 for every $\beta\sqsupset\alpha$ we have $ \beta\Vdash p$, as 
 desired.
\end{proof}

%

\subsection{Arithmetical Completeness}

\begin{theorem}\label{Theorem-Arith-Complete-iglphatsigma-direct}
$\iglphatsigma$ is the relative $\Sigma_1$-provability 
logic of $\PA$ in $\HA$,
i.e. $\PLS(\PA,\HA)=\iglphatsigma$.
\end{theorem}
\begin{proof}
The soundness is straightforward and left to the reader.  
For the completeness part,
let $\iglphatsigma\nvdash A$. Then by 
\Cref{Theorem-Kripke-semantic-iglphatsigma}, there is some semi-perfect atom-complete
${\sf Suc}$-classical Kripke model $\Kripke$
such that $\kcal,\alpha_0\nVdash A$ for some $\alpha_0\in K$. 
Without loss of generality, we may assume that
$K=(\alpha_0\preccurlyeq)\cup(\alpha_0\R)$.
Let $\kcal'=(K',\preccurlyeq',\R',V')$ indicates the 
Smor\'ynski's extension of $\kcal$ at $\alpha_0$ 
with the fresh node $\alpha_1$.
For the simplicity of notations, we may use $\preccurlyeq$ and $\R$
instead of $\preccurlyeq'$ and $\R'$.
Define the recursive function $F$ as follows. Since $K'$ is a  finite set,
we might assign a unique number $\bar{\alpha}$ to each node $\alpha$
and speak about $K'$ and its relationships $\preccurlyeq$ and $\R$
inside the language of arithmetic. For simplicity of notations, we may 
simply use  $\alpha\preccurlyeq \beta$ and $\alpha\R\beta$ 
corresponding to its equivalent arithmetical formula.

 Define 
$F(0):=\alpha_1$  and 
$$F(n+1):=\begin{cases}
\beta &: F(n)\R\beta \text{ and }  r(\beta,n+1)< n+1 \text{ and } (n)_0=\beta\\
\beta &: F(n)\prec\beta  \text{ and }    F(n)\not\sqsubset\beta\text{ and }  
 \ F(r(\beta,n+1))=\alpha_1\\ 
  & \ \text{ and }  r(\beta,n+1)< r(F(n),n+1)\text{ and } (n)_0=\beta\\
F(n) &:\text{otherwise}
\end{cases}$$
in which $L=\beta$ is shorthand for 
$\exists x\forall y\geq x(F(y)=F(x))$, $(n)_0$ is 
the exponent of $2$ in $n$ and 
$$r(\alpha,n):={\sf min}\left(\{ x\in\mathbb{N}:  
 \exists\,t\leq n\, {\sf Proof}_{_{{\sf PA}_x}}(t,\gnumber{L\neq \alpha})\}\cup
\{n\}\right)$$
Note that $r(\alpha,n)< n$ 
implies $\Box^+(L\neq \alpha)$.
$F$ is a provably total recursive function in $\HA$,
 i.e. $F(x)=y$ could be 
expressed as a $\Sigma_1$-formula in the language of arithmetic and 
all of its expected properties are provable in $\HA$. 
Hence we may use the function symbol $F$ in the language of arithmetic.

 Define  the arithmetical substitution 	$\sigma(p)$     in this way:
 $$\sigma(p):=\bigvee_{\kcal,\alpha\Vdash p} 
 \exists x\, F(x)=\alpha$$

Consider  the triple  $\mathcal{I}:=(K^*,\preccurlyeq^*, T)$ as follows:
\begin{itemize}
\item $K^*:=\{\alpha\in K: \nexists\beta\in K (\beta\R\alpha)\}$.
\item $\alpha\preccurlyeq^*\beta$ iff $\alpha\preccurlyeq\beta$ for every $\alpha,\beta\in K^*$.
 Again, by abuse of notations, we use $\preccurlyeq$ instead of $\preccurlyeq^*$.
\item $T(\alpha):=\PA+(L=\alpha)$. 
\end{itemize}

 By 
 \Cref{Theorem-Smorynski's general method of Kripke model construction} and \Cref{Lem-Iframe}, 
 we have some 
 first-order Kripke model $\kcal^*=(K^*,\preccurlyeq,\mathfrak{M})$
 such that $\kcal^*\Vdash \HA$ and $\kcal^*,\alpha\models T(\alpha)$. 
By \Cref{Lemma-Sigma-local-global}  
\begin{equation}\label{eq102}
 \kcal^*,\alpha\Vdash \exists x\, F(x)=\beta  \quad \Longrightarrow \quad \beta\preccurlyeq\alpha 
\end{equation}
  Hence by \Cref{Lemma-Sigma-local-global},   
   for every $\alpha\in K^*$ 
\begin{equation}\label{eq103}
\kcal^*,\alpha\Vdash\sigma_\tinysub{\sf PA}(p)  \quad \Longleftrightarrow \quad 
\kcal,\alpha\Vdash p
\end{equation}
For every classical node $\alpha\in K^*$,  since the Kripke model above $\alpha$ is just a classical  Kripke model, 
one may repeat the Solovay's argument and show that for every modal proposition $B$ we have 
\begin{equation}\label{eq104}
\begin{cases}
\kcal,\alpha\Vdash B  \quad &\Longrightarrow \quad \PA\vdash L=\alpha\to \sigma_\tinysub{\sf PA}(B)\\
\kcal,\alpha\nVdash B  \quad &\Longrightarrow \quad \PA\vdash L=\alpha\to \neg\sigma_\tinysub{\sf PA}(B)
\end{cases}
\end{equation}

\noindent We may use \Cref{Lem-Solovay-4,Lem-Solovay-5} and \cref{eq103} to conclude  
$$
\kcal^*,\alpha\Vdash\sigma_\tinysub{\sf PA}(B)  \quad \Longleftrightarrow \quad 
\kcal,\alpha\Vdash B
$$
\noindent 
for every modal proposition $B$ and $\alpha\in K^*$. 
Since $\kcal,\alpha_0\nVdash A$, we have $\kcal^*,\alpha\nVdash \sigma_\tinysub{\sf PA}(A)$, and hence $\HA\nvdash \sigma_\tinysub{\sf PA}(A)$, as desired.
\end{proof}

\begin{lemma}\label{Lem-Sol-fun1}
For arbitrary $\alpha,\beta\in K'$ we have 
\begin{enumerate}
\item $\PA\vdash \exists x F(x)=\alpha\to
\bigvee_{\alpha(\preccurlyeq\cup\sqsubseteq) \beta} L=\beta $,
\item $\PA\vdash L=\alpha \to \neg \Box^+(  L\neq\beta )$, for 
every $\alpha\R \beta$,
\item $\PA\vdash (L=\alpha) \rhd (L=\beta)$, 
for every $\alpha\prec \beta$,
\item $\mathbb{N}\models L=\alpha_1$,
\item $\PA\vdash L=\alpha\to 
\Box^+( L\neq \alpha\wedge \exists x F(x)=\alpha)$, 
for every $\alpha\neq\alpha_1$.
\end{enumerate}
\end{lemma}
\begin{proof}
All proofs are straightforward and left to the reader. 
\end{proof}

\begin{lemma}\label{Lem-Iframe}
$\mathcal{I}$, as defined in the proof of \Cref{Theorem-Arith-Complete-iglphatsigma-direct},  is an 
 $I$-frame \uparan{see \Cref{Definition-Iframe}}. 
\end{lemma}
\begin{proof}
Use \Cref{Theorem-Orey} and the 
 items 2,3 and 4, of \Cref{Lem-Sol-fun1}.
 \end{proof}

\begin{lemma}\label{Lem-Solovay-3}
For every $\alpha\in K$ we have 
$\PA\vdash L=\alpha\to  \Box^+(\bigvee_{\alpha\R\beta}L=\beta)$.
\end{lemma}
\begin{proof}
It is enough to show that $\PA\vdash L=\alpha\to \Box^+(L\neq\beta)$ for  every $\beta\succcurlyeq \alpha$ such that 
$\beta\not\sqsupset\alpha$, holds. Consider some $\beta\succcurlyeq\alpha$ with $\beta\not\sqsupset \beta$. If $\beta=\alpha$,  
by item 5 in \Cref{Lem-Sol-fun1} we have the desired result. So 
we may let $\beta\neq\alpha$. 
We reason inside $\PA$. Let 
$L=\alpha$. Hence for some $x$ we have $F(x)=\alpha$.
 Then we reason inside $\Box^+$.  By $\Sigma_1$-completeness of 
 $\PA$   (see\Cref{Lemma-bounded Sigma completeness}), 
 we have $F(x)=\alpha$. Assume that $L=\beta$.  Let $x_0$ be the first number such that $F(x_0)=\beta$. Hence for some $r$
 such that $\Box^+_r(L\neq\beta)$ holds, we have  $F(r)=\alpha_1$. 
 Then $r\leq x$ and hence by \Cref{Lemma-Reflection} we may deduce 
 $L\neq\beta$, in contradiction with $L=\beta$. 
\end{proof}

\begin{lemma}\label{Lem-Solovay-4}
For every $\alpha$ in $K$ and proposition $B$
we have 
\begin{equation}  
\kcal,\alpha\Vdash   \Box B \quad \Longrightarrow \quad 
\PA\vdash L=\alpha\to\sigma_\tinysub{\sf PA}( \Box  B)
\end{equation}
\end{lemma}
\begin{proof}
Let $\kcal,\alpha\Vdash \Box B$. Hence for every 
$\beta\sqsupset \alpha$ we have $\kcal,\beta\Vdash B$. 
Since every $\beta\sqsupset \alpha$ is classical, by \cref{eq104}  
we have $\PA\vdash \bigvee_{\alpha\R\beta}L=\beta\to \sigma_\tinysub{\sf PA}(B)$. Hence $\PA\vdash \Box^+(\bigvee_{\alpha\R\beta}L=\beta) \to \sigma_\tinysub{\sf PA}(\Box B)$. 
\Cref{Lem-Solovay-3} implies 
$\PA\vdash L=\alpha\to \sigma_\tinysub{\sf PA}(\Box B)$.
\end{proof}
\begin{lemma}\label{Lem-Solovay-5}
For every $\alpha$ in $K$ and proposition $B$
we have 
\begin{equation} 
\kcal,\alpha\nVdash   \Box B \quad \Longrightarrow \quad 
\PA\vdash L=\alpha\to\neg\sigma_\tinysub{\sf PA}( \Box  B)
\end{equation}
\end{lemma}
\begin{proof}
Let $\kcal,\alpha\nVdash \Box B$. Hence for every 
$\beta\sqsupset \alpha$ we have $\kcal,\beta\nVdash B$. 
Since every $\beta\sqsupset \alpha$ is classical, by \cref{eq104}  
we have $\PA\vdash  L=\beta\to \neg\sigma_\tinysub{\sf PA}(B)$. 
Hence $\PA\vdash  \sigma_\tinysub{\sf PA}(B)\to L\neq \beta$ and 
then $\PA\vdash  \Box^+\sigma_\tinysub{\sf PA}(B)\to \Box^+L
\neq \beta$ and $\PA\vdash  
\neg\Box^+{L\neq \beta}\to 
\neg\Box^+\sigma_\tinysub{\sf PA}(B)$.  Hence item 2 of 
\Cref{Lem-Sol-fun1} implies 
 $\PA\vdash L=\alpha\to \neg\Box^+\sigma_\tinysub{\sf PA}(B)$.
\end{proof}

\subsection{Reductions}

\begin{lemma}\label{Lemma-reduction-ihatHsigma-GLsigma}
For every $A\in\lcalb$, if  
$\lles\vdash A^\dagger$ then $\iglphatsigma\vdash A$.
\end{lemma}
\begin{proof}
Let $\iglphatsigma\nvdash A$. Since in 
${\sf iK4}+\Box\PEM$
 we have 
$A\lr A^\dagger$, then $\iglphatsigma\nvdash A^\dagger$.
Hence by $\acs{\iglphatsigma}{\PA}{\PA}$ from 
\Cref{Theorem-Arith-Complete-iglphatsigma-direct},
we have some $\Sigma_1$-substitution $\sigma$ such that 
$\HA\nvdash\sigma_{_{\sf PA}}(A^\dagger)$. Then 
\Cref{Lemma-HA-proves-dagger-eq} implies 
$\HA\nvdash \sigma_{_{\sf HA}}(A^\dagger)$, and hence 
by arithmetical soundness of $\lles$ (\Cref{Theorem-sigma-provability-HA}) we have 
$\lles\nvdash A^\dagger$, as desired.
\end{proof}

\begin{theorem}\label{Reduction-Sigma-(PA-to-HA)-HA}
$\iglphatsigma=\PLS(\PA,\HA)\leq \PLS(\HA,\HA)=\lles$.
\end{theorem}
\begin{proof}
The soundness of $\iglphatsigma$ is straightforward
 and left to the reader.  Also by \Cref{Theorem-sigma-provability-HA}, we have $\PLS(\HA,\HA)=\lles$. 
 So, it is enough to show 
 $\acs{\iglphatsigma}{\PA}{\HA}\leq_{f,\fbar[]}\acs{\lles}{\HA}{\HA}$.
\\
Define $f(A):=A^\dagger$ and $\fbar$ as identity function.
\begin{itemize}
\item[R1.]  Use \Cref{Lemma-reduction-ihatHsigma-GLsigma}.
\item[R2.] Use \Cref{Lemma-HA-proves-dagger-eq}.\qedhere
\end{itemize}
\end{proof}

\begin{theorem}\label{Reduction-Sigma-(PA-to-HA)-PA}
${\GLV}=\PLS(\PA,\PA)\leq \PLS(\HA,\PA)=\ihatHsigma$.
\end{theorem}
\begin{proof}
We already have ${\GLV}=\PLS(\PA,\PA)$ and  $\PLS(\HA,\PA)=\ihatHsigma$ by \Cref{Theorem-PA-relative,Solovey}.
 So, it is enough to show 
 $\acs{{\GLV}}{\PA}{\PA}\leq_{f,\fbar[]}\acs{\ihatHsigma}{\HA}{\PA}$.
\\
Define $f(A):=A^\dagger$ and $\fbar$ as identity function.
\begin{itemize}
\item[R1.] Let $\ihatHsigma\vdash A^\dagger$.  By 
\Cref{Lemma-reduction-ihatHsigma-GLsigma} we have ${\GLV}\vdash A$.
\item[R2.] Use \Cref{Lemma-HA-proves-dagger-eq}.\qedhere
\end{itemize}
\end{proof}

\noindent
The arithmetical completeness of $\iglphat$ will be reduced to 
the  one for $\iglphatsigma$ via the following lemma. 
This argument is similar to the one 
explained in \cite{reduction}. One may use a direct proof for the arithmetical completeness of $\iglphat$, similar to what we 
do for $\iglphatsigma$. However this is not enough for our later use in\cref{section-PA*} of the arithmetical completeness of $\iglphat$.

\begin{lemma}\label{Lemma-Reduction-iglphat}
For every modal proposition $A$, $\iglphat\vdash A$ iff 
for every propositional modal $(.)^\Boxin$-substitution 
$\tau$ \uparan{\Cref{Definition-boxin-sub}}
we have $\iglphatsigma\vdash\tau(A)$. 
\end{lemma}
\begin{proof}
One direction holds since $\iglphat$ is closed under 
substitutions and is included in $\iglphatsigma$. 
For the other way around, let $\iglphat\nvdash A$. 
By \Cref{Theorem-Kripke-semantic-iglphat}, 
there is some 
${\sf Suc}$-classical, semi-perfect $\R$-branching 
Kripke  model $\Kripke$ such that  $\kcal\nVdash A$. 
For every $\alpha\in K$, let $p_{_\alpha}$ be a fresh atomic variable such that 
for every 
$\alpha\neq\beta$ we have $p_{_\alpha}\neq p_{_\beta}$. 
For every $\alpha\in K$, define $A_\alpha$ via induction on the 
$\prec$-height of $\alpha$ (the maximum number $n$ such that a 
sequence $\alpha=\alpha_0\prec\ldots\prec\alpha_n$ exists). 
So as induction hypothesis,
let $A_\beta$ for every $\beta\succ\alpha$ is defined. 
$$
A_\alpha^+:=\bigvee_{\alpha\prec\beta}A_\beta
\quad \text{,}\quad
A_\alpha:=p_{_\alpha}\wedge 
\bigwedge_{\alpha\sqsubseteq \beta}
\Box\neg\Boxdot p_{_\beta}\to A_\alpha^+  
$$
Let 
$\bar{\kcal}=
(K,\preccurlyeq,\R,\bar{V})$, in which $\alpha\, \bar{V}\, p$
iff $p=p_{_\beta}$ for some 
$\beta(\preccurlyeq\cup\sqsubset) \alpha$.
Define $$\tau(p):= \bigvee_{\kcal,\alpha\Vdash p} A_\alpha$$
Then by induction on the complexity of the modal proposition 
$B$, we show
$$\kcal,\alpha\Vdash B \quad \Longleftrightarrow \quad 
\bar{\kcal},\alpha\Vdash \tau(B)$$
\begin{itemize}[leftmargin=*]
\item $B$ is atomic variable: For every $\alpha\in K$
such that $\kcal,\alpha\Vdash B$, by 
\Cref{Lemma-1-reduction-iglphat} we have 
$\bar{\kcal},\alpha\Vdash A_\alpha$ and hence 
$\bar{\kcal},\alpha\Vdash \tau(p)$. Also if 
$\bar{\kcal},\alpha\Vdash \tau(B)$, then for some 
$\beta\in K$   we have $\kcal,\beta\Vdash B$ and 
$\bar{K},\alpha\Vdash A_\beta$. Hence by 
\Cref{Lemma-1-reduction-iglphat} we have 
$\beta\preccurlyeq\alpha$, 
which implies $\kcal,\alpha\Vdash B$, as desired.
\item All the other cases are trivial and left to the reader.
\end{itemize}
Then we have  $\bar{\kcal}\nVdash \tau(A)$. 
Obviously the Kripke model $\bar{\kcal}$ inherits all 
properties from $\kcal$ and moreover it is atom-complete.
Hence by  soundness part of the 
\Cref{Theorem-Kripke-semantic-iglphatsigma}, 
$\iglphatsigma\nvdash\tau(A)$, as desired.
\end{proof}
\begin{lemma}\label{Lemma-1-reduction-iglphat}
Let $\bar{\kcal}$ and $A_\alpha$, as defined in the proof of 
\Cref{Lemma-Reduction-iglphat}.
For every $\alpha,\beta\in K$ we have 
$\bar{\kcal},\alpha\Vdash A_\beta$ iff 
$\alpha\succcurlyeq \beta$.
\end{lemma}
\begin{proof}
 We use induction on the 
$\prec$-height of $\beta$.  As induction hypothesis,
let for every $\beta\succ\beta_0$ and $\alpha\in K$ we have 
$\bar{\kcal},\alpha\Vdash A_\beta$ 
iff $\beta\succcurlyeq\alpha$. Note that by induction 
hypothesis we have $\bar{\kcal},\beta\Vdash A^+_{\beta_0}$ iff
 $\beta\succ\beta_0$. 
 \begin{itemize}[leftmargin=*]
 \item ($\alpha\succcurlyeq\beta_0$ implies 
 $\bar{\kcal},\alpha\Vdash A_{\beta_0}$): It is enough to show that 
 $\bar{\kcal},\beta_0\Vdash A_{\beta_0}$. Then for evey $\alpha\succcurlyeq\beta_0$ we have $\bar{\kcal},\alpha\Vdash A_{\beta_0}$,
 as desired.   By definition of $\bar{\kcal}$,
 we have $\bar{\kcal},\beta_0\Vdash p_{_{\beta_0}}$. 
 Consider some $\gamma\sqsupseteq\beta_0$. 
 Again by definition of $\bar{\kcal}$,
 we have $\bar{\kcal},\beta_0\nVdash \Box\neg p_{_\gamma}$ and 
for every $\delta\succ\beta_0$ we have 
 $\bar{\kcal},\delta\Vdash A^+_{\beta_0}$. Hence 
 $\bar{\kcal},\beta_0\Vdash \Box\neg p_{_{\gamma}}\to A^+_{\beta_0}$.
 This argument shows that $\bar{\kcal},\beta_0\Vdash A_{\beta_0}$, as desired.
 \item ($\bar{\kcal},\alpha\Vdash A_{\beta_0}$ implies  $\alpha\succcurlyeq\beta_0$): Let $\bar{\kcal},\alpha\Vdash A_{\beta_0}$. 
 Since $\bar{\kcal},\alpha\Vdash p_{_{\beta_0}}$, we have 
 $\beta_0(\preccurlyeq\cup\R)\alpha$. If 
 $\beta_0\preccurlyeq\alpha$, we are done. So let 
 $\beta_0\not\preccurlyeq \alpha$ and $\beta_0\R\alpha$. 
 Hence for arbitrary $\gamma\sqsupseteq\beta_0$ we have 
 $\bar{\kcal},\alpha\Vdash \neg\Box\neg p_{_\gamma}$. This 
 by ${\sf Suc}$-classicality, 
 implies that there is some $\delta\sqsupset\alpha$ 
 such that 
 $\bar{\kcal},\delta\Vdash p_{_\gamma}$. Then we have 
 $\gamma(\preccurlyeq\cup\R)\delta$. 
 By ${\sf Suc}$-classicality, 
 we have $\gamma\sqsubseteq\delta$. Since $\bar{\kcal}$ is with tree
 frame, we have either $\alpha\sqsubseteq\gamma$ or 
 $\gamma\sqsubseteq\alpha$. On the other hand, 
 since $\bar{\kcal}$ is $\R$-branching, 
 there must be some $\gamma\sqsupset\beta_0$ which 
 is $\R$-incomparable with $\alpha$, a 
 contradiction with our previous argument. \qedhere
\end{itemize}
\end{proof}
\begin{theorem}\label{Reduction-PA-HA-to-Sigma-PA-HA}
$\iglphat=\PL(\PA,\HA)\leq \PLS(\PA,\HA)=\iglphatsigma$.
\end{theorem}
\begin{proof}
The arithmetical soundness of $\iglphat$ is straightforward
and left to the reader.  Also by \Cref{Reduction-Sigma-(PA-to-HA)-HA}  we have $\PLS(\PA,\HA)=\iglphatsigma$. 
It remains to show 
$$\ac{\iglphat}{\PA}{\HA}\leq_{f,\fbar[]}\acs{\iglphat}{\PA}{\HA}$$
Let $A\in\lcalb$ 
such that $\iglphat\nvdash A$.
 Then  by \Cref{Lemma-Reduction-iglphat} 
 there is some substitution $\tau$ such that 
$
 \iglphatsigma\nvdash \tau(A) 
$.
Define the function $f$ as follows:
$$f(A):=\begin{cases}
\tau(A) &:\iglphat\nvdash A\\
\text{whatever you like} &:\text{otherwise}
\end{cases}$$
Also let $\fbar(\sigma):=\sigma_{_{\sf PA}}\circ\tau$.
Then one may easily observe that R0, R1 and R3 holds for this 
$f,\fbar[]$.
\end{proof}

\begin{lemma}\label{Lem-reduction-GLsigma-iglphat}
For $A\in\lcalb$, if  $\iglphatsigma\vdash A^\negout$ 
then  
${\GLV}\vdash A$%
.
\end{lemma}
\begin{proof}
Let $\iglphatsigma\vdash A^\negout$. Then ${\GLV}\vdash A^\negout$ and since $A^\negout$ is classically equivalent to $A$ 
we have ${\GLV}\vdash A$. 
%
\end{proof}
\begin{theorem}\label{Reduction-Sigma-PA-(PA-to-HA)}
${\GLV}=\PLS(\PA,\PA)\leq\PLS(\PA,\HA)=\iglphatsigma$.
\end{theorem}
\begin{proof}
By  \Cref{Solovey,Reduction-Sigma-(PA-to-HA)-HA}
 we have ${\GLV}=\PLS(\PA,\PA)$ and 
$\PLS(\PA,\HA)=\iglphatsigma$.
We must show
$\acs{{\GLV}}{\PA}{\PA}\leq_{f,\fbar[]} \acs{\iglphatsigma}{\PA}{\HA}$.  
Given $A\in\lcalb$, define $f(A):=A^\negout$ and $\fbar$ as identity function.
\begin{itemize}
\item[R1.] If $\iglphatsigma\vdash A^\negout$, then 
by \Cref{Lem-reduction-GLsigma-iglphat} we have 
${\GLV}\vdash A$. 
\item[R2.] Holds by \Cref{Lemma-neg-translate-1st-order,Lemma-HA-PA-neg-translation}. \qedhere
\end{itemize}
\end{proof}


\begin{theorem}\label{Reduction-Sigma-PA-(HA-to-nat)}
$\iglphatsigma=\PLS(\PA,\HA)\leq\PLS(\PA,\nat)={\GLSV}$.
\end{theorem}
\begin{proof}
By \Cref{Solovey,Reduction-PA-HA-to-Sigma-PA-HA} 
 we have 
$\PLS(\PA,\nat)={\GLSV}$ and $\iglphatsigma=\PLS(\PA,\HA)$.
We must show
$\acs{\iglphat}{\PA}{\HA}\leq_{f,\fbar[]} \acs{{\GLSV}}{\PA}{\nat}$.  
Given $A\in\lcalb$, define $f(A)=\Box A$ and $\fbar$ as identity function.  
\begin{itemize}
\item[R1.] Let ${\GLSV}\vdash \Box A$. By soundness
of $\ihatHSsigma=\PLS(\HA,\nat)$, for every $\Sigma_1$-substitution $\sigma$
we have $\nat\models \sigma_{_{\sf HA}}(\Box A)$ and hence 
$\HA\vdash \sigma_{_{\sf HA}}(A)$. Then by arithmetical 
completeness of $\PLS(\HA,\HA)$, we have $\lles\vdash A$.\\
One also may prove this item with a direct propositional argument. For simplicity reasons, we chose the indirect way.
\item[R2.] Let $\nat\not\models \sigma_{_{\sf HA}}(\Box A)$. 
Then $\HA\nvdash \sigma_{_{\sf HA}}(A)$, as desired. \qedhere
\end{itemize}
\end{proof}

\begin{theorem}\label{Reduction-Sigma-PA-(PA-to-nat)}
${\GLV}=\PLS(\PA,\PA)\leq\PLS(\PA,\nat)={\GLSV}$.
\end{theorem}
\begin{proof}
By  \Cref{Solovey}
 we have ${\GLV}=\PLS(\PA,\PA)$ and 
$\PLS(\PA,\nat)={\GLSV}$.
We must show
$\acs{{\GLV}}{\PA}{\PA}\leq_{f,\fbar[]} \acs{{\GLSV}}{\PA}{\nat}$.  
Given $A\in\lcalb$, define $f(A)=\Box A$ and $\fbar$ as identity function.  
\begin{itemize}
\item[R1.] Let ${\GLSV}\vdash \Box A$. By soundness
of ${\GLSV}=\PLS(\PA,\nat)$, for every $\Sigma_1$-substitution $\sigma$
we have $\nat\models \sigma_{_{\sf PA}}(\Box A)$ and hence 
$\PA\vdash \sigma_{_{\sf PA}}(A)$. Then by arithmetical 
completeness of $\PLS(\PA,\PA)$, we have ${\GLV}\vdash A$.
\\
One also may prove this item with a direct propositional argument, using Kripke semantics. For simplicity reasons, we chose the indirect way.
\item[R2.] Let $\nat\not\models \sigma_{_{\sf PA}}(\Box A)$. 
Then $\PA\nvdash \sigma_{_{\sf PA}}(A)$, as desired. \qedhere
\end{itemize}
\end{proof}

\begin{theorem}\label{Reduction-PA-(PA-to-nat)}
$\GL=\PLS(\PA,\PA)\leq\PLS(\PA,\nat)=\GLS$.
\end{theorem}
\begin{proof}
Similar to the proof of \Cref{Reduction-Sigma-PA-(PA-to-nat)}
and left to the reader. 
\end{proof}

\section{Relative provability logics for ${\sf PA}^*$}\label{section-PA*}

In this section, we characterize several relative provability 
logics for $\PA^*$ via reductions.   
All reductions are shown at once in the diagram 
\ref{Diag-PA^*-reductions}. 
The head of arrow reduces to its tail, 
via some simple reduction (\cref{Section-Reduction-tool}). 
The translation $f$ in the 
reduction, is shown over the arrow lines and the number 
which appears under arrow, is the corresponding theorem.

\begin{figure}[h]
\[ \hspace*{-1cm}
\begin{tikzcd}[column sep=3em, row sep=4em] 
\tcboxmath{\PL(\PA^*,\HA)}
\arrow[d, "\ref{Reduction-(-to-sigma)-PA*-HA}", "\tau" ']
&
\tcboxmath{\PL(\PA^*,\PA)}
\arrow[d,"\ref{Reduction-(-to-sigma)-PA*-PA}", "\tau" ']
&
\tcboxmath{\PL(\PA^*,\PA^*)}
\arrow[r,"\ref{Reduction-PA*-(PA*-to-nat)}"',"\Box(.)"]
\arrow[d,"\ref{Reduction-(-to-sigma)-PA*-PA*}","\tau"']
&
\tcboxmath{\PL(\PA^*,\nat)}
\arrow[d,"\ref{Reduction-(-to-sigma)-PA*-nat}","\tau"']
\\
\tcboxmath{\PLS(\PA^*,\HA)}
\arrow[d, "\ref{Reduction-sigma-(PA*-to-PA)-HA}", "(.)^\Boxin" ']
&
\tcboxmath{\PLS(\PA^*,\PA)}
\arrow[l, "\ref{Reduction-sigma-PA*-(PA-to-HA)}", "(.)^\negout" ']
\arrow[d, "\ref{Reduction-sigma-(PA*-to-PA)-PA}", "(.)^\Boxin" ']
&
\tcboxmath{\PLS(\PA^*,\PA^*)}
\arrow[l, "\ref{Reduction-Sigma-PA*-(PA*-to-PA)}", "(.)^\Boxout" ']
\arrow[r,"\ref{Reduction-Sigma-PA*-(PA*-to-nat)}" ', " \Box(.)" ]
&
\tcboxmath{\PLS(\PA^*,\nat)}
\arrow[d,"\ref{Reduction-sigma-(PA*-to-PA)-nat}", "(.)^\Boxin" ']
\\
 \tcboxmath{\PLS(\PA,\HA) }
& 
 \tcboxmath{\PLS(\PA,\PA) }
 \arrow[rr,   "\Box (.)",  " \ref{Reduction-Sigma-PA-(PA-to-nat)}" ' ]
 \arrow[l,"(.)^\negout"  ' , "\ref{Reduction-Sigma-PA-(PA-to-HA)}" ] 
&  &
 \tcbhighmath{\PLS(\PA,\nat) }
\end{tikzcd}
\]
\caption[Diagram of  reductions for relative provability logics of $\PA^*$]{Reductions for relative provability logics of $\PA^*$ \label{Diag-PA^*-reductions}}
\end{figure}
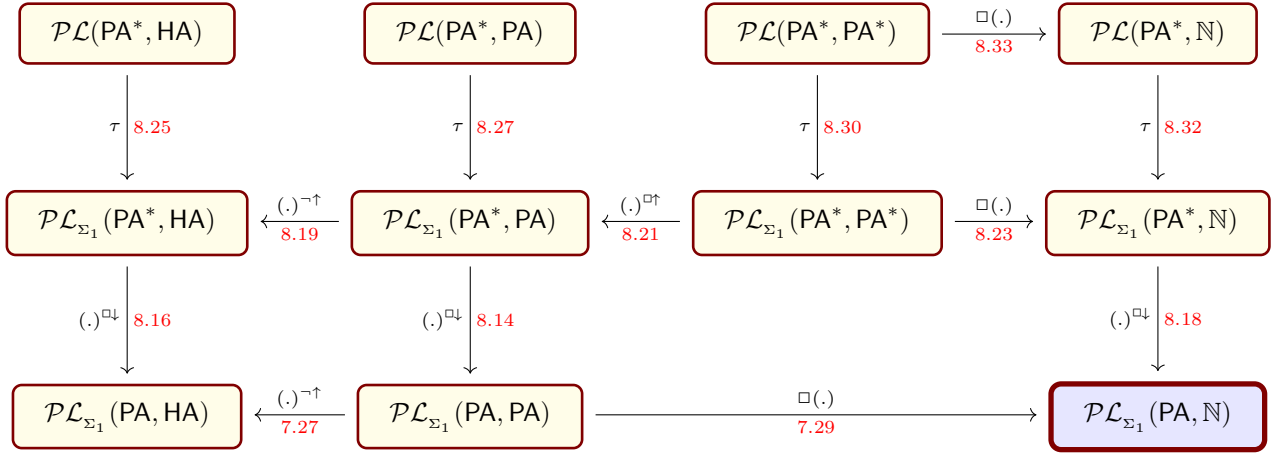

\subsection{Kripke Semantics}

In the following lemma, we will show that the axioms $\CP$ and $\TP$ are local over $\iGL$, i.e. whenever we can deduce some 
proposition $A$ from $\CP+\TP$ in $\iGL$, then we may deduce it 
by those instances of $\CP$ and $\TP$ which use the subformulas 
of $A$:  
\begin{lemma}\label{Lem-iglct-subformule}
For every $A$, if $\iglct\vdash A$ then 
\begin{equation*}
\igl\vdash 
\Boxdot\left[\bigwedge_{E\to F\in {\sf sub}(A)}\Box(E\to F)\to 
(E\vee (E\to F))
\wedge
\bigwedge_{E\in {\sf sub}(A)} (E\to \Box E)\right]
\to A
\end{equation*}
\end{lemma}
\begin{proof}
For the simplicity of notations, in this proof,
 let $$\varphi:=\Boxdot\bigwedge_{E\to F\in {\sf sub}(A)}
 \Box(E\to F)\to 
(E\vee (E\to F))\wedge 
\Boxdot\bigwedge_{E\in {\sf sub}(A)} (E\to \Box E)$$ and 
$\vdash$ indicates derivablity in $\iGL+\varphi$.\\
One side is trivial. For the other way around, assume that 
$ \igl\nvdash \varphi \to A$. 
We will construct some finite Kripke model 
$\kcal=(K,\R,\preccurlyeq,V)$  with $\sqsubset=\prec$
such that 
$\kcal,\alpha\nVdash A$, 
which by soundness of $\iglct$ for finite  Kripke models with 
$\prec=\sqsubset$, 
we have 
the desired result. The proof is almost identical to the proof 
of \Cref{Theorem-Propositional Completeness LC} in 
\cite[Theorem 4.26]{Sigma.Prov.HA}. To be self-contained, 
we elaborate  it here.

Let ${\sf Sub}(A)$ be the set of sub-formulae of
$A$. Then define  $$X:=\{B,\Box B \ | \ B\in {\sf Sub}(A)\}$$ It is
obvious that $X$ is a finite adequate set. We define
$\mathcal{K}=(K,\hlf{\preccurlyeq},\R,V)$ as follows. 
Define 
\begin{itemize}
\item $K$ as the set of
all $X$-saturated sets with respect to $\igl+\varphi$, 
\item $\alpha\R\beta$ iff $\{D: \Box D\in \alpha\}\subseteq \beta$
 and $\alpha\subsetneqq \beta$,
 \item $\alpha\preccurlyeq\beta$ iff  $\alpha\R\beta$ or $\alpha=\beta$,
 \item $\alpha V
p$ iff $p\in \alpha$,  for atomic $p$. 
\end{itemize}

$\mathcal{K}$ trivially satisfies all the properties of 
finite 
Kripke model with $\R\,=\,\prec$. So we must only show that 
$\kcal\nVdash A$. To this end,
we first show by
induction on $B\in X$ that $B\in \alpha$ iff 
$\alpha\Vdash B$, for each
$\alpha\in K$. 
The only non-trivial cases are 
 $B=\Box C$ and $B= E\to F$. 
\begin{itemize}
\item $B=\Box C$:
Let
$\Box C\not\in \alpha$. 
We must show $\alpha\nVdash \Box C$. The other
direction is easier to prove and we leave it to reader. Let
$\beta_0:=\{D\in X\ |\ \alpha\vdash\Box D\}$. 
If $\beta_0,\Box C\vdash
C$, since by definition of $\beta_0$, 
we have $\alpha\vdash\Box \beta_0$ and
hence by L\"{o}b's axiom,  $\alpha\vdash \Box C$, which is in contradiction with
$\Box C\not\in \alpha$. 
Hence $\beta_0,\Box C\nvdash C$ and so there
exists some $X$-saturated set $\beta$ such that 
$\beta\nvdash C$,
$\beta\supseteq \beta_0\cup\{\Box C\}$. 
Hence $\beta\in K$ and $\alpha\R \beta$.
Then by \hl{the induction hypothesis}, 
$\beta\nVdash C$ and hence $\alpha\nVdash
\Box C$.
\item 
 Let $E\to F\not\in \alpha$. Then $F\not\in \alpha$. 
 If $E\in \alpha$, by induction hypothesis we have 
 $\alpha\Vdash E$ and $\alpha\nVdash F$, and hence 
 $\alpha\nVdash E\to F$, as desired. 
 So we may let $E\not\in \alpha$.
Define 
 $\beta_0:=\{D: \Box D\in\alpha \} $. If $\alpha\vdash \bigwedge \beta_0\to (E\to F)$, then $\alpha\vdash \Box (E\to F)$ and hence by $\TP$, either we have $\alpha\vdash E$ or 
 $\alpha\vdash E\to F$, a contradiction.
 So we may let $\alpha\nvdash \bigwedge\beta_0\to (E\to F)$, 
 and use  \Cref{Lemma-saturation} to find
 $\beta\supseteq \beta_0\cup\alpha\cup\{E\}$ as some 
 $X$-saturated node in $\kcal$. Hence $\alpha\R \beta$ 
 which implies $\alpha\prec \beta$ and by 
 induction hypothesis $\beta\Vdash E$ and $\beta\nVdash F$, which implies $\alpha\nVdash E\to F$, as desired.

\end{itemize} 
 
\noindent Since $\igl+\varphi\nvdash A$, by \Cref{Lemma-saturation}, there exists some
$X$-saturated set $\alpha\in K$ such that $\alpha\nvdash A$, and hence by
the above argument we have $\alpha\nVdash A$.
\end{proof}

\begin{lemma}\label{Lem-60}
For arbitrary 
proposition $A$
\begin{center}
$\iglct\vdash A^\Box$ implies  
$\igl+ \Box \CP+\TP\vdash A^\Box$.
\end{center}
\end{lemma}
\begin{proof}
Let $\iglct\vdash A^\Box$. Hence by \Cref{Lem-iglct-subformule}
the following is derivable in $\igl$
$$\left[\Boxdot\overbrace{\bigwedge_{B\in {\sf sub}(A^\Box)} B\to\Box B}^{G}\wedge 
\Boxdot
\overbrace{\bigwedge_{E\to F\in {\sf sub}(A^\Box)}\Box(E\to F)\to 
(E\vee 
(E\to F))}^{H}\right]\to A^\Box
$$
Hence by \Cref{Lemma-3}
$\iGL\vdash G^\Boxout\wedge H^\Boxout\to A^\Box$. 
By \Cref{Lemma-4-2}, ${\sf iK4}\vdash  G^\Boxout$. Also $(\Box G)^\Boxout=\Box G$ which is an instance of $\Box\CP$.
Let us consider some arbitrary conjunct $\Box(E\to F)\to (E\vee 
(E\to F))$
in $H$. Since $E\to F$ is a subformula of $A^\Box$, we have 
$E=E_0^\Box$ and $F=F_0^\Box$. Hence inside ${\sf iK4}$, 
the $H^\Boxout$ is equivalent to some instance of $\TP$. 
Hence $\igl+\Box\CP+\TP\vdash A^\Box$.
\end{proof}

\begin{theorem}\label{Theorem-Kripke-completeness-iglchatthat}
 $\iglchatthat$ is sound and complete for 
semi-perfect ${\sf Suc}$-quasi-classical 
Kripke models.
\end{theorem}
\begin{proof}
The soundness is easy and left to the reader. 
For the 
completeness, we first show the completeness for finite  
brilliant  irreflexive transitive  
${\sf Suc}$-quasi-classical Kripke models. 
Let $\iglchatthat\nvdash A$. Let
$$X:=\{ B,\Box B: B\in{\sf Sub}(A)\}$$
and define the 
Kripke model $\kcal=(K,\preccurlyeq,\R,V)$  as follows:
\begin{itemize}
\item $K$ is the family of all $X$-saturated sets 
with respect to $\iglchatthat$.
\item $\alpha\R\beta$ iff $\alpha\neq\beta$ and 
$\{B ,\Box B: \Box B\in \alpha\}
\subseteq\beta$ and  $\beta$ is $X$-saturated with respect to 
$\iglct$.
\item $\alpha\prec  \beta$  iff  $\alpha\subsetneqq\beta$ and 
either 
 $\alpha\sqsubset \beta$   or 
$\gamma\not\sqsubset\alpha$, for every $\gamma\in K$.
\item $\alpha\,V\,p$ iff $p\in \alpha$.
\end{itemize}

 It is straightforward to show that $\kcal$ 
is a finite brilliant irreflexive  transitive  
${\sf Suc}$-quasi-classical Kripke model. 
We leave them to the reader.
 We only show that $\kcal,\alpha\Vdash B$ iff 
$B\in \alpha$ for every $\alpha\in K$ and $B\in X$. 
Then by \Cref{Lemma-saturation} one may  find some 
$\alpha\in K$ such that $\kcal,\alpha\nVdash A$, as desired.

 Use induction on the complexity of $B\in X$. 
 All inductive steps are trivial, except for:
 \begin{itemize}[leftmargin=*]
 \item $B=\Box C$:  If $\Box C\in \alpha$ and $\alpha\R\beta$, then by definition, 
 $C\in \beta$ and hence by induction hypothesis 
 $\beta\Vdash C$.
 This implies $\alpha\Vdash \Box C$. For the other way around,
 let $ \Box C\not\in \alpha$. 
 Consider the set 
 $\Delta:=\{E,\Box E:\Box E\in \alpha\}$.  If 
 $\iglct\vdash\bigwedge\Delta\to(\Box C\to C)$, then 
 $\iglchatthat\vdash \Box(\bigwedge \Delta)\to \Box C$. Since
 ${\sf iK4}+\alpha\vdash \Box(\bigwedge \Delta)$, we may deduce 
 $\iglchatthat+\alpha\vdash\Box C$, a contradiction. 
 Hence $\iglct\nvdash (\bigwedge \Delta\wedge \Box C)\to C$. 
 By  \Cref{Lemma-saturation}, there exists some $X$-saturated 
 set $\beta\supseteq \Delta\cup\{\Box C\}$ with respect to 
 $\iglct$ such that $C\not\in \beta$. Hence $\beta\in K$ and 
 $\alpha\R\beta$ and $C\not\in \beta$. Induction hypothesis implies that 
 $\beta\nVdash C$ and hence $\alpha\nVdash\Box C$. 
\item $B=C\to D$: If $ C\to D\in \alpha$ and $\alpha
\preccurlyeq\beta$ and $\beta\Vdash C$, by induction hypothesis 
$C\in \beta$ and hence $D\in\beta$. Again by induction 
hypothesis we have $\beta\Vdash D$. This shows that $\alpha\Vdash C\to D$. 
For the other way around, let $C\to D\not\in \alpha$. We have two cases:
\begin{itemize}[leftmargin=*]
\item There is some $\gamma\R\alpha$: Hence 
$\alpha$ is $X$-sturated w.r.t~$\iglct$. 
Let $\Delta:=\{E:\Box E\in \alpha\}$. 
We have tow subcases:
\begin{itemize}[leftmargin=*]
\item If $\iglct+\Delta+\alpha\vdash C\to D$, then $\iglct+\Box \alpha+\Box \Delta\vdash \Box (C\to D)$. By the completeness principle, we have 
$\iglct+\alpha\vdash \Box (C\to D)$. By $\TP$ we have 
$\iglct+\alpha\vdash C\vee (C\to D)$. Since 
$\alpha$ is $X$-saturated with respect to $\iglct$, we have 
either $C\in\alpha$ or $C\to D\in\alpha$. The latter is 
impossible, hence $C\in\alpha$. Again by $X$-saturatedness
 of $\alpha$, we can deduce $D\not\in\alpha$. Hence 
by induction hypothesis we have $\alpha\Vdash C$ and 
$\alpha\nVdash D$, which implies $\alpha\nVdash C\to D$, as 
desired. 
\item If $\iglct+\Delta+\alpha\nvdash C\to D$, then by 
\Cref{Lemma-saturation} there exists some $X$-saturated 
$\beta\supseteq \alpha\cup\Delta\cup\{C\}$ w.r.t~$\iglct$ 
(and a fortiori $\iglchatthat$) such that $D\not\in\beta$. 
Induction hypothesis implies $\beta\Vdash C$ and 
$\beta\nVdash D$. One may 
observe that  $\alpha\prec \beta$ or $\alpha=\beta$, and hence 
$\alpha\nVdash C\to D$, as desired.
\end{itemize}
\item There is no $\gamma\R\alpha$: since $\iglchatthat+\alpha\nvdash C\to D$, by \Cref{Lemma-saturation}, there exists some
$X$-saturate set 
$\beta\supseteq \alpha\cup\{C\}$ with respect to $\iglchatthat$
such that $D\not\in \beta$. Hence by induction hypothesis 
$\beta\Vdash C$ and $\beta\nVdash D$. One may observe that 
$\beta\succcurlyeq\alpha$ and hence $\alpha\nVdash C\to D$. 
\end{itemize}
\end{itemize}

Next we use the construction method \cite{IemhoffT} to fulfil the other conditions: being neat and tree.  Let $\kcal_t:=(K_t,\preccurlyeq_t,\R_t,V_t)$
as follows:
\begin{itemize}
\item $K_t$ is the set of all finite sequences $r:=\langle\alpha_0,\ldots\alpha_n\rangle$ such that for any $i<n$  either we have $\alpha_i\prec\alpha_{i+1}$ or $\alpha_i\R\alpha_{i+1}$.  Let $f(r)$ indicates the final element of the sequence $r$.
\item $r\preccurlyeq_t s$ iff $r$ is an initial segment of $s$ and $f(r)\preccurlyeq f(s)$.
\item $r\R_t s$ iff $r$ is an initial segment of $s=\langle\alpha_0,\ldots\alpha_n\rangle$, e.g. $r=\langle\alpha_0,\ldots\alpha_k\rangle$  for some $k<n$ and $ \alpha_i\R\alpha_{i+1}$ for some $k\leq i<n$.
\item $r\, V_t \,p$ iff $f(r)\,V \,p$.
\end{itemize}
It is straightforward to show that $\kcal_t$ is semi-perfect 
${\sf Suc}$-quasi-classical Kripke model and for every $r\in K_t$ and formula $B$ we have 
\[\kcal_t,r\Vdash B \quad \quad \Longleftrightarrow \quad\quad 
\kcal, f(r)\Vdash B.\qedhere\]
\end{proof}

\begin{theorem}\label{Theorem-Kripke-completeness-iglchatthatsigma}
 $\iglchatthatsigma$ is sound and complete for 
semi-perfect ${\sf Suc}$-quasi-classical  atom-complete
Kripke models.
\end{theorem}
\begin{proof}
Similar to the proof of \Cref{Theorem-Kripke-completeness-iglchatthat} and left to the reader.
\end{proof}
\begin{theorem}\label{Theorem-Kripke-completeness-ihatglchatt}
For every proposition $A$, we have 
$\ihatglchatt\vdash A$ iff for every quasi-classical 
perfect   Kripke model $\kcal$ and every boolean interpretation 
$I$  and arbitrary node $\alpha$ in $\kcal$ 
we have  $\kcal,\alpha,I\models A$.
\end{theorem}
\begin{proof}
The soundness is easy and left to the reader. For the 
completeness part, let $\ihatglchatt\nvdash A$.  
Let $A'$ be 
a boolean equivalent of $A$ which is a conjunction of implications 
$E\to F $
in which $E$ is a conjunction of a set of atomics or boxed propositions 
and $F$ is a disjunction of atomics or boxed proposition. 
Evidently such $A'$ exists for every $A$.
Hence $\ihatglchatt\nvdash A'$. 
Then there must be some 
conjunct $E\to F$ of 
$A'$ such that $\ihatglchatt\nvdash E\to F$, $E$ is a conjunction of atomic and boxed propositions
 and $F$ is a disjunction of atomic and  
boxed propositions. 
Let $X_E$ be the set of atomic conjuncts in $E$ and $X_F$
the set of atomic disjuncts in $F$. Note here that $X_E$ and $X_F$ are disjoint sets. 
Define $\bar{E}$
and $\bar{F}$ as the replacement of $X_E$ and $X_F$ by 
$\top$ and $\bot$ in $E$ and $F$, respectively.
Hence $\bar{E}\to\bar{F}$, does not have any outer atomics and 
then $(\bar{E}\to\bar{F})^\Box$ is equivalent in 
$\igl+\Box \CP$ with $\bar{E}\to\bar{F}$. Then  
$\iGL+\TP+\Box\CP\nvdash (E\to F)^\Box$
 and by \Cref{Lem-60}
we have $\iglct\nvdash \bar{E}\to\bar{F}$. Then by 
\Cref{Theorem-Propositional Completeness LC}, there is some 
perfect, quasi-classical Kripke model $\kcal$ such that 
$\kcal,\alpha\nVdash \bar{E}\to\bar{F}$. Hence there is some 
$\beta\succcurlyeq \alpha$ such that $\kcal,\beta\Vdash \bar{E}
$ and $\kcal,\beta\nVdash \bar{F}$. Let the boolean 
interpretation $I$ defined such that:
$$
I(p):=
\begin{cases}
\text{true} \quad &:p\in X_E\\
\text{false} &: p\in X_F\\
\text{no matter, true or false} &:\text{otherwise}
\end{cases}
$$
One may observe that 
$\kcal,\beta,I\not\models E\to F$ and hence 
$\kcal,\beta,I\not\models A$. 
\end{proof}

\begin{theorem}\label{Theorem-Kripke-completeness-ihatglchattsigma}
For every proposition $A$, we have 
$\ihatglchattsigma\vdash A$ iff for every quasi-classical 
perfect   Kripke model $\kcal$ and arbitrary node $\alpha$ in $\kcal$ 
we have  $\kcal,\alpha\models A$.
\end{theorem}
\begin{proof}
The  proof is very similar to the one for 
\Cref{Theorem-Kripke-completeness-ihatglchatt}, except for the 
argument for $X_E$ and $X_F$ and $\bar{E}$ and $\bar{F}$ and the boolean interpretation $I$, which are unnecessary here with the presence of the ${\sf CP_a}$. For readability reasons, we bring the adapted proof here.

The soundness is straightforward and left to the reader. 
For the completeness, let $\ihatglchattsigma\nvdash A$.  
Let $A'$ be 
a boolean equivalent of $A$ which is a conjunction of implications 
$E\to F $
in which $E$ is a conjunction of a set of atomics or boxed propositions 
and $F$ is a disjunction of atomics or boxed proposition. 
Evidently such $A'$ exists for every $A$.
Hence $\ihatglchattsigma\nvdash A'$. 
Then there must be some 
conjunct $E\to F$ of 
$A'$ such that $\ihatglchatt\nvdash E\to F$, $E$ is a conjunction of atomic and boxed propositions
 and $F$ is a disjunction of atomic and  
boxed propositions. Hence
${E}^\Box\to {F}^\Box$ is equivalent in 
${\sf iK4}+{\sf CP_a}+\Box \CP$ with ${E}\to{F}$. Then  
$\iGL+\TP+\Box\CP+{\sf CP_a}\nvdash (E\to F)^\Box$
 and by \Cref{Lem-60}
we have $\iglct\nvdash {E}\to {F}$. Then by 
\Cref{Theorem-Propositional Completeness LC}, there is some 
perfect, quasi-classical Kripke model $\kcal$ such that 
$\kcal,\alpha\nVdash {E}\to{F}$. Hence there is some 
$\beta\succcurlyeq \alpha$ such that $\kcal,\beta\Vdash {E}
$ and $\kcal,\beta\nVdash {F}$. 
Then
$\kcal,\beta\not\models E\to F$ and hence 
$\kcal,\beta\not\models A$. 
\end{proof}

\begin{theorem}\label{Theorem-Kripke-completeness-ihatglchatts}
$\ihatglchatts\vdash A^\Boxin$ iff for every quasi-classical 
perfect  Kripke model $\kcal$ and every boolean interpretation 
$I$  and arbitrary $A^\Boxin$-sound node $\alpha$ in $\kcal$ 
we have  $\kcal,\alpha,I\models A^\Boxin$.
\end{theorem}
\begin{proof}
Both directions are   proved contra-positively. 
For the soundness part, assume that 
$\kcal,\alpha,I\not\models A^\Boxin$ for some boolean interpretation $I$ and  quasi-classical
 perfect Kripke
model $\kcal:={(K,\preccurlyeq,\R,V)}$  
which  is  $A^\Boxin$-sound at  $\alpha\in K$.
Since derivability is finite,  
it is enough to show that for every finite set 
$\Gamma$ of modal propositions we have 
$$\ihatglchatt\nvdash \bigwedge_{B\in \Gamma}(\Box B^\Box\to B^\Box)\to A^\Boxin.$$
By \Cref{Theorem-Kripke-completeness-ihatglchatt} 
and \Cref{Lemma-10}, it is enough to
find some  number $i$ 
such that 
$$ \kcal^{(i)},\alpha_i,I\not\models \bigwedge_{B\in \Gamma}(\Box B^\Box\to B^\Box)\to A^\Boxin.$$
Let us define $n_i$ and 
$m_i$
as the number of
 propositions in the sets
 $N_i:=\{{B\in \Gamma}: \kcal^{(i)},\alpha_i,I\models B^\Box\wedge\Box B^\Box\}$ and 
 $M_i:=\{B\in \Gamma: 
 \kcal^{(i)},\alpha_i ,I\models \Box B^\Box\wedge \neg B^\Box\}$, respectively. 
We use induction  on $k$ and prove the following statement:
\begin{align}
\nonumber \varphi(k):=\text{for every $i$,} & \text{ if  $n_i<k$ then there is some
$0\leq j\leq 1+n_i$ such that } 
 \\ \label{eq-51}
 &\kcal^{(i+j)},\alpha_{i+j},I\models 
\bigwedge_{B\in \Gamma}(\Box B^\Box\to B^\Box)
\end{align}
Then by $\varphi(n_0+1)$, one may find some number $j$ such that 
$\kcal^{j},\alpha_j,I\models 
\bigwedge_{B\in \Gamma}(\Box B^\Box\to B^\Box) $, and by \Cref{Lemma-10}
we also have $\kcal^{j},\alpha_j,I\not \models 
A^\Boxin $, as desired.

$\varphi(0)$ trivially holds. As induction hypothesis, let 
$\varphi(k)$  holds and show that $\varphi(k+1)$ holds as follows. 
Let some number $i$ such that $n_i<k+1$. If $n_i<k$, by induction hypothesis we have the desired conclusion. So 
let $n_i=k$. If $m_i=0$, we may let $j=0$ and
we have \cref{eq-51}. 
So let 
$B\in \Gamma$ such that 
$\kcal^{(i)},\alpha_i,I\models \Box B^\Box\wedge \neg B^\Box$. We have two sub-cases:
\begin{itemize}
\item $m_{i+1}=0$: observe in this case that \cref{eq-51} holds for $j=1$. 
\item $m_{i+1}>0$: in this case we have $n_{i+1}<k$ and hence by 
application of the induction hypothesis with $i:=i+1$, we get some 
$0\leq j' \leq 1+n_{i+1}$ such that 
$\kcal^{(i+1+j')},\alpha_{i+1+j'}\models 
\bigwedge_{B\in \Gamma}(\Box B^\Box\to B^\Box) $.
Hence if we let $j:=j'+1$ we have $0\leq j\leq 1+ n_i$ and 
\cref{eq-51}, as desired.
\end{itemize}
For the completeness part, assume that 
$\ihatglchatts\nvdash A^\Boxin$. Hence 
$$\ihatglchatt\nvdash 
\left(\bigwedge_{\Box B^\Box\in{\sf Sub}(A^\Boxin)} (\Box B^\Box\to B^\Box)\right)\to A^\Boxin$$
Hence  \Cref{Theorem-Kripke-completeness-ihatglchatt}
implies the desired result.
\end{proof}
\begin{theorem}\label{Theorem-Kripke-completeness-ihatglchattssigma}
$\ihatglchattssigma\vdash A^\Boxin$ iff for every quasi-classical 
perfect  Kripke model $\kcal$ 
and arbitrary $A$-sound node $\alpha$ in $\kcal$ 
we have  $\kcal,\alpha\models A$.
\end{theorem}
\begin{proof}
The proof is similar to the one for \Cref{Theorem-Kripke-completeness-ihatglchatts}.
One must use \Cref{Theorem-Kripke-completeness-ihatglchattsigma}  
instead of \Cref{Theorem-Kripke-completeness-ihatglchatt}
in the proof.
\end{proof}

\subsection{Reductions}

\begin{lemma}\label{Lemma-70}
$\iglct\vdash A$ implies 
$\GL\vdash A^\Box$.
\end{lemma}
\begin{proof}
Use induction on the proof $\iglct\vdash A$.
\end{proof}

\begin{lemma}\label{Lemma-79}
$\GL\vdash A$ implies 
$\iglphat\vdash \Box A$.
\end{lemma}
\begin{proof}
Let $\GL\vdash A$. Hence $\igl\vdash \Boxdot\PEM\to A$. Since 
necessitation is admissible to $\igl$, we have 
$\igl\vdash\Box \PEM\to \Box A$ which implies $\iglphat\vdash \Box A$.
\end{proof}
\begin{definition}\label{Def-tilde}
For a Kripke model $\Kripke$, let $\tilde{\kcal}$, indicates
the Kripke model  derived from $\kcal$ by making 
every $\R$-accessible node  as a classical node. 
More precisely, we define 
$\tilde{\kcal}:=(K,\tilde{\preccurlyeq},\R,V))$ in this way:
\begin{center}
$\alpha\tilde{\preccurlyeq} \beta$ iff ``$\alpha$ is not 
$\R$-accessible ($\alpha\not\in {\sf Suc}$) and $\alpha\preccurlyeq\beta$" or \\
``$\alpha$ is $\R$-accessible ($\alpha\not\in {\sf Suc}$) and $\alpha=\beta$"
\end{center}
\end{definition}
\begin{lemma}\label{Lemma-791}
For every ${\sf Suc}$-quasi-classical semi-perfect Kripke model  $\Kripke$ and $\alpha\not\in {\sf Suc}$ and arbitrary proposition  
$A$ we have 
$$\kcal,\alpha\Vdash A^\Boxin 
\quad \Longleftrightarrow \quad  \tilde{\kcal},\alpha\Vdash A^\Boxin.$$
\end{lemma}
\begin{proof}
First observe that  for every $\alpha\in{\sf Suc}$ and every proposition $B$ we have 
$$\tilde{\kcal},\alpha\Vdash B \quad \Longleftrightarrow \quad  \tilde{\kcal},\alpha\models_c B 
\quad \Longleftrightarrow \quad  {\kcal},\alpha\models_c B$$
Then we may use \Cref{Corol-truth-forcing} and for 
$\alpha\in {\sf Suc}$ deduce 
\begin{equation}\label{Eq-765}
{\kcal},\alpha\Vdash B^\Box\quad \Longleftrightarrow \quad  \tilde{\kcal},\alpha\Vdash B^\Box.
\end{equation}
We use induction on the complexity of $A$ and 
prove the assertion of the lemma. All cases are obvious 
except for the cases $A=\Box B$ in which we have 
$A^\Boxin=\Box B^\Box$.
We have  
\begin{align*}
\kcal,\alpha\nVdash \Box B^\Box 
&\Longleftrightarrow \text{there exists some } \beta\sqsupset\alpha \text{ such that } \kcal,\beta\nVdash B^\Box\\
&\Longleftrightarrow \text{there exists some } \beta\sqsupset\alpha \text{ such that } \tilde{\kcal},\beta\nVdash B^\Box \\
&\Longleftrightarrow \tilde{\kcal},\alpha\nVdash \Box B^\Box
\end{align*}
in which in the second line we use  \cref{Eq-765}.
\end{proof}

\begin{lemma}\label{Lemma-ihatglchatsigma}
For every $A\in\lcalb$ we have $\ihatglchattsigma\vdash A$ iff 
${\GLV}\vdash A^\Boxin$.
\end{lemma}
\begin{proof}
We use induction on the proof  $\ihatglchattsigma\vdash A$ and show 
${\GLV}\vdash A^\Boxin$. All cases are similar to the one 
for $\ihatglchatt$, except for 
 \begin{itemize}
 \item $A=p\to \Box p$: then ${\sf iK4}\vdash A^\Boxin\leftrightarrow A$ and 
 hence ${\GLV}\vdash A^\Boxin$.
\end{itemize} 
For the other way around, let $\ihatglchattsigma\nvdash A$. 
Then by $\Box \CP$  we have 
$A^\Boxin\leftrightarrow A$,  and then we may deduce $\ihatglchattsigma\nvdash A^\Boxin$.
By \Cref{Theorem-Kripke-completeness-ihatglchattsigma}, there exists some 
quasi-classical perfect Kripke model $\kcal$  
such that $\kcal,\alpha\not\models A^\Boxin$. \Cref{Corol-truth-forcing} implies 
$\kcal,\alpha\not\models_c A^\Boxin$, which by soundness of ${\GLV}$ 
for classical Kripke models with the property of truth-ascending (i.e. if $p$
 is true at some node, then it is true also at all accessible nodes), implies 
 ${\GLV}\nvdash A^\Boxin$.
\end{proof}
\begin{theorem}\label{Reduction-sigma-(PA*-to-PA)-PA}
$\ihatglchattsigma=\PLS(\PA^*,\PA)\leq\PLS(\PA,\PA)={\GLV}$.
\end{theorem}
\begin{proof}
The arithmetical soundness of $\ihatglchattsigma$ is straightforward and left to the reader.
Also ${\GLV}=\PLS(\PA,\PA)$ holds by \Cref{Solovey}.  It is enough here to show that 
$$\acs{\iglchattsigma}{\PA^*}{\PA}\leq_{f,\fbar[]}\acs{{\GLV}}{\PA}{\PA}.$$ 
Given $A\in\lcalb$,
let $f(A):=(A)^\Boxin$ and $\fbar$ as identity function. 
\begin{itemize}
\item[R1.]  \Cref{Lemma-ihatglchatsigma}.
\item[R2.] If $\PA\nvdash \sigmapa{A^\Boxin}$,  
for a $\Sigma_1$-substitution $\sigma$, 
then by \Cref{Lemma-PA-Box-translate} we have $\PA\nvdash\sigmapas{A}$.\qedhere
\end{itemize} 
\end{proof}

\begin{lemma}\label{Lemma-iglchatthat}
For every $A\in\lcalb$ we have $\iglchatthatsigma\vdash A$ iff 
$\iglphatsigma\vdash A^\Boxin$.
\end{lemma}
\begin{proof}
We use induction on the proof  
 $\iglchatthatsigma\vdash A$ 
 and show $\iglphatsigma\vdash A^\Boxin$. 
 All cases are identical to the corresponding on in the
  previous proof, except for when $A=p\to \Box p$,
   which trivially we have $\iglphatsigma\vdash A^\Boxin$.

For the other way around, let $\iglchatthatsigma\nvdash A$. Then by 
\Cref{Lemma-A-ABoxin}  we have 
$A^\Boxin\leftrightarrow A$,  and hence
 $\iglchatthatsigma\nvdash A^\Boxin$.
By \Cref{Theorem-Kripke-completeness-iglchatthatsigma}, 
there exists some 
${\sf Suc}$-quasi-classical  
semi-perfect atom-complete  Kripke model $\kcal$ 
such that $\kcal,\alpha\nVdash A^\Boxin$, for some 
node $\alpha$. We may assume
 $\alpha\not\in{\sf Suc}$, otherwise 
eliminate all nodes not in 
$(\alpha\preccurlyeq)\cup(\alpha\R)$ 
and consider this new Kripke model instead of $\kcal$. 
Obviously the new Kripke model still refutes 
$A^\Boxin$ at $\alpha$ 
and is ${\sf Suc}$-quasi-classical  
semi-perfect and atom-complete. 
Hence  \Cref{Lemma-791} implies that 
$\tilde{\kcal},\alpha\nVdash A^\Boxin$, in which 
$\tilde{\kcal}$   indicates
the Kripke model  derived from $\kcal$ by making 
every $\R$-accessible node  as a classical node.  Precise definition of $\tilde{\kcal}$ came before \Cref{Lemma-791}.
It is obvious that $\tilde{\kcal}$ is a ${\sf Suc}$-classical 
semi-perfect atom-complete Kripke model. Hence  
\Cref{Theorem-Kripke-semantic-iglphatsigma} implies 
$\iglphatsigma\nvdash A^\Boxin$, as desired.
\end{proof}
\begin{theorem}\label{Reduction-sigma-(PA*-to-PA)-HA}
$\iglchatthatsigma=\PLS(\PA^*,\HA)\leq\PLS(\PA,\HA)=\iglphatsigma$.
\end{theorem}
\begin{proof}
The arithmetical soundness of $\iglchatthatsigma$ is straightforward and left to the reader.
Also $\iglphatsigma=\PLS(\PA,\HA)$ holds by \Cref{Theorem-Arith-Complete-iglphatsigma-direct}.  It is enough here to show that 
$$\acs{\iglchatthatsigma}{\PA^*}{\HA}\leq_{f,\fbar[]}\acs{\iglphatsigma}{\PA}{\HA}.$$ Given $A\in\lcalb$,
let $f(A):=(A)^\Boxin$ and $\fbar$ as identity function. 
\begin{itemize}
\item[R1.]  \Cref{Lemma-iglchatthat}.
\item[R2.] If $\HA\nvdash \sigmapa{A^\Boxin}$,  
for a $\Sigma_1$-substitution $\sigma$, 
then by \Cref{Lemma-PA-Box-translate} we have $\HA\nvdash\sigmapas{A}$.\qedhere
\end{itemize} 
\end{proof}

\begin{lemma}\label{Lemma-ihatglchattssigma}
For every $A\in\lcalb$ we have $\ihatglchattssigma\vdash A$ iff 
${\GLSV}\vdash A^\Boxin$.
\end{lemma}
\begin{proof}
We use induction on the proof  $\ihatglchattssigma\vdash A$ and show 
${\GLSV}\vdash A^\Boxin$. All cases are similar to the one 
for $\ihatglchatts$, except for 
 \begin{itemize}
 \item $A=p\to \Box p$: then ${\sf iK4}\vdash A^\Boxin\leftrightarrow A$ and 
 hence ${\GLV}\vdash A^\Boxin$.
\end{itemize} 
For the other way around, let $\ihatglchattssigma\nvdash A$. 
Then by $\Box \CP$  we have 
$A^\Boxin\leftrightarrow A$,  and then we may deduce $\ihatglchattssigma\nvdash A^\Boxin$.
By \Cref{Theorem-Kripke-completeness-ihatglchattssigma}, there exists some 
quasi-classical perfect Kripke model $\kcal$  
such that $\kcal,\alpha\not\models A^\Boxin$ and $\kcal$ is $A^\Boxin$-sound at $\alpha$. \Cref{Corol-truth-forcing} implies 
$\kcal,\alpha\not\models_c A^\Boxin$, which by soundness of ${\GLSV}$ 
for classical Kripke models with the property of truth-ascending (i.e. if $p$
 is true at some node, then it is true also at all accessible nodes), implies 
 ${\GLSV}\nvdash A^\Boxin$.
\end{proof}
\begin{theorem}\label{Reduction-sigma-(PA*-to-PA)-nat}
$\ihatglchattssigma=\PLS(\PA^*,\nat)\leq\PLS(\PA,\nat)={\GLSV}$.
\end{theorem}
\begin{proof}
The arithmetical soundness of $\ihatglchattssigma$ is straightforward and left to the reader.
Also $\PLS(\PA,\nat)={\GLSV}$ holds by 
\Cref{Solovey}.  It is enough here to show that 
$$\acs{\ihatglchattssigma}{\PA^*}{\nat}
\leq_{f,\fbar[]}\acs{{\GLSV}}{\PA}{\nat}.$$ 
Given $A\in\lcalb$,
let $f(A):=(A)^\Boxin$ and $\fbar$ as identity function. 
\begin{itemize}
\item[R1.]  \Cref{Lemma-ihatglchattssigma}.
\item[R2.] If $\nat\not\models\sigmapa{A^\Boxin}$,  
for a $\Sigma_1$-substitution $\sigma$, 
then by \Cref{Lemma-PA-Box-translate} we have $\nat\not\models\sigmapas{A}$.\qedhere
\end{itemize} 
\end{proof}


\begin{theorem}\label{Reduction-sigma-PA*-(PA-to-HA)}
$\ihatglchattsigma=\PLS(\PA^*,\PA)\leq\PLS(\PA^*,\HA)=\iglchatthatsigma$.
\end{theorem}
\begin{proof}
We already have 
 $\PLS(\PA^*,\PA)=\ihatglchattsigma$ and 
 $\ihatglchattsigma=\PLS(\PA^*,\PA)$ by 
\Cref{Reduction-sigma-(PA*-to-PA)-HA,Reduction-sigma-(PA*-to-PA)-PA}.  It is enough here to show that 
$\acs{\ihatglchattsigma}{\PA^*}{\PA}\leq_{f,\fbar[]}\acs{\iglchatthatsigma}{\PA^*}{\HA}$. Given $A\in\lcalb$,
let $f(A):=(A)^\negout$ and $\fbar$ as identity function. 
\begin{itemize}
\item[R1.] If  $\iglchatthatsigma\vdash A^\negout$ then  
$\ihatglchattsigma\vdash A^\negout$, and  since 
we have $\PEM$ in $\ihatglchattsigma$, we may conclude 
$\ihatglchattsigma\vdash A$.
\item[R2.] If $\HA\nvdash\sigmapas{A^\negout}$,  
for a $\Sigma_1$-substitution $\sigma$, 
then by  \Cref{Lemma-neg-translate-1st-order} we have $\HA\nvdash(\sigmapas{A})^\neg$. Hence by \Cref{Lemma-HA-PA-neg-translation} we have $\PA\nvdash \sigmapas{A}$. \qedhere
\end{itemize} 
\end{proof}

\begin{lemma}\label{Lemma-ihatglchattsigma-iglct}
For every $A\in\lcalb$, if  
 $\ihatglchattsigma\vdash A^\Boxout$, then $\iglct\vdash A$.
\end{lemma}
\begin{proof}
Let $\iglct\nvdash A$. Hence
by \Cref{Theorem-Propositional Completeness LC}, there is some 
perfect quasi-classical Kripke model $\kcal$ such that
$\kcal,\alpha\nVdash A$. Then 
\Cref{Corol-truth-forcing2} implies 
$\kcal,\alpha\not\models A^\Boxout$, and hence by 
soundness of $\ihatglchattsigma$ (\Cref{Theorem-Kripke-completeness-ihatglchattsigma}) implies $\ihatglchattsigma\nvdash A^\Boxout$.
\end{proof}

\begin{theorem}\label{Reduction-Sigma-PA*-(PA*-to-PA)}
$\iglct=\PLS(\PA^*,\PA^*)\leq\PLS(\PA^*,\PA)=\ihatglchattsigma$.
\end{theorem}
\begin{proof}
The soundness of $\iglct$ is straightforward and left to the reader. 
By \Cref{Reduction-sigma-(PA*-to-PA)-PA}
we have  
$\PLS(\PA^*,\PA)=\ihatglchattsigma$.
We must show
$\acs{\iglct}{\PA^*}{\PA^*}\leq_{f,\fbar[]} \acs{\ihatglchattsigma}{\PA^*}{\PA}$.  
Given $A\in\lcalb$, define $f(A)= A^\Boxout$ 
and $\fbar$ as identity function.  
\begin{itemize}
\item[R1.] \Cref{Lemma-ihatglchattsigma-iglct}.
\item[R2.] Let $\PA\nvdash \sigmapas{ A^\Boxout}$. 
Then by \Cref{Label-HA-HA*-Box-trans2}, 
$\PA\nvdash \sigmapas{A}^\PA$, and hence by definition of $\PA^*$, we have $\PA^*\nvdash \sigmapas{A}$. \qedhere
\end{itemize}
\end{proof}
\begin{corollary}\label{Corollary-iglct-ihatglchattsigma}
For every $A\in\lcalb$, we have  
$\iglct\vdash A$ iff  $\ihatglchattsigma\vdash A^\Boxout$.
\end{corollary}
\begin{proof}
Use \Cref{Corollary-Reduction,Reduction-Sigma-PA*-(PA*-to-PA)}.
\end{proof}
\begin{theorem}\label{Reduction-Sigma-PA*-(PA*-to-nat)}
$\iglct=\PLS(\PA^*,\PA^*)\leq\PLS(\PA^*,\nat)=\ihatglchattssigma$.
\end{theorem}
\begin{proof} 
By \Cref{Reduction-sigma-(PA*-to-PA)-nat,Reduction-Sigma-PA*-(PA*-to-PA)}
we have  
$\PLS(\PA^*,\nat)=\ihatglchattssigma$ and 
$\PLS(\PA^*,\PA^*)=\iglct$.
We must show
$\acs{\iglct}{\PA^*}{\PA^*}\leq_{f,\fbar[]} \acs{\ihatglchattssigma}{\PA^*}{\nat}$.  
Given $A\in\lcalb$, define $f(A)=\Box A$ and $\fbar$ as identity function.  
\begin{itemize}
\item[R1.] Let $\ihatglchattssigma\vdash \Box A$. By soundness
of $\ihatglchattssigma=\PLS(\PA^*,\nat)$, for every $\Sigma_1$-substitution $\sigma$
we have $\nat\models \sigmapas{\Box A}$ and hence 
$\PA^*\vdash \sigmapas{A}$. Then by arithmetical 
completeness of $\iglct=\PLS(\PA^*,\PA^*)$, we have $\iglct\vdash A$.
\\
One also may prove this item with a direct propositional argument, using Kripke semantics. For simplicity reasons, we chose the indirect way.
\item[R2.] Let $\nat\not\models \sigmapas{\Box A}$. 
Then $\PA^*\nvdash \sigmapas{A}$, as desired. \qedhere
\end{itemize}
\end{proof}

\begin{lemma}\label{lemma-iglphat-iglchatthat}
For every $A\in\lcalb$, we have 
$\iglchatthat\vdash A$ iff $\iglphat\vdash A^\Boxin$.
\end{lemma}
\begin{proof}
We use induction on the proof  $\iglchatthat\vdash A$ 
 and show $\iglphat\vdash A^\Boxin$:
 \begin{itemize}[leftmargin=*]
 \item $\igl\vdash A$:  by \Cref{Lemma-3} we have $\igl\vdash A^\Boxin$.
 \item $A$ is an axiom instance of $\Box\CP$ or $\Box \TP$:
 Then $A=\Box B$ and $\iglct\vdash B$ and by
  \Cref{Lemma-70} we have 
 $\GL\vdash B^\Box$. By \Cref{Lemma-79}
  we have $\iglphat\vdash \Box B^\Box$.
 \item $\iglchatthat\vdash B$ and $\iglchatthat\vdash B\to A$ 
 with lower proof length: by induction hypothesis we 
 have $\iglphat\vdash B^\Boxin$ and $\iglphat\vdash B^\Boxin\to A^\Boxin$, 
 which implies 
 $\iglphat\vdash A^\Boxin$, as desired.
\end{itemize} 
For the other way around, let $\iglchatthat\nvdash A$. Then by 
\Cref{Lemma-A-ABoxin}  we have 
$A^\Boxin\leftrightarrow A$,  and hence
 $\iglchatthat\nvdash A^\Boxin$.
By \Cref{Theorem-Kripke-completeness-iglchatthat}, 
there exists some 
${\sf Suc}$-quasi-classical  
semi-perfect Kripke model $\kcal$ 
such that $\kcal,\alpha\nVdash A^\Boxin$, for some 
node $\alpha$. We may let $\alpha\not\in{\sf Suc}$, otherwise 
eliminate all nodes not in 
$(\alpha\preccurlyeq)\cup(\alpha\R)$ 
and consider this new Kripke model instead of $\kcal$. 
Obviously the new Kripke model still refutes 
$A^\Boxin$ at $\alpha$ 
and is ${\sf Suc}$-quasi-classical  
semi-perfect. Hence  \Cref{Lemma-791} implies that 
$\tilde{\kcal},\alpha\nVdash A^\Boxin$, in which 
$\tilde{\kcal}$   indicates
the Kripke model  derived from $\kcal$ by making 
every $\R$-accessible node  as a classical node.  Precise definition of $\tilde{\kcal}$ came before \Cref{Lemma-791}.
It is obvious that $\tilde{\kcal}$ is a ${\sf Suc}$-classical
semi-perfect Kripke model. Hence  
\Cref{Theorem-Kripke-semantic-iglphat} implies $\iglphat\nvdash A^\Boxin$, as desired.
\end{proof}

\begin{theorem}\label{Reduction-(-to-sigma)-PA*-HA}
$\iglchatthat=\PL(\PA^*,\HA)\leq\PLS(\PA^*,\HA)=\iglchatthatsigma$.
\end{theorem}
\begin{proof} 
The arithmetical soundness of $\iglchattsigma$ is 
straightforward and left to the reader. 
By  \Cref{Reduction-sigma-(PA*-to-PA)-HA}
we have  
$\PL(\PA^*,\HA)=\iglchatthatsigma$.
We must show
$$\ac{\iglchatthat}{\PA^*}{\HA}\leq_{f,\fbar[]} \acs{\iglchatthatsigma}{\PA^*}{\HA}.$$  
Given $A\in\lcalb$,   if $\iglchatthat\vdash A$, 
  define $f(A):=\top$. If $\iglchatthat\nvdash A$, 
by \Cref{lemma-iglphat-iglchatthat} we have 
$\iglphat\nvdash A^\Boxin$, and hence 
by \Cref{Lemma-Reduction-iglphat} 
there exists some propositional $(.)^\Boxin$-substitution
 $\tau$ such that 
$\iglphatsigma\nvdash \tau(A^\Boxin)$. 
Define $f(A):=\tau(A)$ and $\fbar (\sigma):=\sigma_{_{\sf PA^*}}\circ \tau$.
\begin{itemize}
\item[R1.] Let $\iglchatthat\nvdash A$. By 
\Cref{lemma-iglphat-iglchatthat} we have $\iglphat\nvdash A^\Boxin$ and then \Cref{Lemma-Reduction-iglphat} 
 implies $\iglphatsigma\nvdash \tau(A^\Boxin)$, in which $\tau$ is as used for the definition of $f(A)$. Since $\tau$  
 is a $(.)^\Boxin$-substitution, 
 by \Cref{Lemma-Boxin-sub}
 we have $\iglphatsigma\nvdash (\tau(A))^\Boxin$. Then 
 \Cref{Lemma-iglchatthat}  implies that 
 $\iglchatthatsigma\nvdash \tau(A)$, or in other words
 $\iglchatthatsigma\nvdash f(A)$.
\item[R2.] Let $\HA\nvdash \sigmapas{f(A)}$ for some 
$\Sigma_1$-substitution $\sigma$.  By definition of $f(A)$, we must have $\iglchatthat\nvdash A$, otherwise $f(A):=\top$, which contradicts $\HA\nvdash \sigmapas{f(A)}$. Hence 
$f(A)=\tau(A)$ for  some propositional $(.)^\Boxin$-substitution $\tau$. 
By 
\Cref{Lemma-Properties of Box translation 2} we have
$\HA\nvdash \sigmapa{\tau(A)^\Boxin}$. Since 
${\sf iK4}+{\sf CP_a}$ is included in 
$\PLS(\PA,\HA)=\iglphatsigma$ (\Cref{Reduction-Sigma-(PA-to-HA)-HA}), we have $\HA\nvdash \sigmapa{\tau(A^\Boxin)}$. 
This implies that $\HA\nvdash [\fbar(\sigma)]_{_{\sf PA}}(A^\Boxin)$ and again by \Cref{Lemma-Properties of Box translation 2} we have $\HA\nvdash [\fbar(\sigma)]_{_{\sf PA^*}}(A)$.
\qedhere
\end{itemize}
\end{proof}

\begin{lemma}\label{lemma-gl-ihatglchatt}
For every $A\in\lcalb$, we have 
$\ihatglchatt\vdash A$ iff $\GL\vdash A^\Boxin$.
\end{lemma}
\begin{proof}
We use induction on the proof  $\ihatglchatt\vdash A$ 
 and show $\GL\vdash A^\Boxin$:
 \begin{itemize}
 \item $A=\Box B$ and $\iglct\vdash B$: by \Cref{Lemma-70} we have 
 $\GL\vdash B^\Box$ and hence by necessitation $\GL\vdash \Box B^\Box$.
 \item $\igl\vdash A$:  by \Cref{Lemma-3} we have $\igl\vdash A^\Boxin$.
 \item $A=B\vee \neg B$: Then $A^\Boxin=B^\Boxin\vee\neg B^\Boxin$ 
 which is valid in $\GL$.
 \item $\ihatglchatt\vdash B$ and $\ihatglchatt\vdash B\to A$ 
 with lower proof length than the one for $A$: by induction hypothesis we 
 have $\GL\vdash B^\Boxin$ and $\GL\vdash B^\Boxin\to A^\Boxin$, 
 which implies 
 $\GL\vdash A^\Boxin$, as desired.
\end{itemize} 
For the other way around, let $\ihatglchatt\nvdash A$. Then by $\Box \CP$  we have 
$A^\Boxin\leftrightarrow A$,  and then 
we may deduce $\ihatglchatt\nvdash A^\Boxin$.
By \Cref{Theorem-Kripke-completeness-ihatglchatt}, there exists some 
quasi-classical perfect Kripke model $\kcal$ and some boolean interpretation $I$
such that $\kcal,\alpha,I\not\models A^\Boxin$. \Cref{Corol-truth-forcing} implies 
$\kcal,\alpha,I\not\models_c A^\Boxin$, which by soundness of $\GL$ for classical Kripke models, implies $\GL\nvdash A^\Boxin$.
\end{proof}

\begin{theorem}\label{Reduction-(-to-sigma)-PA*-PA}
$\ihatglchatt =\PL(\PA^*,\PA)\leq
\PLS(\PA^*,\PA)=\ihatglchattsigma$.
\end{theorem}
\begin{proof} 
The arithmetical soundness of $\ihatglchatt$ is 
straightforward and left to the reader. 
By  \Cref{Reduction-sigma-(PA*-to-PA)-PA}
we have  
$\PLS(\PA^*,\PA)=\ihatglchattsigma$.
We must show
$$\ac{\ihatglchatt}{\PA^*}{\PA}\leq_{f,\fbar[]} 
\acs{\ihatglchattsigma}{\PA^*}{\PA}.$$  
Given $A\in\lcalb$,   if $\ihatglchatt\vdash A$, 
  define $f(A):=\top$. If $\ihatglchatt\nvdash A$, 
by \Cref{lemma-gl-ihatglchatt} we have 
$\GL\nvdash A^\Boxin$, and hence 
by \Cref{Remark-reduction-GL-GLS}
there exists some propositional $(.)^\Boxin$-substitution
 $\tau$ such that 
${\GLV}\nvdash \tau(A^\Boxin)$. 
Define $f(A):=\tau(A)$ and $\fbar (\sigma):=\sigma_{_{\sf PA^*}}\circ \tau$.
\begin{itemize}
\item[R1.] Let $\ihatglchatt\nvdash A$. By 
\Cref{lemma-gl-ihatglchatt} we have $\GL\nvdash A^\Boxin$ and then 
\Cref{Remark-reduction-GL-GLS}
 implies ${\GLV}\nvdash \tau(A^\Boxin)$, 
 in which $\tau$ is as used for the definition of $f(A)$. 
 Since $\tau$   is a $(.)^\Boxin$-substitution, 
 by \Cref{Lemma-Boxin-sub}
 we have ${\GLV}\nvdash (\tau(A))^\Boxin$. Then 
\Cref{Lemma-ihatglchatsigma} implies that 
 $\ihatglchattsigma\nvdash \tau(A)$, or in other words
 $\ihatglchattsigma\nvdash f(A)$.
\item[R2.] Let $\PA\nvdash \sigmapas{f(A)}$ for some 
$\Sigma_1$-substitution $\sigma$.  By definition of $f(A)$, 
we must have $\ihatglchatt\nvdash A$, otherwise $f(A):=\top$,
which contradicts $\PA\nvdash \sigmapas{f(A)}$. Hence 
$f(A)=\tau(A)$ for  some propositional $(.)^\Boxin$-substitution 
$\tau$. By \Cref{Lemma-Properties of Box translation 2} we have
$\PA\nvdash \sigmapa{\tau(A)^\Boxin}$. Since 
${\sf iK4}+{\sf CP_a}$ is included in 
$\PLS(\PA,\PA)={\GLV}$ (\Cref{Solovey}), 
by \Cref{Lemma-Boxin-sub}
we have $\PA\nvdash \sigmapa{\tau(A^\Boxin)}$. 
This implies that $\PA\nvdash [\fbar(\sigma)]_{_{\sf PA}}(A^\Boxin)$ and again by \Cref{Lemma-Properties of Box translation 2} we have $\PA\nvdash [\fbar(\sigma)]_{_{\sf PA^*}}(A)$.
\qedhere
\end{itemize}
\end{proof}

\begin{lemma}\label{lemma-gl-iglct}
For every $A\in\lcalb$, we have 
$\iglct\vdash A$ iff $\GL\vdash A^\Box$.
\end{lemma}
\begin{proof}
One may use induction on the proof  $\iglct\vdash A$ to 
show that ${\sf GL}\vdash A^\Box$. 
For the other direction, we reason contrapositively. 
Let $\iglct\nvdash A$. Since in $\iglc$ we have $A\leftrightarrow A^\Box$, we have $\iglct\nvdash A^\Box$.
 Hence by \Cref{Theorem-Propositional Completeness LC} there is 
 some perfect quasi-classical model $\kcal$ such that 
 $\kcal,\alpha\nVdash A^\Box$. Hence by \Cref{Corol-truth-forcing} $\kcal,\alpha\not\models_c A^\Box$. Since 
 $\models_c$ is just a classical semantics for the modal logic 
 ${\sf GL}$, by the soundness  of ${\sf GL}$ for 
 finite irreflexive Kripke models \cite[Chapter 2.2]{Smorynski-Book},
 we may deduce ${\sf GL}\nvdash A^\Box$, as desired.
\end{proof}

\begin{lemma}\label{lemma-glsigma-iglct}
For every $A\in\lcalb$, we have 
$\iglct\vdash A$ iff ${\GLV}\vdash A^\Box$.
\end{lemma}
\begin{proof}
There are two options (atleast) for the proof. First is that one repeat 
a similar argument of the proof o for \Cref{lemma-gl-iglct}. 
Second proof follows:
By 
\Cref{Corollary-iglct-ihatglchattsigma}, $\iglct\vdash A$ iff 
$\ihatglchattsigma\vdash A^\Boxout$, and 
\Cref{Lemma-ihatglchatsigma} implies $\ihatglchattsigma\vdash A^\Boxout$ iff ${\GLV}\vdash (A^\Boxout)^\Boxin$. Since 
${\sf iK4}\vdash A^\Box\lr (A^\Boxout)^\Boxin$, we have the desired result.
\end{proof}

\begin{theorem}\label{Reduction-(-to-sigma)-PA*-PA*}
$\iglct =\PL(\PA^*,\PA^*)\leq
\PLS(\PA^*,\PA^*)=\iglct$.
\end{theorem}
\begin{proof} 
The arithmetical soundness of $\iglct$  for 
general substitutions, i.e.~$\mathcal{AS}(\iglct;\PA^*,\PA^*)$,
is straightforward and left to the reader. 
By  \Cref{Reduction-Sigma-PA*-(PA*-to-PA)}
we have  
$\PLS(\PA^*,\PA^*)=\iglct$.
We must show
$$\ac{\iglct}{\PA^*}{\PA^*}\leq_{f,\fbar[]} 
\acs{\iglct}{\PA^*}{\PA^*}.$$  
Given $A\in\lcalb$,   if $\iglct\vdash A$, 
  define $f(A):=\top$. If $\iglct\nvdash A$, 
by \Cref{lemma-gl-iglct} we have 
$\GL\nvdash A^\Box$, and hence 
by \Cref{Remark-reduction-GL-GLS}
there exists some propositional $(.)^\Boxin$-substitution
 $\tau$ such that 
${\GLV}\nvdash \tau(A^\Box)$. 
Define $f(A):=\tau(A)$ and $\fbar (\sigma):=\sigma_{_{\sf PA^*}}\circ \tau$.
\begin{itemize}
\item[R1.] Let $\iglct\nvdash A$. By 
\Cref{lemma-gl-iglct} we have $\GL\nvdash A^\Box$ and then 
\Cref{Remark-reduction-GL-GLS}
 implies ${\GLV}\nvdash \tau(A^\Box)$, 
 in which $\tau$ is as used for the definition of $f(A)$. 
 Since $\tau$   is a $(.)^\Boxin$-substitution, 
 by \Cref{Lemma-Boxin-sub}
 we have ${\GLV}\nvdash \tau(A)^\Box$. Then 
\Cref{lemma-glsigma-iglct} implies that 
 $\iglct\nvdash \tau(A)$, or in other words
 $\iglct\nvdash f(A)$.
\item[R2.] Let $\PA^*\nvdash \sigmapas{f(A)}$ for some 
$\Sigma_1$-substitution $\sigma$.  By definition of $f(A)$, 
we must have $\iglct\nvdash A$, otherwise $f(A):=\top$,
which contradicts $\PA^*\nvdash \sigmapas{f(A)}$. Hence 
$f(A)=\tau(A)$ for  some propositional $(.)^\Boxin$-substitution 
$\tau$. Then we have $\PA\nvdash \sigmapas{\tau(A)}^\PA$ and 
by \Cref{Label-HA-HA*-Box-trans} we have
$\PA\nvdash \sigmapa{\tau(A)^\Box}$. Since 
${\sf iK4}+{\sf CP_a}$ is included in 
$\PLS(\PA,\PA)={\GLV}$ (\Cref{Solovey}), 
by \Cref{Lemma-Boxin-sub}
we have $\PA\nvdash \sigmapa{\tau(A^\Box)}$. 
This implies that $\PA\nvdash [\fbar(\sigma)]_{_{\sf PA}}(A^\Box)$ 
and again by \Cref{Label-HA-HA*-Box-trans} we have 
$\PA\nvdash ([\fbar(\sigma)]_{_{\sf PA^*}}(A))^\PA$.  Hence 
$\PA^*\nvdash [\fbar(\sigma)]_{_{\sf PA^*}}(A)$.
\qedhere
\end{itemize}
\end{proof}

\begin{lemma}\label{lemma-ihatglchatts-gls}
For every $A\in\lcalb$, 
we have 
$\ihatglchatts\vdash A$ iff 
$\GLS\vdash A^\Boxin$.
\end{lemma}
\begin{proof}
We use induction on the proof  $\ihatglchatts\vdash A$ 
 and show $\GLS\vdash A^\Boxin$. 
 All cases are similar to the one for item 3 above, except for 
 \begin{itemize}
 \item $A=\Box B\to B^\Box$:  
 since  ${\sf iK4}\vdash A^\Boxin\leftrightarrow (\Box B^\Box \to B^\Box)$, 
 we may deduce $\GLS\vdash A^\Boxin$.
\end{itemize} 
For the other way around, let $\ihatglchatts\nvdash A$. 
Then by $\Box \CP$  we have 
$A^\Boxin\leftrightarrow A$,  and then 
we may deduce $\ihatglchatts\nvdash A^\Boxin$.
By \Cref{Theorem-Kripke-completeness-ihatglchatts}, there exists some 
quasi-classical perfect Kripke model $\kcal$ and some boolean interpretation $I$
such that $\kcal,\alpha,I\not\models A^\Boxin$
and $\kcal$ is $A^\Boxin$-sound at $\alpha$. \Cref{Corol-truth-forcing} implies 
$\kcal,\alpha,I\not\models_c A^\Boxin$, which by soundness of $\GLS$ (restricted to sub-formulas of $A^\Boxin$)
 for  $A^\Boxin$-sound classical Kripke models, implies $\GLS\nvdash A^\Boxin$.
\end{proof}

\begin{theorem}\label{Reduction-(-to-sigma)-PA*-nat}
$\ihatglchatts =\PL(\PA^*,\nat)\leq
\PLS(\PA^*,\nat)=\ihatglchattssigma$.
\end{theorem}
\begin{proof} 
The arithmetical soundness of $\ihatglchatts$  
is straightforward and left to the reader. 
By  \Cref{Reduction-sigma-(PA*-to-PA)-nat}
we have  
$\PLS(\PA^*,\nat)=\ihatglchattssigma$.
We must show
$$\ac{\ihatglchatts}{\PA^*}{\nat}\leq_{f,\fbar[]} 
\acs{\ihatglchattssigma}{\PA^*}{\nat}.$$  
Given $A\in\lcalb$,   if $\ihatglchatts\vdash A$, 
define $f(A):=\top$. If $\ihatglchatts\nvdash A$, 
by \Cref{lemma-ihatglchatts-gls} we have 
$\GLS\nvdash A^\Boxin$, and hence 
by \Cref{Remark-reduction-GL-GLS}
there exists some propositional $(.)^\Boxin$-substitution
 $\tau$ such that 
${\GLSV}\nvdash \tau(A^\Boxin)$. 
Define $f(A):=\tau(A)$ and $\fbar (\sigma):=\sigma_{_{\sf PA^*}}\circ \tau$.
\begin{itemize}
\item[R1.] Let $\ihatglchatts\nvdash A$. By 
\Cref{lemma-ihatglchatts-gls} we have $\GLS\nvdash A^\Boxin$
 and then \Cref{Remark-reduction-GL-GLS}
 implies ${\GLSV}\nvdash \tau(A^\Boxin)$, 
 in which $\tau$ is as used for the definition of $f(A)$. 
 Since $\tau$   is a $(.)^\Boxin$-substitution, 
 by \Cref{Lemma-Boxin-sub}
 we have ${\GLSV}\nvdash \tau(A)^\Boxin$. Then 
\Cref{Lemma-ihatglchattssigma}  implies that 
 $\ihatglchattssigma\nvdash \tau(A)$, or in other words
 $\ihatglchattssigma\nvdash f(A)$.
\item[R2.] Let $\nat\not\models \sigmapas{f(A)}$ for some 
$\Sigma_1$-substitution $\sigma$.  
By definition of $f(A)$,  we must have 
$\ihatglchatts\nvdash A$, otherwise $f(A):=\top$,
which contradicts $\nat\not\models\sigmapas{f(A)}$. Hence 
$f(A)=\tau(A)$ for  some propositional $(.)^\Boxin$-substitution 
$\tau$. We have $\nat\not\models\sigmapas{\tau(A)}$ and 
by  \Cref{Lemma-Properties of Box translation 2}we have
$\nat\not\models\sigmapa{\tau(A)^\Boxin}$. Since 
${\sf iK4}+{\sf CP_a}$ is included in 
$\PLS(\PA,\nat)={\GLSV}$ (\Cref{Solovey}), 
by \Cref{Lemma-Boxin-sub}
we have $\nat\not\models\sigmapa{\tau(A^\Boxin)}$. 
This implies that $\nat\not\models[\fbar(\sigma)]_{_{\sf PA}}(A^\Boxin)$ 
and again by  \Cref{Lemma-Properties of Box translation 2} we have 
$\nat\\not\models [\fbar(\sigma)]_{_{\sf PA^*}}(A)$.  
\qedhere
\end{itemize}
\end{proof}

\begin{theorem}\label{Reduction-PA*-(PA*-to-nat)}
$\iglct=\PL(\PA^*,\PA^*)\leq\PL(\PA^*,\nat)=\ihatglchatts$.
\end{theorem}
\begin{proof} 
By \Cref{Reduction-(-to-sigma)-PA*-nat,Reduction-(-to-sigma)-PA*-PA*}
we have  $\PL(\PA^*,\PA^*)=\iglct$ and 
$\PL(\PA^*,\nat)=\ihatglchatts$.
We must show
$\ac{\iglct}{\PA^*}{\PA^*}\leq_{f,\fbar[]} \ac{\ihatglchatts}{\PA^*}{\nat}$.  
Given $A\in\lcalb$, define $f(A)=\Box A$ and $\fbar$ as identity function.  
\begin{itemize}
\item[R1.] Let $\ihatglchatts\vdash \Box A$. By soundness
of $\ihatglchatts=\PL(\PA^*,\nat)$, for every substitution $\sigma$
we have $\nat\models \sigmapas{\Box A}$ and hence 
$\PA^*\vdash \sigmapas{A}$. Then by arithmetical 
completeness of $\iglct=\PL(\PA^*,\PA^*)$, we have $\iglct\vdash A$.
\item[R2.] Let $\nat\not\models \sigmapas{\Box A}$. 
Then $\PA^*\nvdash \sigmapas{A}$, as desired. \qedhere
\end{itemize}
\end{proof}

\begin{theorem}\label{Reduction-PA*-(PA-to-HA)}
$\ihatglchatt=\PL(\PA^*,\PA)\leq\PL(\PA^*,\HA)=\iglchatthat$.
\end{theorem}
\begin{proof}
We already have 
 $\iglchatthat=\PL (\PA^*,\HA)$ and 
  $\PL(\PA^*,\PA)=\ihatglchatt$  by 
\Cref{Reduction-(-to-sigma)-PA*-PA,Reduction-(-to-sigma)-PA*-HA}.  It is enough here to show that 
$\ac{\ihatglchatt}{\PA^*}{\PA}\leq_{f,\fbar[]}\ac{\iglchatthat}{\PA^*}{\HA}$. Given $A\in\lcalb$,
let $f(A):=(A)^\negout$ and $\fbar$ as identity function. 
\begin{itemize}
\item[R1.] If  $\iglchatthat\vdash A^\negout$ then  
$\ihatglchatt\vdash A^\negout$, and  since 
we have $\PEM$ in $\ihatglchattsigma$, we may conclude 
$\ihatglchattsigma\vdash A$.
\item[R2.] If $\HA\nvdash\sigmapas{A^\negout}$,  
for a  substitution $\sigma$, 
then by  \Cref{Lemma-neg-translate-1st-order} we have $\HA\nvdash(\sigmapas{A})^\neg$. Hence by \Cref{Lemma-HA-PA-neg-translation} we have $\PA\nvdash \sigmapas{A}$. \qedhere
\end{itemize} 
\end{proof}

\section{Conclusion}
From Diagram \ref{Diagram-full}, it turns out that the truth 
$\Sigma_1$-provability logic of $\HA$,
 is the hardest provability logic  among 
 the provability logics in  Table \ref{Table-Theories}.
  Closer inspection in the reductions provided in previous sections, 
  reveals that all propositional reductions, i.e.~the functions $f$,
are computable. Hence by decidability of $\PLS(\HA,\nat)$ 
(\Cref{Corollary-Decidability-ihatglchats}) and  
\Cref{Theorem-decidability-Reduction}, 
we have the decidability of all  provability logics in  Table 
\ref{Table-Theories}:
\begin{corollary}
All provability logics in the Table \ref{Table-Theories}, are decidable.
\end{corollary}

So far, we have seen  many reductions of provability logics. The reductions, helped out
 to prove new arithmetical completeness results, have a more general view of all 
 provability logics and intuitively say which provability logic is 
 \emph{harder}. The reader  may 
wonder what other
 reductions hold, beyond the transitive closure of the Diagram 
\ref{Diagram-full}.  
However it seems  more likely that no other reductions holds,  at the moment we can not say anything more than that. 
This question calls for more work.

\begin{conjecture}
We conjecture that the following characterizations and reductions holds:
\begin{enumerate}
\item ${\sf iH} =\PL(\HA,\HA)\leq \PLS(\HA,\HA)=\lles$. 
\item ${\sf iH} =\PL(\HA,\HA)\leq \PL(\HA,\nat)={\sf iH\underline{SP}}$.
\item ${\sf iH\underline{P}} =\PL(\HA,\PA)\leq \PLS(\HA,\PA)={\sf iH_\sigma\underline{P}}$.
\item ${\sf iH\underline{SP}}=\PL(\HA,\nat)\leq \PLS(\HA,\nat)=\ihatHSsigma$.
\item ${\sf iH}^{*}=\PL(\HA^*,\HA^*)\leq \PLS(\HA^*,\HA^*)=\llessstar$. 
\end{enumerate}
Moreover, all reductions are computable and hence all provability 
logics are conjectured to be decidable.  In which 
\begin{itemize}
\item ${\sf iH}$ is as defined in \cite{IemhoffT}, 
\item ${\sf iH}^{*}$  as defined in \cite{Sigma.Prov.HA*}, 
\item  ${\sf iH\underline{P}} $  is  ${\sf iH}$ plus $\underline{\sf P}$, 
\item  ${\sf iH\underline{SP}} $  is  ${\sf iH}$ plus $\underline{\sf S}$ and $\underline{\sf P}$, 
\end{itemize}
\end{conjecture}


\providecommand{\bysame}{\leavevmode\hbox to3em{\hrulefill}\thinspace}
\providecommand{\MR}{\relax\ifhmode\unskip\space\fi MR }
\providecommand{\MRhref}[2]{%
  \href{http://www.ams.org/mathscinet-getitem?mr=#1}{#2}
}
\providecommand{\href}[2]{#2}

%
\newpage
\begin{appendices}
\renewcommand{\figurename}{Table}
\begin{figure}[h]
\bgroup
\def\arraystretch{2}
\begin{center}
\begin{tabular}{|c|c||c|c|}
    \hline
    \textbf{ Name(s)} 
    & \textbf{Axiom Scheme}
    & \textbf{Name(s)} 
    & \textbf{Axiom Scheme}
\\ \hline\hline 
    \underline{\sfk} 
    & $\Box(A\to B)\to (\Box A\to \Box B)$ 
    & \underline{\sf 4}
    & $\Box A\to\Box\Box A$
\\ \hline
    \underline{\sf L\"ob}, \underline{\sf L}
    & $\Box(\Box A\to A)\to \Box A$ 
    & \underline{\sf CP}, \underline{\sf C}
    & $  A\to \Box A$
\\ \hline
    \underline{\sf S}
    & $\Box A \to  A$ 
    &  $\underline{\sf CP_a}$, $\underline{\sf C_a}$
    & \bgroup \def\arraystretch{1}\begin{tabular}{c}
    $p\to \Box p$ \\
    for atomic variable $p$
        \end{tabular} \egroup 
\\ \hline
	$\underline{\sf S^*}$
	& $\Box A\to A^\Box$
	& \underline{\sf PEM}, \underline{\sf P}
	& $A\vee\neg A$
\\ \hline
	$\underline{\sf Le}$
	& $\Box (A\vee B) \to \Box(\Boxdot A\vee\Boxdot B) $
	& $\underline{\sf Le^+}$
	& $\Box A\to \Box A^l$
\\ \hline
	$\underline{\sf TP}$, $\underline{\sf T}$
	& $\Box (A\to B)\to (A\vee (A\to B))$
	&  \underline{\sf i}
	&  		All theorems of $\IPC_\Box$
\\ \hline
	
	& 
	& \underline{\sf V}
	& $A\lr A^-$
\\ \hline
\multicolumn{4}{|c|}{For an axiom scheme \underline{\sf A}, let $\overline{\sf A}$ indicates $\Box\, \underline{\sf A}$ and ${\sf A}$ indicates  $\overline{\sf A}\wedge \underline{\sf A}$}
\\ \hline
    \end{tabular}
\caption[Table  of axiom schemas]{List  of axiom schemas\label{Table-Axioms}}
\end{center}
\egroup
\end{figure}

\begin{figure}
\bgroup\def\arraystretch{2}
\begin{center}
   \begin{tabular}{|c|c|c|c|}
    \hline
    \textbf{Theory}
    &\textbf{Axioms}
    &\textbf{Provability Logic(s)}
    &\textbf{Reference}
    \\ \hline\hline
    \ikfour
    & {\sf i},{\sf K},{\sf 4}
    &
    &
    \\ \hline
    $\igl$
    & {\sf iK4},{\sf L}
    &
    &
    \\ \hline
    \GL
    & $\igl$,{\sf P}
    & $\PL(\PA,\PA)$
    & \cite{Solovay}
     \\ \hline
    $\GLV$
    & $\GL$,${\sf CP_a}$
    & $\PLS(\PA,\PA)$
    & \cite{Visser82}
     \\ \hline
	\GLS
    & $\GL$,\underline{\sf S}
    & $\PL(\PA,\nat)$
    & \cite{Solovay}
     \\ \hline    
	$\GLSV$
    & $\GLV$,\underline{\sf S}
    & $\PLS(\PA,\nat)$
    & \cite{Visser82}
     \\ \hline       
	$\iglct$
    & $\igl$,{\sf C},{\sf T}
    & \bgroup\def\arraystretch{1.2}\begin{tabular}{c}
		    $\PL(\PA^*,\PA^*)$\\
		    $\PLS(\PA^*,\PA^*)$    
    \end{tabular}\egroup
    & \cite{Visser82} 
     \\ \hline
    $\lles$
    & $\igl$,{\sf V}, ${\sf Le}^+$
    & $\PLS(\HA,\HA)$
    & \cite{Sigma.Prov.HA,Jetze-Visser}
     \\ \hline
    $\llessstar$
    & $\{A: \lles\vdash A^\Box\}$
    & $\PLS(\HA^*,\HA^*)$
    & \cite{Sigma.Prov.HA*}
     \\ \hline   
    $\ihatHsigma$
    & $\lles$,\underline{\sf P}
    & $\PLS(\HA,\PA)$
    & \Cref{Theorem-PA-relative}
    \\ \hline
    $\ihatHSsigma$ 
    & $\lles$,\underline{\sf S},\underline{\sf P}
    & $\PLS(\HA,\nat)$
    & \Cref{Theorem-truth-provability}
    \\ \hline 
    $\ihatHSsigmastar$
    & $\{ A: \ihatHSsigma\vdash A^\Boxin\}$
    & $\PLS(\HA^*,\nat)$
    & \Cref{Reduction-Sigma-HA*-nat-to-HA-nat}
    \\ \hline
    $\ihatHsigmastar$
    & $\{ A: \ihatHsigma\vdash A^\Boxin\}$
    & $\PLS(\HA^*,\PA)$
    & \Cref{Reduction-Sigma-HA*-PA-to-HA-PA}
    \\ \hline
    $\llesstar$
    & $\{ A: \lles\vdash A^\Boxin\}$
    & $\PLS(\HA^*,\HA)$
    & \Cref{Reduction-Sigma-HA*-HA-to-HA-HA}
    \\ \hline
	$\iglphatsigma$    
	& $\igl$,$\overline{\sf P}$,${\sf C_a}$
	& $\PLS(\PA,\HA)$
    & \Cref{Reduction-Sigma-(PA-to-HA)-HA}
    \\ \hline
	$\iglphat$    
	& $\igl$,$\overline{\sf P}$
	& $\PL(\PA,\HA)$
    & \Cref{Reduction-PA-HA-to-Sigma-PA-HA}
    \\ \hline    
	$\ihatglchattsigma$
	& $\igl$,$\overline{\sf C}$,${\sf T}$,\underline{P},${\sf C_a}$
	& $\PLS(\PA^*,\PA)$ 
    & \Cref{Reduction-sigma-(PA*-to-PA)-PA}
    \\ \hline
	$\iglchatthatsigma$
	& $\igl$,$\overline{\sf C}$,$\overline{\sf T}$,${\sf C_a}$
	& $\PLS(\PA^*,\HA)$ 
    & \Cref{Reduction-sigma-(PA*-to-PA)-HA}
    \\ \hline
	$\ihatglchattssigma$
	&  $\igl$,$\overline{\sf C}$,${\sf T}$,$\underline{\sf S^*}$,\underline{P},${\sf C_a}$
	& $\PLS(\PA^*,\nat)$ 
    & \Cref{Reduction-sigma-(PA*-to-PA)-nat}
    \\ \hline
	$\ihatglchatt$
	& $\igl$,$\overline{\sf C}$,${\sf T}$,\underline{P}
	& $\PL(\PA^*,\PA)$ 
    & \Cref{Reduction-(-to-sigma)-PA*-PA}
    \\ \hline
	$\iglchatthat$
	& $\igl$,$\overline{\sf C}$,$\overline{\sf T}$
	& $\PL(\PA^*,\HA)$ 
    & \Cref{Reduction-(-to-sigma)-PA*-HA}
    \\ \hline
	$\ihatglchatts$
	&  $\igl$,$\overline{\sf C}$,${\sf T}$,$\underline{\sf S^*}$,\underline{P}
	& $\PL(\PA^*,\nat)$ 
    & \Cref{Reduction-(-to-sigma)-PA*-nat}
    \\ \hline    
    \end{tabular} 
    \caption[Table of all provability logics ]{List of all provability logics \label{Table-Theories}}
\end{center}
\egroup
\end{figure}
\renewcommand{\figurename}{Diagram}

\begin{sidewaysfigure}[ht]
\[ \begin{tikzcd}[column sep=5em, row sep=4em] 
 \tcboxmath{\PL(\PA,\HA)} 
\arrow[ddd, bend right=50, "\ref{Reduction-PA-HA-to-Sigma-PA-HA}", "\tau" ']
& 
\tcboxmath{\PL(\PA,\PA)}
\arrow[ddd,bend left=50, "\tau" ' near start , "\text{\ref{Reduction-Sigma-PA-(PA-nat)}}.2" near start]
\arrow[rr, "\Box (.)", " \ref{Reduction-PA-(PA-to-nat)}" ']
& &
\tcboxmath{ \PL(\PA,\nat) }
\arrow[ddd, bend left=50, "\tau" ', "\text{\ref{Reduction-Sigma-PA-(PA-nat)}}.1"]
\\
\tcboxmath{\PL(\PA^*,\HA)}
\arrow[d, "\ref{Reduction-(-to-sigma)-PA*-HA}", "\tau" ']
&
\tcboxmath{\PL(\PA^*,\PA)}
\arrow[d,"\ref{Reduction-(-to-sigma)-PA*-PA}", "\tau" ']
&
\tcboxmath{\PL(\PA^*,\PA^*)}
\arrow[r,"\ref{Reduction-PA*-(PA*-to-nat)}"',"\Box(.)"]
\arrow[d,"\ref{Reduction-(-to-sigma)-PA*-PA*}","\tau"']
&
\tcboxmath{\PL(\PA^*,\nat)}
\arrow[d,"\ref{Reduction-(-to-sigma)-PA*-nat}","\tau"']
\\
\tcboxmath{\PLS(\PA^*,\HA)}
\arrow[d, "\ref{Reduction-sigma-(PA*-to-PA)-HA}", "(.)^\Boxin" ']
&
\tcboxmath{\PLS(\PA^*,\PA)}
\arrow[l, "\ref{Reduction-sigma-PA*-(PA-to-HA)}", "(.)^\negout" ']
\arrow[d, "\ref{Reduction-sigma-(PA*-to-PA)-PA}", "(.)^\Boxin" ']
&
\tcboxmath{\PLS(\PA^*,\PA^*)}
\arrow[l, "\ref{Reduction-Sigma-PA*-(PA*-to-PA)}", "(.)^\Boxout" ']
\arrow[r,"\ref{Reduction-Sigma-PA*-(PA*-to-nat)}" ', " \Box(.)" ]
&
\tcboxmath{\PLS(\PA^*,\nat)}
\arrow[d,"\ref{Reduction-sigma-(PA*-to-PA)-nat}", "(.)^\Boxin" ']
\\
\tcboxmath{ \PLS(\PA,\HA) }
\arrow[dr, "\ref{Reduction-Sigma-(PA-to-HA)-HA} "  near end, "(.)^\dagger" ' near end]
& 
 \tcboxmath{\PLS(\PA,\PA) }
 \arrow[rr,   "\Box (.)",  " \ref{Reduction-Sigma-PA-(PA-to-nat)}" ' ]
 \arrow[l,"(.)^\negout" ' , "\ref{Reduction-Sigma-PA-(PA-to-HA)}"] 
\arrow[dl, "\ref{Reduction-Sigma-(PA-to-HA)-PA}" ' near end, "(.)^\dagger"  near end]
&  &
 \tcboxmath{\PLS(\PA,\nat) }
\arrow[d, "\text{\ref{Reduction-Sigma-(PA-to-HA)-nat}}" , "(.)^\dagger" ']
\\
\tcboxmath{\PLS(\HA,\PA) }
\arrow[r , "(.)^\negout", "\text{\ref{Reduction-Sigma-HA-(PA-to-HA)}}" ']  
&\tcboxmath{\PLS(\HA,\HA) }
\arrow[rr , "\Box(.)", "\text{\ref{Reduction-Sigma-HA-(HA-to-nat)}}" '] 
&  & \tcbhighmath{\PLS(\HA,\nat)}
\\
 \tcboxmath{\PLS(\HA^*,\PA) }
\arrow[r, "(.)^\negout", "\text{\ref{Reduction-Sigma-HA*-PA-to-HA*-HA}}" '] 
\arrow[u,"(.)^\Boxin", "\text{\ref{Reduction-Sigma-HA*-PA-to-HA-PA}}" ']
&  \tcboxmath{\PLS(\HA^*,\HA) }
\arrow[u,"(.)^\Boxin", "\text{\ref{Reduction-Sigma-HA*-HA-to-HA-HA}}" ']
 & 
  \tcboxmath{\PLS(\HA^*,\HA^*) }
 \arrow[r,"\Box(.)", "\text{\ref{Reduction-Sigma-HA*-(HA*-to-nat)}}" '] 
 \arrow[l,"(.)^\Boxout"', "\text{\ref{Reduction-Sigma-HA*-(HA*-to-HA)}}"]
 &
   \tcboxmath{\PLS(\HA^*,\nat) }
\arrow[u,"(.)^\Boxin", "\text{\ref{Reduction-Sigma-HA*-nat-to-HA-nat}}" ']
\end{tikzcd}
\]
\caption[Diagram of all reductions of  provability logics]{\label{Diagram-full} Reductions of all provability logics.
Arrows  indicate a reduction of the completeness of the left hand side to the right one. The propositional reduction is shown over the arrow line and the theorem number proving this, is shown under arrow line. }
\end{sidewaysfigure}
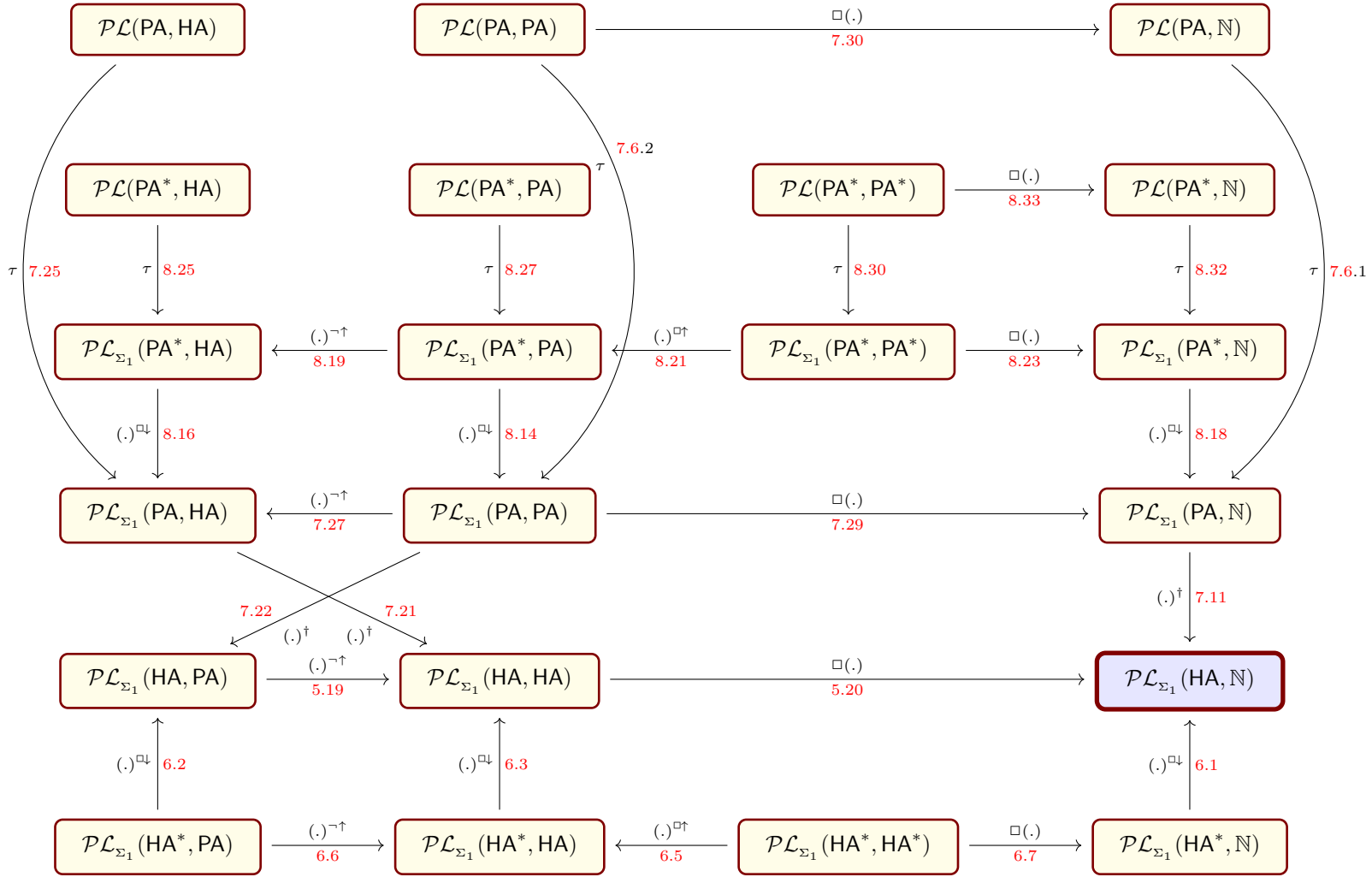
\end{appendices}

\begin{thebibliography}{VvBdJRdL95}

\bibitem[AB04]{ArtBekProv}
S.~Artemov and L.~Beklemishev, \emph{Provability logic}, in Handbook of
  Philosophical Logic (D.~Gabbay and F.~Guenthner, eds.), vol.~13, Springer,
  2nd ed., 2004, pp.~189--360.

\bibitem[AM14]{ArMo14}
M.~Ardeshir and M.~Mojtahedi, \emph{The de {J}ongh property for {B}asic
  {A}rithmetic}, Archive for Mathematical Logic \textbf{53} (2014), no.~7-8,
  881--895.

\bibitem[AM15]{reduction}
\bysame, \emph{{R}eduction of provability logics to {$\Sigma_1$}-provability
  logics}, Logic Journal of IGPL \textbf{23} (2015), no.~5, 842--847.

\bibitem[AM18]{Sigma.Prov.HA}
\bysame, \emph{The {$\Sigma_1$}-{P}rovability {L}ogic of {${\sf HA}$}}, Annals
  of Pure and Applied Logic \textbf{169} (2018), no.~10, 997--1043.

\bibitem[AM19]{Sigma.Prov.HA*}
\bysame, \emph{The {$\Sigma_1$}-{P}rovability {L}ogic of {${\sf HA}^*$}},
  Jornal of Symbolic Logic \textbf{84} (2019), no.~3, 118--1135.

\bibitem[Ber90]{Berarducci}
A.~Berarducci, \emph{The {I}nterpretability {L}ogic of {P}eano {A}rithmetic},
  Journal of Symbolic Logic \textbf{55} (1990), no.~3, 1059--1089.

\bibitem[BV06]{VisBek}
L.~Beklemishev and A.~Visser, \emph{Problems in the logic of provability},
  Mathematical problems from applied logic. {I}, Int. Math. Ser. (N. Y.),
  vol.~4, Springer, New York, 2006, pp.~77--136.

\bibitem[dJ70]{dejongh}
D.~de~Jongh, \emph{The maximality of the intuitionistic predicate calculus with
  respect to heyting's arithmetic}, Journal of Symbolic Logic \textbf{36}
  (1970), 606.

\bibitem[dJVV11]{vi1}
D.~de~Jongh, R.~Verbrugge, and A.~Visser, \emph{Intermediate logics and the de
  jongh property}, Archive for Mathematical Logic \textbf{50} (2011), 197--213.

\bibitem[Fef60]{Feferman}
S.~Feferman, \emph{Arithmetization of metamathematics in a general setting},
  Fundamenta Mathematicae \textbf{49} (1960), no.~1, 35--92 (eng).

\bibitem[Fri75]{Friedman75}
H.~Friedman, \emph{The disjunction property implies the numerical existence
  property}, Proc. Nat. Acad. Sci. U.S.A. \textbf{72} (1975), no.~8,
  2877--2878. \MR{0379141 (52 \#47)}

\bibitem[G{\"o}d31]{Godel}
K.~G{\"o}del, \emph{\"{U}ber formal unentscheidbare {S}\"atze der {P}rincipia
  {M}athematica und verwandter {S}ysteme {I}}, Monatsh. Math. Phys. \textbf{38}
  (1931), no.~1, 173--198.

\bibitem[G{\"o}d33]{Godel33}
\bysame, \emph{Eine interpretation des intuitionistischen aussagenkalkuls},
  Ergebnisse eines mathematischen Kolloquiums \textbf{4} (1933), 39--40,
  English translation in: S. Feferman etal., editors, Kurt G{\"o}del Collected
  Works, Vol. 1, pages 301-303. Oxford University Press, 1995.

\bibitem[HP93]{HP}
P~Hajek and P.~Pudlak, \emph{Metamathematics of first-order arithmetic},
  Springer-Verlag, 1993.

\bibitem[Iem01]{IemhoffT}
R.~Iemhoff, \emph{Provability logic and admissible rules}, Ph.D. thesis,
  University of Amsterdam, 2001.

\bibitem[{L}\"55]{Lob}
{M}. {L}\"{o}b, \emph{{S}olution of a {P}roblem of {L}eon {H}enkin}, Journal of
  Symbolic Logic \textbf{20} (1955), no.~2, 115--118.

\bibitem[Lei75]{Leivant-Thesis}
D.~Leivant, \emph{Absoluteness in intuitionistic logic}, Ph.D. thesis,
  University of Amsterdam, 1975.

\bibitem[Myh73]{Myhill}
J.~Myhill, \emph{A note on indicator-functions}, Proceedings of the AMS
  \textbf{39} (1973), 181--183.

\bibitem[Smo73]{Smorynski-Troelstra}
C.~A. Smory{\'n}ski, \emph{Applications of {K}ripke models}, Metamathematical
  investigation of intuitionistic arithmetic and analysis, Springer, Berlin,
  1973, pp.~324--391. Lecture Notes in Mathematics, Vol. 344.

\bibitem[Smo85]{Smorynski-Book}
C.~Smory{\'n}ski, \emph{Self-reference and modal logic}, Universitext,
  Springer-Verlag, New York, 1985.

\bibitem[Sol76]{Solovay}
R.~M. Solovay, \emph{Provability interpretations of modal logic}, Israel J.
  Math. \textbf{25} (1976), no.~3-4, 287--304.

\bibitem[TvD88]{TD}
A.~S. Troelstra and D.~van Dalen, \emph{Constructivism in mathematics. {V}ol.
  {I}}, Studies in Logic and the Foundations of Mathematics, vol. 121,
  North-Holland Publishing Co., Amsterdam, 1988, An introduction.

\bibitem[Vis81]{VisserThes}
A.~Visser, \emph{Aspects of diagonalization and provability}, Ph.D. thesis,
  Utrecht University, 1981.

\bibitem[Vis82]{Visser82}
\bysame, \emph{On the completeness principle: a study of provability in
  {H}eyting's arithmetic and extensions}, Ann. Math. Logic \textbf{22} (1982),
  no.~3, 263--295.

\bibitem[Vis98]{VisserInterpretability}
\bysame, \emph{An overview of interpretability logic}, Advances in modal logic,
  {V}ol.\ 1 ({B}erlin, 1996), CSLI Lecture Notes, vol.~87, CSLI Publ.,
  Stanford, CA, 1998, pp.~307--359.

\bibitem[Vis02]{Visser02}
\bysame, \emph{Substitutions of {$\Sigma_1^0$} sentences: explorations between
  intuitionistic propositional logic and intuitionistic arithmetic}, Ann. Pure
  Appl. Logic \textbf{114} (2002), no.~1-3, 227--271, Commemorative Symposium
  Dedicated to Anne S. Troelstra (Noordwijkerhout, 1999).

\bibitem[VvBdJRdL95]{Visser-Benthem-NNIL}
A.~Visser, J.~van Benthem, D.~de~Jongh, and G.~R. R.~de Lavalette, \emph{{${\rm
  NNIL}$}, a study in intuitionistic propositional logic}, Modal logic and
  process algebra ({A}msterdam, 1994), CSLI Lecture Notes, vol.~53, CSLI Publ.,
  Stanford, CA, 1995, pp.~289--326.

\bibitem[VZ19]{Jetze-Visser}
A.~Visser and J.~Zoethout, \emph{Provability logic and the completeness
  principle}, Annals of Pure and Applied Logic \textbf{170} (2019), no.~6,
  718--753.

\end{thebibliography}
\end{document}